\documentclass[12pt,british]{article}
\usepackage{babel}
\usepackage{ucs} 
\usepackage[utf8x]{inputenc} 
\usepackage[T1]{fontenc} 

\usepackage{amsmath,amsthm,amssymb,textcase}
\usepackage{enumerate}

\usepackage[full]{textcomp}

\usepackage[noBBpl]{mathpazo}

\usepackage[
bookmarks=true,
breaklinks=true,
bookmarksnumbered = true,
colorlinks= true,
urlcolor= green,
anchorcolor = yellow,
citecolor=blue,
]{hyperref}                   

\usepackage{tikz}
\usepackage{tikz-cd}
\usetikzlibrary{arrows,decorations.markings}
\tikzset{commutative diagrams/.cd,arrow style=tikz,diagrams={>=stealth'}}
\tikzset{middlearrow/.style={decoration={markings,mark=at position #1 with {\arrow[very thick]{stealth' reversed}}},postaction={decorate}}}
\tikzset{middlearrowrev/.style={decoration={markings,mark=at position #1 with {\arrow[very thick]{stealth'}}},postaction={decorate}}}

\makeatletter
\renewcommand{\theenumi}{\alph{enumi}}

\renewcommand{\p@enumii}{}

\def\@enum@{\list{\csname label\@enumctr\endcsname}%
          {\usecounter{\@enumctr}\def\makelabel##1{
\normalfont\ignorespaces\emph{{##1}~}}
\setlength{\labelsep}{3pt}
\setlength{\parsep}{0pt}
\setlength{\itemsep}{0pt}
\setlength{\leftmargin}{0pt}
\setlength{\labelwidth}{0pt}
\setlength{\listparindent}{\parindent}
\setlength{\itemsep}{0pt}
\setlength{\itemindent}{0pt}
\topsep=3pt plus 1pt minus 1 pt}}

\makeatother

\renewcommand{\epsilon}{\ensuremath{\varepsilon}}
\renewcommand{\phi}{\ensuremath{\varphi}}

\renewcommand{\to}{\ensuremath{\longrightarrow}}
\renewcommand{\mapsto}{\ensuremath{\longmapsto}}

\newcommand{\textsu}[1]{\textsuperscript{#1}}
\newcommand{\R}{\ensuremath{\mathbb R}}
\newcommand{\N}{\ensuremath{\mathbb N}}
\newcommand{\Z}{\ensuremath{\mathbb Z}}
\newcommand{\dt}{\ensuremath{\mathbb D}^{2}}
\newcommand{\St}[1][2]{\ensuremath{\mathbb S}^{#1}}
\newcommand{\FF}{\ensuremath{\mathbb F}}
\newcommand{\F}[1][n]{\ensuremath{\FF_{{#1}}}}
\newcommand{\rp}{\ensuremath{\mathbb{R}P^2}}
\newcommand{\ft}[1][n]{\ensuremath{\Delta_{#1}^{2}}}

\newcommand{\dd}{\ensuremath{\mathbb D}}

\makeatletter
\def\@map#1#2[#3]{\mbox{$#1 \colon\thinspace #2 \to #3$}}
\def\map#1#2{\@ifnextchar [{\@map{#1}{#2}}{\@map{#1}{#2}[#2]}}
\makeatother

\newcommand{\altmap}[4][\lhra]{\mbox{$#2 \colon\thinspace #3 #1 #4$}}

\DeclareRobustCommand*{\up}[1]{\textsu{#1}}
\newcommand{\id}{\ensuremath{\operatorname{\text{Id}}}}

\newcommand{\br}[1]{\ensuremath{\widehat{\tau}_{n_{0}}}}
\newcommand{\dl}[1]{\ensuremath{\xi_{n_{0}}}}
\renewcommand{\ker}[1]{\ensuremath{\operatorname{\text{Ker}}\left({#1}\right)}}
\newcommand{\im}[1]{\ensuremath{\operatorname{\text{Im}}\left({#1}\right)}}

\newcommand{\brrp}{\ensuremath{\zeta_{2}}}

\newcommand{\brak}[1]{\ensuremath{\left\{ #1 \right\}}}
\newcommand{\ang}[1]{\ensuremath{\left\langle #1\right\rangle}}

\newcommand{\setangl}[2]{\ensuremath{\ang{\left. #1 \,\right\rvert \, #2}}}
\newcommand{\lhra}{\lhook\joinrel\longrightarrow}
\newcommand{\orbconf}[1][n]{\ensuremath{F_{#1}^{\,\ang{\tau}}(C)}}

\newcommand{\setl}[2]{\ensuremath{\brak{\left. #1 \,\right\rvert \, #2}}}

\newtheoremstyle{theoremm}{}{}{\itshape}{}{\scshape}{.}{ }{}

\theoremstyle{theoremm}
\newtheorem{thm}{Theorem}
\newtheorem{lem}[thm]{Lemma}
\newtheorem{prop}[thm]{Proposition}
\newtheorem{cor}[thm]{Corollary}

\newtheoremstyle{remarkk}{}{}{}{}{\scshape}{.}{ }{}

\theoremstyle{remarkk}

\newtheorem{rem}[thm]{Remark}
\newtheorem{rems}[thm]{Remarks}

\newcommand{\req}[1]{equation~(\protect\ref{eq:#1})}
\newcommand{\reqref}[1]{(\protect\ref{eq:#1})}
\newcommand{\reth}[1]{Theorem~\protect\ref{th:#1}}
\newcommand{\relem}[1]{Lemma~\protect\ref{lem:#1}}
\newcommand{\reco}[1]{Corollary~\protect\ref{cor:#1}}
\newcommand{\repr}[1]{Proposition~\protect\ref{prop:#1}}
\newcommand{\rerem}[1]{Remark~\protect\ref{rem:#1}}
\newcommand{\rerems}[1]{Remarks~\protect\ref{rems:#1}}

\newcommand{\resec}[1]{Section~\protect\ref{sec:#1}}

\allowdisplaybreaks

\begin{document}


\title{The homotopy fibre of the inclusion
$F_n(M) \lhra \prod_{1}^{n}\, M$ for $M$ either $\St$ or
$\rp$ and orbit configuration spaces}  

\author{DACIBERG~LIMA~GON\c{C}ALVES\\
Departamento de Matem\'atica - IME- Universidade de S\~ao Paulo\\
Rua do Mat\~ao, 1010, CEP~05508-090 - S\~ao Paulo - SP - Brazil.\\
e-mail:~\url{dlgoncal@ime.usp.br}\vspace*{4mm}\\
JOHN~GUASCHI\\
Normandie Univ., UNICAEN, CNRS,\\
Laboratoire de Math\'ematiques Nicolas Oresme UMR CNRS~\textup{6139},\\
14000 Caen, France.\\
e-mail:~\url{john.guaschi@unicaen.fr}}

\date{27\textsu{th} September 2017}

\begingroup
\renewcommand{\thefootnote}{}
\endgroup 

\maketitle

\begin{abstract}
\noindent \emph{Let $n\geq 1$, and let $\map{\iota_{n}}{F_{n}(M)}[\prod_{1}^{n}\, M]$ be the natural inclusion of the $n$\textsu{th} configuration space of $M$ in the $n$-fold Cartesian product of $M$ with itself. In this paper, we study the map $\iota_{n}$, its homotopy fibre $I_{n}$, and the induced homomorphisms $(\iota_{n})_{\#k}$ on the $k$\textsu{th} homotopy groups of $F_{n}(M)$ and $\prod_{1}^{n}\, M$ for $k\geq 1$ in the cases where $M$ is the $2$-sphere $\St$ or the real projective plane $\rp$. If $k\geq 2$, we show that the homomorphism $(\iota_{n})_{\#k}$ is injective and diagonal, with the exception of the case $n=k=2$ and $M=\St$, where it is anti-diagonal. We then show that $I_{n}$ has the homotopy type of $K(R_{n-1},1) \times \Omega(\prod_{1}^{n-1} \St)$, where $R_{n-1}$ is the $(n-1)$\textsu{th} Artin pure braid group if $M=\St$, and is the fundamental group $G_{n-1}$ of the $(n-1)$\textsu{th} orbit configuration space of the open cylinder $\St\setminus \brak{\widetilde{z}_{0}, -\widetilde{z}_{0}}$ with respect to the action of the antipodal map of $\St$ if $M=\rp$, where $\widetilde{z}_{0}\in \St$. This enables us to describe the long exact sequence in homotopy of the homotopy fibration $I_{n} \to  F_n(M) \stackrel{\iota_{n}}{\to}  \prod_{1}^{n}\, M$ in geometric terms, and notably the boundary homomorphism $\pi_{k+1}(\prod_{1}^{n}\, M)\to \pi_{k}(I_{n})$. 
From this, if $M=\rp$ and $n\geq 2$, we show that $\ker{(\iota_{n})_{\#1}}$ is isomorphic to the quotient of $G_{n-1}$ by its centre, as well as to an iterated semi-direct product of free groups with the subgroup of order $2$ generated by the centre of $P_{n}(\rp)$ that is reminiscent of the combing operation for the Artin pure braid groups, as well as decompositions obtained in~\cite{GGgold}.}
\end{abstract}

\section{Introduction}\label{sec:introduction}

Let $M$ be a connected surface, perhaps with boundary, and either compact, or with a finite number of points removed from the interior of the surface, and let $\prod_1^n\, M=M\times \cdots \times M$ denote the $n$-fold Cartesian product of $M$ with itself. 
The \emph{$n$\up{th} configuration space} of $M$ is defined by:
\begin{equation*}
F_n(M)=\setl{(x_{1},\ldots,x_{n})\in  \prod_{1}^{n}\, M}{\text{$x_{i}\neq x_{j}$ for all $1\leq i,j\leq n$, $i\neq j$}}.
\end{equation*}
It is well known that the fundamental group $\pi_1(F_n(M))$ of $F_{n}(M)$ is isomorphic to the \emph{pure braid group} $P_n(M)$ of $M$ on $n$ strings~\cite{FaN,FoN}, and 
if $M$ is the $2$-disc then $P_{n}(M)$ is the Artin pure braid group $P_{n}$.
Let $\map{\iota_{n}}{F_n(M)}[\prod_1^n\, M]$ denote the inclusion map, and for $k\geq 0$, let $\map{(\iota_{n})_{\#k}}{\pi_k(F_n(M))}[\pi_k\bigl( \prod_1^n\, M\bigr)]$ be the induced homomorphism of the corresponding homotopy groups. If no confusion is possible, we shall often just write $\iota$ in place of $\iota_{n}$, and $\iota_{\#}$ or $(\iota_{n})_{\#}$ if $k=1$. The homomorphism $\map{\iota_{\#}}{P_n(M)}[\pi_1\bigl( \prod_1^n\, M\bigr)]$ was studied by Birman in 1969~\cite{Bi}, and if $M$ is a compact surface without boundary different from the $2$-sphere $\St$ and the real projective plane $\rp$, Goldberg showed that $\ker{\iota_{\#}}$ is equal to the normal closure in $P_{n}(M)$ of the image of the homomorphism $\map{j_{\#}}{P_{n}}[P_{n}(M)]$ induced by the inclusion $\map{j}{\dt}[M]$ of a topological disc $\dt$ in $M$~\cite{Gol}. In~\cite{GGgold}, we extended this result to $\St$ and $\rp$, and in the case of $\rp$, we proved that $\iota_{\#}$ coincides with Abelianisation, so that $\ker{\iota_{\#}}$ is equal to the commutator subgroup $\Gamma_{2}(P_{n}(\rp))$ of $P_{n}(\rp)$.

The aim of this paper is to study further the maps $\iota_{n}$ and the induced homomorphisms $(\iota_{n})_{\#k}$ in more detail in the cases where $M=\St$ or $\rp$, and to determine the homotopy type of the homotopy fibre of $\iota_{n}$. Following~\cite[pages~91 and~108]{A}, recall that if $\map{f}{(X,x_{0})}[(Y,y_{0})]$ is a map of (pointed) topological spaces and $I$ denotes the unit interval then the \emph{mapping path} of $f$, defined by:
\begin{equation}\label{eq:defmappath}
E_{f}=\setl{(x,\lambda)\in X\times Y^I}{f(x)=\lambda(0)}
\end{equation}
has the same homotopy type as $X$, the map $\map{r_{f}}{E_{f}}[X]$ given by $r_{f}(x,\lambda)=x$ being a homotopy equivalence, and the map $\map{p_{f}}{E_{f}}[Y]$ defined by $p_{f}(x,\lambda)=\lambda(1)$ is a fibration whose fibre is the \emph{homotopy fibre} $I_{f}$ of $f$ defined by:
\begin{equation}\label{eq:defhomfibre}
I_{f}=\setl{(x,\lambda)\in X\times Y^I}{\text{$f(x)=\lambda(0)$ and $\lambda(1)=y_{0}$}}.
\end{equation}
We will refer to the sequence of maps $I_{f} \to X \stackrel{f}{\to} Y$ as the \emph{homotopy fibration} of $f$, where the map $I_{f} \to X$ is the composition of the inclusion $I_{f}\to E_{f}$ by $r_{f}$. More details may be found in~\cite{A,Hat,Wh}. In what follows, we denote the homotopy fibre of $\iota_n$ by $I_{n}$. We shall determine the homotopy type of $I_{n}$ in the cases where $M$ is either $\St$ or $\rp$. This leads to a better understanding of the long exact sequence in homotopy of the fibration $p_{\iota_n}$, as well as an alternative interpretation of $P_n(M)$ 
in terms of exact sequences. Additional motivation for our study comes from the fact that the higher homotopy groups of $F_{n}(M)$ are known to be  isomorphic to those of $\St$ or $\St[3]$ (see \reth{homotopytypefns}), so such exact sequences involve these homotopy groups.

In the rest of this section, let $M$ be $\St$ or $\rp$. This paper is organised as follows. In \resec{mapiota}, we show in \relem{iotainject} that for all $n,k\geq 2$, the homomorphism $(\iota_{n})_{\#k}$ is injective, and in Propositions~\ref{prop:diagS2} and~\ref{prop:diagRP2}, we prove that this homomorphism is diagonal, with the exception of the case $M=\St$ and $n=k=2$, where it is anti-diagonal. The aim of \resec{sec3} is to prove \reth{prop5} stated below that describes the homotopy type of $I_{n}$. In \resec{generalities}, we define the framework and much of the notation that will be used in the rest of the paper. In the case of $\rp$, this description makes use of certain generalisations of configuration spaces, which we now define. Let $\widetilde{z}_{0}\in \St$, let $C=\St \setminus\brak{\widetilde{z}_{0},-\widetilde{z}_{0}}$ denote the open cylinder, and let $\map{\tau}{\St}$ denote the antipodal map defined by $\tau(x)=-x$ for all $x\in \St$ as well as its restriction to $C$. If $n\geq 1$, the \emph{$n$\up{th} orbit configuration space} of $C$ with respect to the group $\ang{\tau}$, which is a subspace of $F_{n}(C)$, is defined by:
\begin{equation*}
\orbconf=\setl{(x_{1},\ldots,x_{n})\in  \prod_{1}^{n}\, M}{\text{$x_{i}\notin \brak{x_{j}, \tau(x_{j})}$ for all $1\leq i,j\leq n$, $i\neq j$}}.
\end{equation*}
Orbit configuration spaces were defined and studied in~\cite{CX}, but some examples already appeared in~\cite{Fa,FvB}. Let $G_{n}=\pi_{1}\bigl(\orbconf\bigr)$ denote the fundamental group of $\orbconf$. In \relem{equifn}, we show that $\orbconf$ is a space of type $K(G_{n},1)$, and that $G_{n}$ may be written as an iterated semi-direct product of free groups similar to that of the Artin combing operation of $P_{n}$. We then have the following description of the homotopy fibre of $I_{\iota_{n}}$. 

\begin{thm}\label{th:prop5}
Let $n\geq 2$ and let $M=\St$ or $\rp$. Then the homotopy fibre $I_{n}$ of the map $\map{\iota_n}{F_n(M)}[\prod_1^n\, M]$ has the homotopy type of:
\begin{enumerate}[(a)]
\item\label{it:prop5a}
$F_{n-1}(\dt) \times \Omega\bigl(\prod_1^{n-1}\,\St\bigr)$ if $M=\St$, or equivalently of $K(P_{n-1}, 1) \times  \Omega\bigl(\prod_1^{n-1} \,\St\bigr)$, where $\Omega\bigl(\prod_1^{n-1} \,\St\bigr)$ denotes the loop space of $\prod_1^{n-1} \,\St$.

\item\label{it:prop5b} 
$\orbconf[n-1] \times \Omega\bigl(\prod_1^{n-1}\, \St\bigr)$ if $M=\rp$, or equivalently of $K(G_{n-1}, 1) \times \Omega\bigl(\prod_1^{n-1}\, \St\bigr)$.
\end{enumerate}
\end{thm}
\reth{prop5} will be proved in \resec{fibreiotaS2}. The basic idea of the proof is to `replace' $\iota_{n}$ by a map that is null homotopic and whose homotopy fibre has the same homotopy type as that of $\iota_{n}$. The fact that this map is null homotopic implies that its homotopy fibre may be written as a product that is the homotopy fibre $I_{c}$ of a constant map (see \rerems{homfibre}(\ref{it:homfibreb})). Another important tool is the relation between the homotopy fibres of fibre spaces and certain subspaces (see the Appendix, and \relem{crabb1} in particular). Taking the long exact sequence in homotopy of the homotopy fibration of $\iota_{n}$ and using the results of \resec{mapiota}, we obtain the following exact sequences, where for $k\geq 1$, $\map{\partial_{n,k}}{\pi_{k}(\prod_{1}^{n}\, M)}[\pi_{k-1}(I_{n})]$ denotes the boundary homomorphism on the level of $\pi_{k}$ associated to the homotopy fibration $I_{n}\to F_{n}(M) \stackrel{\iota_{n}}{\to} \prod_{1}^{n}\, M$.

\begin{cor}\label{cor:lesfibs2} 
Let $n\geq 2$, and let $M=\St$ or $\rp$.
\begin{enumerate}[(a)]
\item\label{it:lesfibs2a} Suppose that $k\geq 3$ (resp.\ $M=\St$ and $n=k=2$). Then we have the following split short exact sequence of Abelian groups:
\begin{equation}\label{eq:splitrp2}
1 \to  \pi_k(F_{n}(M)) \stackrel{(\iota_{n})_{\#k}}{\to} \pi_k\Bigl(\prod_1^{n}\, M\Bigr) \stackrel{\partial_{n,k}}{\to}  \pi_{k-1}\Bigl(\Omega\Bigl(\prod_1^{n-1} \, \St\Bigr)\Bigr)\to 1,
\end{equation}
where the homomorphism $(\iota_{n})_{\#k}$ is diagonal (resp.\ anti-diagonal). Up to isomorphism, this short exact sequence may also be written as:
\begin{equation*}
1 \to  \pi_k(M) \to  \prod_1^{n}\, \pi_k(M) \to  \prod_1^{n-1}\, \pi_k(M) \to 1.
\end{equation*}

\item\label{it:lesfibs2b} Suppose that $k=2$, and that $n\geq 3$ if $M=\St$. Then we have the following exact sequence:
\begin{equation}\label{eq:sespi1rp2}
1 \to  \pi_2\Bigl(\prod_1^{n}\, M\Bigr)  \stackrel{\partial_{n,2}}{\to} R_{n-1}\times\pi_1\Bigl(\Omega\Bigl(\,\prod_1^{n-1} \, \St\Bigr)\Bigr) \to P_n(M)\stackrel{(\iota_{n})_{\#1}}\to \pi_1\Bigl(\prod_1^{n}\, M\Bigr) \to 1,
\end{equation}
where $R_{n-1}=P_{n-1}$ if $M=\St$ and $R_{n-1}=G_{n-1}$ if $M=\rp$, which up to isomorphism, may also be written as:
\begin{gather*}
1 \to  \Z^{n} \to P_{n-1} \oplus  \Z^{n-1} \to P_{n}(\St) \to 1 \quad\text{if $M=\St$}\\
1 \to \Z^n \to G_{n-1}\times \Z^{n-1} \to P_n(\rp) \stackrel{(\iota_{n})_{\#1}}{\to} \Z_2^n \to 1 \quad\text{if $M=\rp$.}
\end{gather*}
In the case $M=\St$, the short exact sequence does not split.
\end{enumerate}
\end{cor}

The homomorphisms that appear in the exact sequences of \reco{lesfibs2} can be made explicit. In order to understand better the homotopy fibration associated to $\iota_{n}$ 
and these exact sequences, it is helpful to study the properties of $\orbconf$ and $G_{n}$, as well as the boundary homomorphism $\partial_{n,k}$, the case $k=2$ being the most complicated due to the appearance of $G_{n-1}$ in $\pi_{1}(I_{n})$. In \resec{presorb}, we analyse $G_{n}$. In \repr{presgn}, we give a presentation, from which we deduce in \repr{gncentre} that the centre of $G_{n}$ is infinite cyclic, generated by an element $\Theta_{n}$ similar in nature to the full-twist braid of $P_{n}(M)$.

If $M=\St$ (resp.\ $\rp$), let $n_{0}=3$ (resp.\ $n_{0}=2$). 
In \resec{boundhomo}, we determine completely the boundary homomorphism $\partial_{n,2}$, that we will denote simply by $\partial_{n}$. 
One of the principal difficulties here is to describe concrete homotopy equivalences between $I_{n}$ and the product spaces $I_{c}$ that appear in the statement of \reth{prop5} and which are homotopy fibres of constant maps. We first introduce geometric representatives of generators of $\pi_{1}(\Omega(M))$ in \resec{geomrepsS2}, and we use them both to describe certain elements of $\pi_{1}(I_{n_{0}})$ in \relem{boundaryS2}, and also to fix a basis $\mathcal{B}=(\widetilde{\lambda}_{\widetilde{x}_{0}}, \widetilde{\lambda}_{\widetilde{x}_{1}}, \ldots, \widetilde{\lambda}_{\widetilde{x}_{n-3}}, \widetilde{\lambda}_{\widetilde{z}_{0}}, \widetilde{\lambda}_{-\widetilde{z}_{0}})$ (resp.\ $\mathcal{B}= (\lambda_{x_{0}}, \lambda_{x_{1}}, \ldots, \lambda_{x_{n-2}}, \lambda_{z_{0}})$) of $\pi_2(\prod_1^n\, M)$ in \req{basispi2}. Here, $\widetilde{x}_{0}, \widetilde{x}_{1}, \ldots, \widetilde{x}_{n-2}, \widetilde{z}_{0}$ are certain basepoints of $\St$ that are defined in \resec{generalities}, and $x_{0}, x_{1}, \ldots, x_{n-2}, z_{0}$ are their projections in $\rp$ under the universal covering map from $\St$ to $\rp$. In \req{defdeltav}, for each element $\widetilde{\lambda}_{y}$ (resp.\ $\lambda_{y}$) of $\mathcal{B}$, we define the element $\widetilde{\delta}_{y}$ (resp.\ $\delta_{y}$) to be its image in $\pi_{1}(I_{n})$ under $\partial_{n}$. In \resec{notation}, in \relem{fundhomoequiv}, we construct explicit homotopy equivalences between the intermediate homotopy fibres that were introduced in \resec{generalities} whose composition yields a homotopy equivalence between $I_{c}$ and $I_{n}$. Here we also require \reco{alphanprime} that is proved in the Appendix. In \repr{deftaunhat}, we introduce an element $\widehat{\tau}_{n}$ of $\pi_{1}(I_{n})$ that under the projection from $\pi_{1}(I_{n})$ to $P_{n}(M)$, is sent to the full-twist braid $\ft$ of $P_{n}(M)$, and we relate $\widehat{\tau}_{n}^{2}$ to the elements of the set $\partial_{n}(\mathcal{B})$. In \resec{n3S2}, we describe $\partial_{n}$ in the case $n=n_{0}$, the main results being \reth{brrelnS2} and \reco{sesnn0}. The proof of \reth{brrelnS2} is geometric in nature, and makes use of \relem{deformwS2} that is also used later on in the paper. The analysis of $\partial_{n}$ when $n>n_{0}$ is carried out in \resec{general}. The basic idea is to consider the possible projections of $I_{n}$ onto $I_{n_{0}}$. However, need to take care with the basepoints here, and we bring in to play various homeomorphisms ensure that they coincide with those of a standard copy of $I_{n_{0}}$ as studied in \resec{n3S2}. The main result of \resec{general} is the following.

\begin{thm}\label{th:tauhatsquareS2} 
In $\pi_{1}(I_{n})$, we have:
\begin{equation*}
\widehat{\tau}_{n}^{2}=\begin{cases}
\partial_{n}(\widetilde{\lambda}_{\widetilde{x}_0}+\cdots+\widetilde{\lambda}_{\widetilde{x}_{n-3}} + \widetilde{\lambda}_{\widetilde{z}_{0}} - \widetilde{\lambda}_{-\widetilde{z}_{0}}) & \text{if $M=\St$}\\
\partial_{n}(\lambda_{x_0}+\cdots+\lambda_{x_{n-2}} + \lambda_{z_{0}}) & \text{if $M=\rp$.}
\end{cases}
\end{equation*}
\end{thm}

From this, we obtain \reco{sesnngen} that describes $\partial_{n}$ completely, and that generalises \reco{sesnn0}. 
This enables us to reprove several isomorphisms involving $P_{n}(\St)$, and to build upon the description of $\Gamma_{2}(P_{n}(\rp))$ given in~\cite{GGgold}.

\begin{prop}\label{prop:sesnngen}
Let $M=\St$ or $\rp$, and let $n\geq n_{0}$.
\begin{enumerate}[(a)]
\item\label{it:sesnngena} If $M=\St$ then there are isomorphisms:
\begin{equation}\label{eq:isoPnS2}
\ker{(\iota_{n})_{\#1}}\!=P_{n}(\St) \cong P_{n-1}/\langle \Delta_{n-1}^{4} \rangle \cong \F[n-2]\rtimes (\F[n-3] \rtimes (\cdots\rtimes(\F[3]\rtimes \F[2])\cdots))\times \Z_{2}.
\end{equation}
\item If $M=\rp$ then there are isomorphisms:
\begin{equation}\label{eq:isoKnGn}
\!\ker{(\iota_{n})_{\#1}}\!=\!\Gamma_{2}(P_{n}(\rp))\!\cong\! G_{n-1}/\langle \Theta_{n-1}^{2} \rangle \!\cong\! \F[2n-3]\rtimes (\F[2n-5] \rtimes (\cdots\rtimes(\F[5]\rtimes \F[3])\cdots))\times \Z_{2}.
\end{equation}
\end{enumerate}
In each case, the $\Z_{2}$-factor corresponds to the subgroup $\ang{\ft}$ of $P_{n}(M)$. 
\end{prop}

The result of part~(\ref{it:sesnngena}) is in agreement with~\cite[equation~(2)]{GGgold}. 


Finally, in an Appendix, we prove \repr{homofibre} that relates the homotopy fibres of fibre spaces and certain subspaces, and that implies that one of the maps that appears in the between $I_{c}$ and $I_{n}$ is indeed a homotopy equivalence. This result seems to be well known to the experts, but we were not able to find a proof in the literature. 

Some of the results and constructions of this paper have since been generalised in~\cite{GGG}. More precisely, if $X$ is a topological manifold without boundary of dimension at least three, under certain conditions, the homotopy type of the homotopy fibre of the inclusion map $\map{\iota_{n}}{F_n(X)}[\prod_1^n\, X]$ was determined, and was used to study the cases where either the universal covering of $X$ is contractible,
or $X$ is an orbit space $\St[k]/G$ of a tame, free action of a Lie group $G$ on the $k$-sphere $\St[k]$. A complete description of the long exact sequence in homotopy of the homotopy fibration of $\iota_{n}$, similar to that of \reco{lesfibs2}, was given in the case $X=\St[k]/G$, where the group $G$ is finite and $k$ is odd. The authors have also written a survey that summarises the current situation, and that includes some questions and open problems about graph and (orbit) configuration spaces~\cite{GGsurvey}.


\subsubsection*{Acknowledgements}

The authors are grateful to Michael~Crabb for his help with the results of the Appendix, and especially for proposing proofs of \relem{crabb1} and \repr{homofibre}. This work on this paper started in 2013 initially as part of the paper~\cite{GGgold}. The authors were partially supported by the FAPESP Projeto Temático Topologia Algébrica, Geométrica 2012/24454-8 (Brazil), by the Réseau Franco-Brésilien en Mathématiques, by the CNRS/FAPESP project n\textsu{o}~226555 (France) and n\textsu{o}~2014/50131-7 (Brazil), and the CNRS/FAPESP PRC project n\textsu{o}~275209 (France) and n\textsu{o}~2016/50354-1 (Brazil).



\section{Properties of $\iota_{n}$ and $(\iota_{n})_{\#k}$}\label{sec:mapiota}

In this section, we determine some properties of $\iota_{n}$ and of $(\iota_{n})_{\#k}$, where $k\in \N$. If $X$ and $Y$ are topological spaces, and $\map{f,g}{X}[Y]$ are maps between $X$ and $Y$, we write $f\simeq g$ if $f$ and $g$ are homotopic, and we denote the homotopy class of $f$ by $[f]$.

We first state the following description of the homotopy type of the universal covering of the configuration spaces of $\St$ and $\rp$.

\begin{thm}[\cite{BCP,FZ}, {\cite[Proposition~10(b)]{GG3}}]\label{th:homotopytypefns}  
Let $M=\St$ or $\rp$, and let $n\in \N$. Then the universal covering $\widetilde{F_{n}(M)}$ of $F_{n}(M)$ has the homotopy type of $\St$ if $M=\St$ and $n\leq 2$ or if $M=\rp$ and $n=1$, and has the homotopy type of the $3$-sphere $\St[3]$ otherwise.
\end{thm}

\begin{rem}\label{rem:pi2triv}
If $M=\St$ (resp. $M=\rp$), then by~\cite[Corollary, page~244]{FvB} (resp.~\cite[Corollary, page~82]{vB}), $\pi_{2}(F_{n}(M))$ is trivial for all $n\geq 3$ (resp.\ $n\geq 2$).
\end{rem}

Let $M=\St$ or $\rp$ and let $n\geq 2$. For $i\in \brak{1,\ldots,n}$, let $\map{p_{i}}{F_{n}(M)}[M]$ and $\map{\widetilde{p}_{i}}{\prod_{1}^n\, M}[M]$ denote the respective projections onto the $i$\up{th}coordinate. If $1\leq i<j\leq n$, let $\map{\alpha_{i,j}}{F_{n}(M)}[F_{2}(M)]$ and $\map{\widetilde{\alpha}_{i,j}}{\prod_{1}^n\, M}[\prod_{1}^2\, M]$ denote the respective projections onto the $i$\up{th}and $j$\up{th}coordinates. Observe that the maps $p_{i}$, $\widetilde{p}_{i}$, $\alpha_{i,j}$  and $\widetilde{\alpha}_{i,j}$ are fibrations, and that:
\begin{equation}\label{eq:compproj}
\text{$p_{i}=\widetilde{p}_{i} \circ \iota_{n}$ and $\iota_{2}\circ \alpha_{i,j}=\widetilde{\alpha}_{i,j}\circ \iota_{n}$.}
\end{equation}
Let $1\leq i<j\leq n$. As maps from $F_{n}(M)$ to $M$, we have $\widetilde{p}_{i}\circ \iota_{n}=p_{1}\circ \alpha_{i,j}$ and $\widetilde{p}_{j}\circ \iota_{n}=\widetilde{p}_{2} \circ \iota_{2} \circ \alpha_{i,j}$, and using~\reqref{compproj}, we obtain the following commutative diagram:
\begin{equation}\label{eq:bigdiag}
\begin{tikzcd}[cramped, ampersand replacement=\&]
F_{n}(M) \arrow[swap, bend right=60]{ddd}{p_{j}}\arrow{rr}{\alpha_{i,j}} \arrow{dd}{\iota_{n}} \arrow{rd}{p_{i}} \& \& F_{2}(M)\arrow[swap]{ld}{p_{1}}\arrow[swap]{dd}{\iota_{2}} \arrow[bend left=60]{ddd}{p_{2}}\\
\& M \&\\
\prod_{1}^{n}\,M \arrow{ru}{\widetilde{p}_{i}} \arrow{rr}{\widetilde{\alpha}_{i,j}} \arrow{d}{\widetilde{p}_{j}} \& \& \prod_{1}^{2}\,M\arrow[swap]{lu}{\widetilde{p}_{1}} \arrow[swap]{d}{\widetilde{p}_{2}}\\
M \arrow[equal]{rr}  \& \&M.
\end{tikzcd}
\end{equation}
In all of what follows, let $\map{\tau}{\St}$ denote the antipodal map defined by $\tau(x)=-x$ for all $x\in \St$, and let $\map{\pi}{\St}[\rp]$ denote the universal covering.

\begin{lem}\label{lem:iotainject}
Let $M=\St$ or $\rp$, and let $n\geq 2$.
\begin{enumerate}[(a)]
\item\label{it:iotainjecta} Let $1\leq i\leq n$. If either $k\geq 3$, or $k=n=2$ and $M=\St$, the homomorphism $\map{(p_{i})_{\#k}}{\pi_{k}(F_{n}(M))}[\pi_{k}(M)]$ is an isomorphism.
\item\label{it:iotainjectb} Let $k\geq 2$. Then the homomorphism $\map{(\iota_{n})_{\#k}}{\pi_{k}(F_{n}(M))}[\pi_{k}(\prod_{1}^n\, M)]$ is injective.
\end{enumerate} 
\end{lem}

\begin{rem}
Let $n$ and $k$ be as in \relem{iotainject}(\ref{it:iotainjecta}). \reth{homotopytypefns} implies that $\pi_{k}(F_{n}(M)) \cong \pi_{k}(M)$, and the lemma provides an explicit isomorphism.
\end{rem}

\begin{proof}[Proof of \relem{iotainject}]
Let $n,k\geq 2$. 
\begin{enumerate}[(a)]
\item Let $1\leq i\leq n$. Then $\map{p_{i}}{F_{n}(M)}[M]$ is a Fadell-Neuwirth fibration whose fibre, which is an Eilenberg-Mac~Lane space of type $K(\pi,1)$, may be identified with $F_{n-1}(M\setminus \brak{x_{i}})$. The result follows by taking the long exact sequence in homotopy of this fibration (note that the fibre is contractible if $M=\St$ and $n=2$).

\item If $k\geq 3$, or $k=n=2$ and $M=\St$ then the statement is a consequence of part~(\ref{it:iotainjecta}) and~\reqref{compproj}. So assume that $k=2$ and that $n\neq 2$ if $M=\St$. The result is a consequence of \rerem{pi2triv}.\qedhere
\end{enumerate}
%
%
\end{proof}

As we shall now see, the homomorphism $\map{(\iota_{n})_{\#k}}{\pi_{k}(F_{n}(M))}[\pi_{k}(\prod_{1}^{n}\,M)]$ is diagonal for most values of $n$ and $k$, with the exception of one case when it is anti-diagonal. 

\begin{prop}\label{prop:diagS2}
Let $n,k\geq 2$.
\begin{enumerate}[(a)]
\item\label{it:diaga} If $n=k=2$ then the homomorphism $\map{(\iota_{2})_{\#2}}{\pi_{2}(F_{2}(\St))}[\pi_{2}(\prod_{i=1}^{2}\,\St)]$ is the anti-diagonal homomorphism that sends a generator of $\pi_{2}(F_{2}(\St))$ to $(1,-1)$ or to $(-1,1)$ depending on the choice of orientation of $\St$.

\item\label{it:diagb} If $n\geq 2$ and $k\geq 3$ then $\map{(\iota_{n})_{\#k}}{\pi_{k}(F_{n}(\St))}[\pi_{k}(\prod_{1}^{n}\,\St)]$ is a diagonal homomorphism.

\end{enumerate}
\end{prop}


\begin{proof}
Let $n=k=2$ (resp.\ $n\geq 2$ and $k\geq 3$). In part~(\ref{it:diaga}) (resp.~(\ref{it:diagb})), we shall show that $\map{(\iota_{n})_{\#k}}{\pi_{k}(F_{n}(\St))}[\pi_{k}(\prod_{1}^{n}\,\St)]$ is an anti-diagonal (resp.\ diagonal) homomorphism. To do so, it suffices to prove that $(\widetilde{p}_{1})_{\#2}\circ (\iota_{n})_{\#2}=-(\widetilde{p}_{2})_{\#2}\circ (\iota_{n})_{\#2}$ (resp.\ $(\widetilde{p}_{i})_{\#k}\circ (\iota_{n})_{\#k}=(\widetilde{p}_{j})_{\#k}\circ (\iota_{n})_{\#k}$ for all $1\leq i<j\leq n$), which is equivalent to showing that $(p_{1})_{\#2}=-(p_{2})_{\#2}$ (resp.\ $(p_{i})_{\#k}=(p_{j})_{\#k}$).

\begin{enumerate}[(a)]
\item Suppose that $n=k=2$. Thus $\pi_{2}(F_{2}(\St))\cong \pi_{2}(\St)\cong \Z$, and for $i=1,2$, $p_{i}$ is a homotopy equivalence, with homotopy inverse $\map{s_{i}}{\St}[F_{2}(\St)]$, where $s_{1}(x)=(x,-x)$ and $s_{2}(x)=(-x,x)$~\cite[page~859]{GG12}. Since $p_{1}\circ s_{1}=\id_{\St}$ and $p_{2}\circ s_{1}=\tau$, we have $(s_{1})_{\#2}=((p_{1})_{\#2})^{-1}$, $(p_{2})_{\#2}\circ ((p_{1})_{\#2})^{-1}=\tau_{\#2}$, and $(p_{2})_{\#2} =\tau_{\#2}\circ (p_{1})_{\#2}=-(p_{1})_{\#2}$ since $\tau$ is of degree $-1$. Now $\pi_{2}(F_{2}(\St))\cong \pi_{2}(\St) \cong \Z$, and once we have fixed an orientation of $\St$ and a generator of $\pi_{2}(F_{2}(\St))$, $(\iota_{2})_{\#2}$ sends this generator to $(1,-1)$ or to $(-1,1)$ using~\reqref{compproj} and \relem{iotainject}(\ref{it:iotainjecta}).

\item Let $k\geq 3$. To compare the homomorphisms $\map{(p_{i})_{\#k}}{\pi_{k}(F_{n}(\St))}[\pi_{k}(\St)]$, where $i\in \brak{1,\ldots,n}$, we analyse the elements $(p_{i})_{\#k}([\alpha])$ of $\pi_{k}(\St)$ for maps $\map{\alpha}{\St[k]}[F_{n}(\St)]$. We split the proof into two cases.

\begin{enumerate}[(i)]
\item First suppose that $n=2$. Let $\beta=p_{1}\circ \alpha$, and let $s_{1}$ be as in the proof of~(\ref{it:diaga}). Then $s_{1}\circ \beta = s_{1}\circ p_{1}\circ \alpha \simeq \alpha$. Since $\beta$ is a map from $\St[k]$ to $\St$, it follows from the long exact sequence in homotopy of the Hopf fibration $\St[1]\to \St[3]\stackrel{H}{\to} \St$ that $\beta$ factors through the Hopf map $\map{H}{\St[3]}[\St]$, and so there exists $\map{\beta_{1}}{\St[k]}[{\St[3]}]$ such that $\beta=H\circ \beta_{1}$. We thus have the following diagram: 
\begin{equation*}
\begin{tikzcd}[cramped]
\St[k] \arrow{r}{\alpha} \arrow{rd}{\beta} \arrow{d}{\beta_{1}} & F_{2}(\St) \arrow[xshift=0.5ex]{d}{p_{1}}\\
\St[3] \arrow{r}{H} & \St \arrow[xshift=-0.5ex]{u}{s_{1}}
\end{tikzcd}
\end{equation*}
that is commutative, with the exception of the paths involving $s_{1}$ that are commutative up to homotopy. Using the homotopy equivalences of part~(\ref{it:diaga}), for $i=1,2$, we have:
\begin{equation*}
p_{i}\circ \alpha\simeq p_{i}\circ s_{1}\circ \beta = p_{i} \circ s_{1}\circ H\circ \beta_{1}\simeq \tau^{i-1} \circ H\circ \beta_{1}.
\end{equation*}
But by a result of Hopf (see~\cite[page~269, Lemma~8.2.5]{A} or~\cite[Satz~II and~IIb']{Ho}), $\tau\circ H\simeq H$, and so $p_{1}\circ \alpha = \beta \simeq p_{2}\circ \alpha$. Thus $(p_{1}\circ \alpha)_{\#k} = (p_{2}\circ \alpha)_{\#k}$ for all maps $\map{\alpha}{\St[k]}[F_{2}(\St)]$, hence $(p_{1})_{\#k} = (p_{2})_{\#k}$, from which we conclude that $(\iota_{2})_{\#k}$ is a diagonal homomorphism.

\item\label{it:diagbii} Now suppose that $n\geq 3$. To prove that $(\iota_{n})_{\#k}$ is diagonal, it suffices to show that the homotopy class $[p_{i} \circ \alpha]$ in $\pi_{k}(\St)$ of the map $\map{p_{i} \circ \alpha}{\St[k]}[\St]$ does not depend on $i$, or equivalently that $[p_{j} \circ \alpha]=[p_{1} \circ \alpha]$ in $\pi_{k}(\St)$ for all $2\leq j \leq n$.
As we saw in the previous paragraph, $\map{(\iota_{2})_{\#k}}{\pi_{k}(F_{2}(\St))}[\pi_{k}(\prod_{i=1}^{2}\,\St)]$ is a diagonal homomorphism, in particular, $(\widetilde{p}_{1}\circ \iota_{2})_{\#k}=(\widetilde{p}_{2}\circ \iota_{2})_{\#k}$. Using this and taking $i=1$ and $2\leq j\leq n$ in the commutative diagram~\reqref{bigdiag}, for any map $\map{\alpha}{\St[k]}[F_{n}(\St)]$, we have:
\begin{equation*}
[p_{1} \circ \alpha] = [\widetilde{p}_{1}\circ \iota_{2}\circ \alpha_{1,j} \circ \alpha]=[\widetilde{p}_{2}\circ \iota_{2}\circ \alpha_{1,j} \circ \alpha] =[p_{j} \circ \alpha],
\end{equation*}
and thus $(p_{1})_{\#k}=(p_{j})_{\#k}$ as required.
%
\qedhere
\end{enumerate}
\end{enumerate}
\end{proof}

The following proposition is the analogue of \repr{diagS2} for $\rp$.

\begin{prop}\label{prop:diagRP2}
Let $n,k\geq 2$. Then $\map{(\iota_{n})_{\#k}}{\pi_{k}(F_{n}(\rp))}[\pi_{k}(\prod_{1}^{n}\,\rp)]$ is a diagonal homomorphism.
\end{prop}

\begin{rem}
Let $M=\St$ (resp.\ $M=\rp$). In the cases not covered by \relem{iotainject}(\ref{it:iotainjecta}), namely $k=2$ and $n\geq 3$ (resp.\ $n\geq 2$), $\map{(\iota_{n})_{\#2}}{\pi_{2}(F_{n}(M))}[\pi_{2}(\prod_{1}^{n}\,M)]$ is trivially diagonal by \rerem{pi2triv}, but the homomorphism $(p_{i})_{\#k}$ of \relem{iotainject}(\ref{it:iotainjecta}) is not an isomorphism.
\end{rem}

\begin{proof}[Proof of \repr{diagRP2}]
Let $k,n\geq 2$. We consider four cases.
\begin{enumerate}[(a)]
\item If $k=2$ then $\pi_{2}(F_{n}(\rp))=1$ by \rerem{pi2triv}, and hence $(\iota_{n})_{\#2}$ is the trivial (diagonal) homomorphism.

\item\label{it:n2k3} Let $n=2$ and $k=3$. By~\reqref{compproj}, it suffices to show that $(p_{1})_{\#3}=(p_{2})_{\#3}$. With \reth{homotopytypefns} in mind, let $\gamma$ (resp.\ $\rho$) denote a generator of the infinite cyclic group $\pi_{3}(F_{3}(\rp))$ (resp.\ $\pi_{3}(\rp)$). By \relem{iotainject}(\ref{it:iotainjecta}), $\map{(p_{l})_{\#3}}{\pi_{3}(F_{3}(\rp))}[\pi_{3}(\rp)]$ is an isomorphism for all $1\leq l\leq 3$, so there exists $\epsilon_{l}\in \brak{1,-1}$ such that $(p_{l})_{\#3}(\gamma)=\epsilon_{l} \rho$. Hence there exist $1\leq i<j\leq 3$ and $\epsilon\in \brak{1,-1}$ such that $(p_{i})_{\#3}(\gamma)=(p_{j})_{\#3}(\gamma)=\epsilon \rho$. Consider~\reqref{bigdiag}, where we take $n=3$. The map $\alpha_{i,j}$ is a Fadell-Neuwirth fibration, whose fibre may be identified with $F_{1}(\rp \setminus \brak{x_{i},x_{j}})$, and as in the proof of \relem{iotainject}(\ref{it:iotainjecta}), we see that $\map{(\alpha_{i,j})_{\#3}}{\pi_{k}(F_{3}(\rp))}[\pi_{k}(F_{2}(\rp))]$ is an isomorphism, thus $\gamma'=(\alpha_{i,j})_{\#3}(\gamma)$ is a generator of the infinite cyclic group $\pi_{3}(F_{2}(\rp))$.
Taking the commutative diagram~\reqref{bigdiag} on the level on $\pi_{3}$, we see that:
\begin{align*}
\epsilon \rho &=(p_{i})_{\#3}(\gamma)=(p_{1}\circ \alpha_{i,j})_{\#3}(\gamma)= (p_{1})_{\#3}(\gamma') \;\text{and}\\
\epsilon \rho &=(p_{j})_{\#3}(\gamma)=(p_{2}\circ \alpha_{i,j})_{\#3}(\gamma)= (p_{2})_{\#3}(\gamma'),
\end{align*}
from which it follows that $(p_{1})_{\#3}=(p_{2})_{\#3}$ as required.

\item\label{it:n2k4} Suppose that $n=2$ and $k\geq 4$. Let $\map{h}{\widetilde{F_{2}(\rp)}}[F_{2}(\rp)]$ be the universal covering map, and with \reth{homotopytypefns} in mind, let $\map{f}{\St[3]}[\widetilde{F_{2}(\rp)}]$ and $\map{g}{\widetilde{F_{2}(\rp)}}[{\St[3]}]$ be homotopy equivalences. If $i\in \brak{1,2}$ then $\map{p_{i}\circ h\circ f}{\St[3]}[\rp]$ lifts to a map $\map{\widehat{p}_{i}}{\St[3]}[\St]$, and we have the following commutative diagram:
\begin{equation}\label{eq:ntwokthree}
\begin{tikzcd}[cramped]
\St[3] \arrow{r}{\widehat{p}_{i}}  \arrow{d}{h\circ f} & \St \arrow{d}{\pi}\\
F_{2}(\rp) \arrow{r}{p_{i}} &  \rp.
\end{tikzcd}
\end{equation}
From \relem{iotainject}(\ref{it:iotainjecta}) and the proof of case~(\ref{it:n2k3}) above, $\map{(p_{i})_{\#k}}{\pi_{k}(F_{2}(\rp))}[\pi_{k}(\rp)]$ is an isomorphism and  $(p_{i})_{\#3}(\gamma')=\epsilon \rho$. Considering the commutative diagram~\reqref{ntwokthree} on the level of $\pi_{3}$, $(h\circ f)_{\#3}$ and $\pi_{\#3}$ are also isomorphisms, and it follows that $(\widehat{p}_{i})_{\#3}=(\pi_{\#3})^{-1} \circ (p_{i})_{\#3} \circ (h\circ f)_{\#3}$ is an isomorphism that is independent of $i$. Hence $\widehat{p}_{i}$ is homotopic to the Hopf map $H$, and on the level of $\pi_{k}$, we have $(p_{i})_{\#k}=\pi_{\#k} \circ H_{\#k} \circ ((h\circ f)_{\#k})^{-1}$ for all $k\geq 3$. We conclude that $(p_{1})_{\#k}=(p_{2})_{\#k}$ as required.

\item Finally, if $n,k\geq 3$, it suffices to replace $\St$ by $\rp$ in case~(\ref{it:diagb})(\ref{it:diagbii}) of the proof of \repr{diagS2}, and to make use of the fact proved in cases~(\ref{it:n2k3}) and~(\ref{it:n2k4}), that for $n=2$ and all $k\geq 3$, $(\iota_{2})_{\#k}$ is a diagonal homomorphism.\qedhere
\end{enumerate}
\end{proof}

\begin{rem}
To prove the result of \repr{diagRP2} in the case $n=2$ and $k=3$, one may make use of~\cite[Proposition~2.13]{DG} or~\cite[Theorem~1.17]{K} to show that the pair of maps $\map{(\pi\circ H,\pi\circ H)}{\St[3]}[\rp\times \rp]$ may be deformed to a pair of coincidence-free maps, and so up to homotopy, factors through $F_{2}(\rp)$.
\end{rem}

\section{The homotopy type of the homotopy fibre of $\iota_{n}$ for $M=\St,\rp$} \label{sec:sec3}

The aim of this section is to prove \reth{prop5} concerning the homotopy type of the homotopy fibre of the map $\map{\iota_{n}}{F_{n}(M)}[\prod_1^n \, M]$, where $M=\St$ or $\rp$. In \resec{generalities}, we define much of the basic notation that will be used in the rest of the paper, and we advise the reader to refer to this section frequently during the rest of the manuscript. 
In \resec{fibreiotaS2}, we shall prove \reth{prop5} and \reco{lesfibs2}.

\subsection{Generalities and notation}\label{sec:generalities}


Let $X$ be a topological space, and let $x_{0}\in X$ be a basepoint, and let $I=[0,1]$. If $l$ is a path in $X$ parametrised by $t\in I$ then let $l^{-1}$ denote its inverse, in other words $l^{-1}(t)=l(1-t)$ for all $t\in I$, and if $x\in X$ then let $c_{x}$ denote the constant path at $x$. If $l,l'$ are paths in $X$ such that $l(1)=l'(0)$, then $l\ast l'$ will denote their concatenation. Let $\Omega X$ denote the loop space of $(X,x_{0})$, and if $\omega\in \Omega X$ then let $[\omega]$ denote the associated element of $\pi_{1}(X,x_{0})$. If $Y$ is a topological space and $\map{f}{X}[Y]$ is a map, then we shall always take the basepoint of $E_{f}$ and $I_{f}$ defined in equations~\reqref{defmappath} and~\reqref{defhomfibre} to be $(x_{0}, c_{f(x_{0})})$.


We define the following notation that will be used throughout the rest of the manuscript. Let $\St$ be the unit sphere in $\R^{3}$, let $\widetilde{z}_{0}=(0,0,1)$, let $\dd$ be the open disc $\St\setminus \brak{-\widetilde{z}_{0}}$, let $C$ be the open cylinder $\St\setminus \brak{\widetilde{z}_{0},-\widetilde{z}_{0}}$, and let $\map{\tau}{\St}$ be the antipodal map. We equip $\St$ with spherical coordinates $(\theta,\phi)$, where $\theta\in [0,2\pi)$ is the longitude whose prime meridian passes through $\widetilde{x}_{0}$ and $\phi\in [-\frac{\pi}{2},\frac{\pi}{2}]$ is the latitude. Let $n_{0}=3$ (resp.\ $n_{0}=2$) if $M=\St$ (resp.\ $\rp$), let $n\geq n_{0}$, and for $j\in \brak{0,1,\ldots,n-n_{0}}$, let $\widetilde{x}_{j}=\bigl(0, \frac{\pi}{2}\bigl(\frac{j}{j+1}\bigr)\bigr)$ and $x_{j}=\pi(\widetilde{x}_{j})$, where $\map{\pi}{\St}[\rp]$ denotes the universal covering map. Let $M=\St$ (resp.\ $M=\rp$), and let $\mathbb{M}=\dd$ (resp.\ $\mathbb{M}=C$). We take the basepoint $W_{n}=(w_{1},\ldots,w_{n})$ of both $F_{n}(M)$ and $\prod_{1}^{n}\, M$ to be $(\widetilde{x}_{0},\widetilde{x}_{1}, \ldots, \widetilde{x}_{n-3}, \widetilde{z}_{0},-\widetilde{z}_{0})$ (resp.\ $(x_{0},x_{1},\ldots, x_{n-2},z_{0})$), where $z_{0}=\pi(\widetilde{z}_{0})$, and the basepoint of both $F_{n-1}(M\setminus\brak{w_n})$ and $\prod_{i=1}^{n-1}\, M$ to be $W_{n}'= (w_{1},\ldots,w_{n-1})$. Let $G=\brak{\id_{\dd}}$ (resp.\ $G=\ang{\tau\left\lvert_{C}\right.}$). Then $F_{n-1}^{G}(\mathbb{M})$ is the usual configuration space $F_{n-1}(\dd)$ of $\dd$ (resp.\ the orbit configuration space $F_{n-1}^{\ang{\tau}}(C)$ of $C$) that we equip with the basepoint $W_{n}''=(w_{1}'',\ldots,w_{n-1}'')$, where $W_{n}''=(\widetilde{x}_{0},\widetilde{x}_{1}, \ldots, \widetilde{x}_{n-3}, \widetilde{z}_{0})$ (resp.\ $W_{n}''=(\widetilde{x}_{0},\widetilde{x}_{1}, \ldots, \widetilde{x}_{n-2})$).  
So as not to overload the notation, we shall omit these basepoints in much of what follows, and we shall identify $\pi_{1}\left(\prod_{1}^{n} \, \Omega M, (c_{w_{1}}, \ldots, c_{w_n})\right)$ with $\pi_{2}\left(\prod_{1}^{n} \, M, W_{n})\right)$ in the usual manner. 

Let $E_{n}$ and $I_{n}$ (resp.\ $E_{n}'$ and $I_{n}'$) denote the mapping path and homotopy fibre respectively of the inclusion $\altmap{\iota_{n}}{F_{n}(M)}{\prod_{1}^{n}\, M}$ (resp.\ $\altmap{\iota_{n}'}{F_{n-1}(M\setminus\brak{w_n})}{\prod_{i=1}^{n-1}\, M}$)
that are defined using~\reqref{defmappath} and~\reqref{defhomfibre} by:
\begin{align}
E_{n} &=\setl{(x,\lambda)\in F_{n}(M)\times \prod_{1}^{n}\, M^{I}}{\text{$\lambda(0)=x$}}\notag\\
I_{n} &= \setl{(x, \lambda)\in E_{n}}{\lambda(1)= W_{n}}\notag\\
E_{n}' &= \setl{(y, \mu)\in F_{n-1}(M\setminus\brak{w_n})\times \prod_{i=1}^{n-1}\, M^{I}}{\text{$\mu(0)=y$}}\notag\\
I_{n}' &= \setl{(y, \mu)\in E_{n}'}{\mu(1)=W_{n}'}.\label{eq:defI1}
\end{align}
The basepoints are chosen in accordance with the convention of the first paragraph. Let $\map{j_{n}}{I_{n}}[E_{n}]$ and $\map{j_{n}'}{I_{n}'}[E_{n}']$ denote the associated inclusions. By~\cite[Proposition~3.5.8(2) and Remark~3.5.9(1)]{A}, the maps $\map{f_{n}}{F_{n}(M)}[E_{n}]$ and $\map{g_{n}}{E_{n}}[F_{n}(M)]$ (resp.\ $\map{f_{n}'}{F_{n-1}(M\setminus\brak{w_n})}[E_{n}']$ and $\map{g_{n}'}{E_{n}'}[F_{n-1}(M\setminus\brak{w_n})]$) defined by $f_{n}(x)=(x,c_{\iota_{n}(x)})$ and $g_{n}(x,\lambda)=x$ (resp.\ $f_{n}'(y)=(y,c_{\iota_{n}'(y)})$ and $g_{n}'(y,\mu)=y$) are mutual homotopy equivalences. 
Let $\map{\alpha_{n}}{E_{n}'}[E_{n}]$ be defined by $\alpha_{n}(y, \mu)= (y,w_{n}, \mu,c_{w_{n}})$, let $\map{\widehat{\alpha}_{n}}{\prod_{i=1}^{n-1}\, M}[\prod_{1}^{n}\, M]$ be defined by $\widehat{\alpha}_{n}(y)= (y,w_{n})$, and let
\begin{equation*}
 \map{\widehat{\alpha}_{n}\left\lvert_{F_{n-1}(M \setminus\brak{w_{n}})}\right.}{F_{n-1}(M \setminus\brak{w_{n}})}[F_{n}(M)]
\end{equation*}
denote its restriction to the corresponding configuration spaces. We define the map $\map{\widehat{\iota}_{n}}{E_{n}}[\prod_{1}^{n}\, M]$ (resp.\ $\map{\widehat{\iota'_{n}}}{E_{n}'}[\prod_{i=1}^{n-1}\, M]$) by $\widehat{\iota}_{n}(x,\lambda)=\lambda(1)$ (resp.\ $\widehat{\iota_{n}'}(y,\mu)=\mu(1)$). Note that $\widehat{\iota}_{n}\circ f_{n}= \iota_{n}$, and that $\alpha_{n}$ restricts to a map from $I_{n}'$ to $I_{n}$ that we denote by $\alpha_{n}'$. 
Applying~\cite[pages~91~and~108]{A},
we obtain the following commutative diagram:
\begin{equation}\label{eq:commdiagfib}
\begin{tikzcd}[cramped]
I_{n}' \arrow[r, hook, "j_{n}'"] \arrow[d, "\alpha_{n}'"]
& E_{n}' \arrow{r}{\widehat{\iota_{n}'}} \arrow[d, "\alpha_{n}"]
& \prod_{i=1}^{n-1}\, M \arrow[d, "\widehat{\alpha}_{n}"]\\
I_{n} \arrow[r, hook, "j_{n}"] 
& E_{n} \arrow{r}{\widehat{\iota}_{n}} 
& \prod_{1}^{n}\, M,
\end{tikzcd}
\end{equation}
where the rows are fibrations.  
Let $\map{d_{n}'}{\prod_{i=1}^{n-1}\, \Omega M}[I_{n}']$ (resp.\ $\map{d_{n}}{\prod_{1}^{n}\, \Omega M}[I_{n}]$) denote the boundary map corresponding to the upper (resp.\ lower) fibration of~\reqref{commdiagfib}, and for all $k\geq 1$, let:
\begin{equation*}
\text{$\map{\partial_{n,k}^{\,\prime}=(d_{n}')_{\#k}}{\pi_{k}\Bigl(\prod_{i=1}^{n-1}\, M\Bigr)}[\pi_{k-1}(I_{n}')]$ (resp.\ $\map{\partial_{n,k}=(d_{n})_{\#k}}{\pi_{k}\Bigl(\prod_{1}^{n}\, M\Bigr)}[\pi_{k-1}(I_{n})]$)}
\end{equation*}
denote the induced boundary homomorphism on the level of $\pi_{k}$. We recall the construction of these maps in the following lemma that will be used in \resec{boundhomo}. 

\begin{lem}\label{lem:boundary}
With the above notation, we have $d_{n}'(\omega)=(w_{1},\ldots, w_{n-1}, \omega)$ (resp.\ $d_{n}(\omega)=(w_{1},\ldots, w_{n}, \omega)$) for all $\omega\in \prod_{i=1}^{n-1}\, \Omega M$ (resp.\ $\omega\in \prod_{1}^{n}\, \Omega M$). 
\end{lem}

\begin{proof}
The statement of the lemma is a consequence of the general construction of the boundary map of a fibration, which we now recall. Let $\map{f}{(X,x_{0})}[(Y,y_{0})]$ be a map between based topological spaces, let $I_{f}$ denote the homotopy fibre of $f$, let $EY= \setl{\lambda \in Y^{I}}{\lambda(1)=y_{0}}$ be the path space of $Y$, and let $\Omega Y\to EY\stackrel{P_{0}}{\to} Y$ be the path space fibration~\cite[page~85]{A}, where $P_{0}(\lambda)=\lambda(0)$. By~\cite[Proposition~3.3.12 and Definition~3.3.13]{A}, we have the following pullback square:
\begin{equation*}
\begin{tikzcd}[cramped]
I_{f} \arrow{r}{u} \arrow{d}{v} & EY\arrow{d}{P_{0}}\\
X \arrow{r}{f} & Y,
\end{tikzcd}
\end{equation*}
where $u(x,\lambda)=\lambda$ and $v(x,\lambda)=x$ for all $(x,\lambda)\in I_{f}$, $v$ is the principal fibration induced by $f$, and $u$ induces a homeomorphism between the fibre $F_{v}$ of $v$ and $\Omega Y$ whose inverse $\map{u'}{\Omega Y}[F_{v}]$ is defined by $u'(\omega)=(x_{0},\omega)$ for all $\omega\in \Omega Y$. The boundary map $\map{d}{\Omega Y}[I_{f}]$ is then defined by $d= i\circ u'$, where $\map{i}{F_{v}}[I_{f}]$ is inclusion, so $d(\omega)=(x_{0},\omega)$ for all $\omega\in \Omega Y$.
The result then follows by taking $X=F_{n-1}(M\setminus \brak{w_n})$, $Y=\prod_{i=1}^{n-1}\, M$ and $f=\iota_{n}'$ (resp.\ $X=F_{n}(M)$, $Y=\prod_{1}^{n}\, M$ and $f=\iota_{n}$).
\end{proof}

The following remarks will play an important r\^{o}le in the proof of \reth{prop5}.

\begin{rems}\mbox{}\label{rems:homfibre}
\begin{enumerate}[(a)]
\item\label{it:homfibrea} Let $\map{p}{E}[B]$ be a fibration with fibre $F$, let $E'$ be a subspace of $E$, and suppose that the restriction $\map{p'=p\left\lvert_{E'}\right.\!}{E'}[B]$ is also a fibration, with fibre $F'$. Let $\map{\iota}{E'}[E]$ denote inclusion, and let $\map{\iota'=\iota\left\lvert_{F'}\right.\!}{F'}[F]$ denote the restriction of $\iota$ to $F'$. The inclusions of the fibres in the corresponding total spaces induce an explicit homotopy equivalence between the homotopy fibres $I_{\iota}$ and $I_{\iota'}$ of $\iota$ and $\iota'$ respectively. Although this result seems to be folklore and is well known to the experts in the field, we were unable to locate a proof in the literature, and so we give one 
in \repr{homofibre} in the Appendix of this paper. A consequence of this proposition is that the map $\map{\alpha_{n}'}{I_{n}'}[I_{n}]$ is a homotopy equivalence (see \reco{alphanprime}).

\item\label{it:homfibreb} Suppose that $\map{f}{X}[Y]$ is null homotopic, and let $\map{c}{X}[Y]$ denote the constant map that sends the whole of $X$ onto $y_{0}$. By~\cite[Proposition~3.3.17]{A}, $I_{f}$ and $I_{c}$ have the same homotopy type. On the other hand, $I_{c}=\setl{(x,\lambda)\in X\times Y^I}{\lambda(0)=\lambda(1)= y_{0}}= X\times \Omega Y$. Therefore $I_{f}$ has the same homotopy type as $X\times \Omega Y$.
\end{enumerate}
\end{rems}

Let $M=\St$ (resp.\ $M=\rp$),
and let $\map{\iota_{c}}{F_{n-1}^{G}(\mathbb{M})}[\prod_{i=1}^{n-1}\, \St]$ be the constant map that sends every point of $F_{n-1}^{G}(\mathbb{M})$ to $W_{n}''$, which we take to be the basepoint of $\prod_{i=1}^{n-1}\, \St$ at this level. By \rerems{homfibre}(\ref{it:homfibreb}), its homotopy fibre is given by:
\begin{equation}\label{eq:defIc}
I_{c}=F_{n-1}^{G}(\mathbb{M}) \times \Omega\biggl(\prod_{i=1}^{n-1}\, \St\biggr),
\end{equation}
which we equip with the basepoint $W_{c}$ in accordance with the convention given in the first paragraph of this section.
Let $I_{n}''$ be the homotopy fibre of the inclusion map $\map{\iota_{n}''}{F_{n-1}^{G}(\mathbb{M})}[\prod_{i=1}^{n-1}\, \St]$ defined by $\iota_{n}''(u)=u$ for all $u\in F_{n-1}^{G}(\mathbb{M})$, and also equipped with the basepoint $W_{c}$. So:
\begin{equation}\label{eq:defInprime2}
I_{n}''= \setl{(u, \mu)\in F_{n-1}^{G}(\mathbb{M})\times \left(\prod_{i=1}^{n-1}\St\right)^{I}}{\text{$\mu(0)=u$ and $\mu(1)=W_{n}''$}}.
\end{equation}
If $M=\St$, then $\iota_{n}''=\iota_{n}'$, so $I_{n}''=I_{n}'$ by~\reqref{defI1}. Let $\map{h_{n}}{I_{c}}[F_{n-1}^{G}(\mathbb{M})]$ and $\map{h_{n}''}{I_{n}''}[\!F_{n-1}^{G}(\mathbb{M})]$ denote the projections onto the first factor, and let $\map{\widehat{\pi}}{\St}[M]$ be the identity $\id_{\St}$ (resp.\ the universal covering $\map{\pi}{\St}[\rp]$). Then $\widehat{\pi}$ restricts to a map from $\dd$ (resp.\ from $C$) to $M\setminus\brak{w_{n}}$ that we also denote by $\widehat{\pi}$. Let $\map{\alpha_{\pi}}{I_{n}''}[I_{n}']$ be defined by:
\begin{equation}\label{eq:defalphapi}
\alpha_{\pi}(y_{0},\ldots, y_{n-2}, l_{0},\ldots,l_{n-2})=(\widehat{\pi}(y_{0}), \ldots, \widehat{\pi}(y_{n-2}), \widehat{\pi}\circ l_{0},\ldots, \widehat{\pi}\circ l_{n-2}).
\end{equation}
Note that if $M=\St$ then $\alpha_{\pi}$ is the identity. 
Let $\map{p}{\prod_1^{n-1}\,\St}[\prod_1^{n-1}\,M]$ denote the identity (resp.\ the $2^{n-1}$-fold universal covering), and let $\map{p'}{F_{n-1}^{G}(\mathbb{M})}[F_{n-1}(M \setminus\brak{w_{n}})]$ denote the restriction of $p$ to $F_{n-1}^{G}(\mathbb{M})$. If $\map{\alpha_{c}}{I_{c}}[I_{n}'']$ is a pointed map that satisfies $h_{n}''\circ \alpha_{c}(u,\mu)=u$ for all $(u,\mu)\in I_{c}$, using the definitions, one may check that the following diagram of pointed spaces and maps is commutative: 
\begin{equation}\label{eq:bigcommdiagfib}
\begin{tikzcd}[cramped]
I_{c} \arrow{r}{h_{n}} \arrow[d, "\alpha_{c}"] & F_{n-1}^{G}(\mathbb{M}) \arrow[r, "\iota_{c}"] \arrow[d, "\id_{F_{n-1}^{G}(\mathbb{M})}"] & \prod_{i=1}^{n-1}\,\St \arrow{d}{\id_{\prod_{i=1}^{n-1}\,\St}}\\
I_{n}'' \arrow{r}{h_{n}''} \arrow[d, "\alpha_{\pi}"] & F_{n-1}^{G}(\mathbb{M}) \arrow[r, "\iota_{n}''"] \arrow[d, "p'"] & \prod_{i=1}^{n-1}\,\St \arrow[d, "p"]\\
I_{n}' \arrow{r}{g_{n}'\circ j_{n}'} \arrow[d, "\alpha_{n}'"] & F_{n-1}(M \setminus\brak{w_{n}}) \arrow[r, "\iota_{n}'"] \arrow{d}{\widehat{\alpha}_{n}\left\lvert_{F_{n-1}(M \setminus\brak{w_{n}})}\right.} & \prod_{i=1}^{n-1}\,M \arrow{d}{\widehat{\alpha}_{n}}\\
I_{n} \arrow{r}{g_{n}\circ j_{n}} & F_{n}(M) \arrow[r, "\iota_{n}"] & \prod_{1}^{n}\, M,
\end{tikzcd}
\end{equation}
with the possible exception of the upper right-hand square. In the proof of \reth{prop5}, we shall see that this square commutes up to homotopy (see \rerem{In2null}), and in \resec{notation}, we exhibit a map $\alpha_{c}$ that satisfies the above condition (see \req{defalphac}).

To end this section, we state and prove the following lemma that will be used to analyse the upper part of the long exact sequence in homotopy of the homotopy fibration of $\iota_{n}$.

\begin{lem}\label{lem:diagonal}
Let $A$ and $B$ be Abelian groups, and let $n\geq 2$. Assume that $\map{\Theta}{A}[\bigoplus_{i=1}^n B]$ is an injective homomorphism, and suppose that there exists $i\in \brak{1,\ldots,n}$ such that the composition $\map{\pi_{i} \circ \Theta}{A}[B]$ is an isomorphism, where $\map{\pi_{i}}{\bigoplus_{i=1}^n B}[B]$ is projection onto the $i$\up{th} factor. Then the short exact sequence:
\begin{equation}\label{eq:sesab}
\textstyle 1\to A \stackrel{\Theta}{\to} \bigoplus_{i=1}^n B \stackrel{\zeta}{\to} \left(\bigoplus_{i=1}^n B\right)/\im{\Theta}\to 1
\end{equation}
splits, where $\map{\zeta}{\bigoplus_{i=1}^n B}[\left(\bigoplus_{i=1}^n B\right)/\im{\Theta}]$ denotes the canonical projection. 
\end{lem}

\begin{proof}
Since the groups that appear in the short exact sequence~\reqref{sesab} are Abelian, the existence of a section for $\zeta$ is equivalent to that of a section for $\Theta$. 
If $\map{g}{B}[A]$ is inverse of the isomorphism $\pi_{i} \circ \Theta$ then the map $\map{h}{\bigoplus_{i=1}^n B}[A]$ given by $h=g\circ \pi_{i}$ is a homomorphism and defines a section for $\Theta$.
\end{proof}


\subsection{Proof of \reth{prop5}}\label{sec:fibreiotaS2}

In this section, we prove \reth{prop5}, from which we shall deduce \reco{lesfibs2}. 
We first prove \relem{equifn}, which will play an important r\^ole in the proof of \reth{prop5}(\ref{it:prop5b}), but is interesting in its own right.

\begin{lem}\label{lem:equifn}
For all $n\in \N$, the space $\orbconf$ is a $K(G_{n}, 1)$, and $G_{n}$ is isomorphic to an iterated semi-direct product of the form $\F[2n-1]\rtimes (\F[2n-3]\rtimes (\cdots \rtimes (\F[3]\rtimes \Z)\cdots ))$, where for $i=1,\ldots, n-1$, $\F[2i+1]$ is a free group of rank $2i+1$.
\end{lem}

\begin{rem}
This description of $G_{n}$ is similar to that of Artin combing for the Artin pure braid groups. The actions will be computed in \resec{presorb}.
\end{rem}

\begin{proof}[Proof of \relem{equifn}] 
If $n=1$ then $\orbconf[1]=C$, and the result is clear. So suppose that $n\geq 2$. The proof is similar to that for the configuration spaces of the disc, and makes use of the Fadell-Neuwirth fibrations. The map $\map{q_{n}}{\orbconf}[{\orbconf[n-1]}]$ defined by forgetting the last coordinate is well defined, and as in~\cite{FaN}, may be seen to be a locally-trivial fibration (see also~\cite[Theorem~2.2]{Fa} and~\cite[Lemma~4]{CX}). The fibre over the basepoint $W_{n}''=(\widetilde{x}_{0},\widetilde{x}_{1}, \ldots, \widetilde{x}_{n-2})$ of $\orbconf[n-1]$ may be identified with the set $C \setminus \brak{\pm \widetilde{x}_{0},\ldots, \pm \widetilde{x}_{n-2}}$. Furthermore, $q_{n}$ admits a section $s_{n}$ given by choosing (for example) a point sufficiently close to $\widetilde{z}_{0}$ in a continuous manner (see also~\cite[Theorem~3.2]{Fa}, as well as the beginning of \resec{presorb} and \rerem{secondsec} for two different possibilities for $s_{n}$). Since the fibre has the homotopy type of a wedge of circles, taking the long exact sequence in homotopy of the fibration $q_{n}$, we see that $(q_{n})_{\#k}\colon\thinspace \pi_{k}\bigl(\orbconf\bigr)\to \pi_{k}\bigl(\orbconf[n-1]\bigr)$ is an isomorphism for all $k\geq 3$ and is injective if $k=2$, and that there exists a split short exact sequence:
\begin{equation}\label{eq:sesequi}
\begin{tikzcd}[cramped]
1 \arrow{r} & \pi_{1}(C \setminus\brak{\pm \widetilde{x}_{0},\ldots, \pm \widetilde{x}_{n-2}}) \arrow[hook]{r} & G_{n}  \arrow[yshift=0.5ex]{r}{(q_n)_{\#1}} & G_{n-1} \arrow[yshift=-0.5ex,dashrightarrow]{l}{(s_n)_{\#1}} \arrow{r} & 1.
\end{tikzcd}
\end{equation}
Induction on $n$ then shows that $\pi_{k}(\orbconf)$ is trivial for all $k\geq 2$. The first part of the statement then follows. By induction on $n$, \req{sesequi}, the isomorphism $G_{1}\cong \Z$ and the fact that $\pi_{1}(C \setminus\brak{\pm \widetilde{x}_{0},\ldots, \pm \widetilde{x}_{n-2}})\cong \F[2n-1]$, we obtain the isomorphism given in the second part of the statement.
\end{proof}

\begin{proof}[Proof of \reth{prop5}]
Let $n\geq 2$. We make use of the notation and basepoints of \resec{generalities}. 
By \reco{alphanprime}, 
$I_{n}$ and $I_{n}'$ have the same homotopy type. 

\begin{enumerate}[(a)]
\item Suppose first that $M=\St$.
We claim that the map $\altmap{\iota_{n}'}{F_{n-1}(\St\setminus\brak{w_n})}{\prod_{i=1}^{n-1}\, \St}$ is null homotopic. To see this, observe that for $1\leq i\leq n-1$, the following diagram is commutative:
\begin{equation*}
\begin{tikzcd}[cramped]
F_{n-1}(\St\setminus\brak{w_n}) \arrow{r}{\iota_{n}'} \arrow{d}{p_{i}} & \prod_1^{n-1}\,\St \arrow{d}{\widetilde{p}_{i}} \\
\St\setminus \brak{w_{n}} \arrow[hook]{r}{\iota_{2}'}  & \St,
\end{tikzcd}
\end{equation*}
where $p_{i}$ and $\widetilde{p}_{i}$ are the projections of \resec{mapiota}.
The contractibility of $\St\setminus \brak{w_{n}}$ implies that the composition $\iota_{2}'\circ p_{i}$ is null homotopic, so the composition $\widetilde{p}_{i}\circ \iota_{n}'$ is null homotopic for all $1\leq i\leq n-1$, and hence $\iota_{n}'$ is null homotopic, which proves the claim. The first part of \reth{prop5}(\ref{it:prop5a}) follows by applying \rerems{homfibre}(\ref{it:homfibreb}) to $\iota_{n}'$.
The equivalence with the second part is clear using the fact that $F_{n-1}(\dd)$ is an Eilenberg-Mac~Lane space of type $K(P_{n-1},1)$~\cite{FaN}.

\item Let $M=\rp$. 
Considering $F_{n-1}(\rp\setminus\brak{w_{n}})$ to be a subspace of $\prod_1^{n-1}\, \rp$ via the inclusion $\iota_{n}'$, one may check that $\orbconf[n-1]=p^{-1}(F_{n-1}(\rp\setminus\brak{w_{n}}))$, and that the map $\map{p'}{\orbconf[n-1]}[F_{n-1}(\rp\setminus \brak{w_n})]$ is also a regular $2^{n-1}$-fold covering map. Now consider the middle right-hand commutative square of~\reqref{bigcommdiagfib}.
We claim that $I_{n}'$ and $I_{n}''$ have the same homotopy type. To see this, first note that since $p$ is the universal covering map and $(\iota_{n}')_{\#1}$ is surjective,
it follows using standard properties of covering spaces that $\map{p'_{\#1}}{\pi_{1}\bigl(\orbconf[n-1] \bigr)}[\ker{(\iota_{n}')_{\#1}}]$ is an isomorphism.
Consider the homotopy fibrations:
\begin{equation}\label{eq:homofibInprime2}
\text{$I_{n}'' \stackrel{h_{n}''}{\to} \orbconf[n-1] \stackrel{\iota_{n}''}{\to} \prod_1^{n-1}\,\St$ and $I_{n}' \xrightarrow{g_{n}' \circ j_{n}'} F_{n-1}(\rp\setminus \brak{w_{n}}) \stackrel{\iota_{n}'}{\to} \prod_1^{n-1}\,\rp$}
\end{equation}
of the second and third rows of~\reqref{bigcommdiagfib},
and for $k\in \N$, let $\map{\partial_{n,k}''}{\pi_{k}(\prod_{i=1}^{n-1}\, \St)}[\pi_{k-1}(I_{n}'')]$ denote the boundary homomorphism corresponding to the first of these. 
Using the commutativity of these two rows, we obtain a map $\map{\xi}{I_{n}''}[I_{n}']$ of homotopy fibres given by $\xi(x,\lambda)=(p'(x), p\circ \lambda)$ for all $(x,\lambda)\in I_{n}''$ that satisfies $p'\circ h_{n}'' = (g_{n}'\circ j_{n}')\circ \xi$, and taking the long exact sequence in homotopy of the homotopy fibrations~\reqref{homofibInprime2}, we obtain the following commutative diagram of exact sequences:
\begin{equation}\label{eq:cdorbconfrp2a}
\begin{tikzcd}[cramped, sep=scriptsize]
\ldots \arrow{r} &[-0.6em] \pi_{k}(I_{n}'') \arrow{r}{(h_{n}'')_{\#k}} \arrow[swap]{d}{\xi_{\#k}} & \pi_{k}(\orbconf[n-1]) \arrow{r}{(\iota_{n}'')_{\#k}} \arrow{d}{p'_{\#k}} &[-0.2em]  \pi_{k}(\prod_1^{n-1}\,\St) \arrow{r}{\partial_{n,k}''} \arrow{d}{p_{\#k}} &[-0.6em] \pi_{k-1}(I_{n}'') \arrow{r} \arrow{d}{\xi_{\#(k-1)}} &[-0.6em] \ldots\\
\ldots \arrow{r} & \pi_{k}(I_{n}') \arrow[outer sep=4pt]{r}{(g_{n}'\circ j_{n}')_{\#k}} & \pi_{k}(F_{n-1}(\rp\setminus \brak{w_{n}})) \arrow[outer sep=2pt]{r}{(\iota_{n}')_{\#k}} & \pi_{k}(\prod_1^{n-1}\,\rp) \arrow{r}{\partial_{n,k}'} & \pi_{k-1}(I_{n}') \arrow{r} & \ldots
\end{tikzcd}
\end{equation}
Now $F_{n-1}(\rp\setminus\brak{w_n})$ is an Eilenberg-Mac~Lane space of type $K(\pi,1)$, and by \relem{equifn}, so is $\orbconf[n-1]$. Since the homomorphism $p_{\#k}$ 
is an isomorphism for all $k\geq 2$, it follows from~\reqref{cdorbconfrp2a} that $\xi_{\#k}$ 
is an isomorphism for all $k\geq 2$. Studying the tail of~\reqref{cdorbconfrp2a}, and using the surjectivity of $(\iota_{n}')_{\#1}$ and the fact that $\orbconf[n-1]$ and $F_{n-1}(\rp\setminus\brak{w_n})$ are path connected, we deduce that $\pi_{0}(I_{n}')$ and $\pi_{0}(I_{n}'')$ consist of a single point, so 
$\xi_{\#0}$ 
is a bijection. 
Finally, from the remaining part of the long exact sequence in homotopy, we have:
\begin{equation*}
\begin{tikzcd}[cramped, sep=scriptsize]
1 \arrow{r} &[-0.5em] \pi_{2}(\prod_1^{n-1}\,\St) \arrow{r}{\partial_{n,2}''} \arrow{d}{p_{\#2}}  & \pi_{2}(I_{n}'') \arrow{r}{(h_{n}'')_{\#1}} \arrow[swap]{d}{\xi_{\#1}} &[0.1em] \pi_{1}(\orbconf[n-1]) \arrow{r} \arrow{d}{p'_{\#1}} & 1  \arrow{d} &[-0.5em] {}\\
1 \arrow{r} & \pi_{2}(\prod_1^{n-1}\,\rp) \arrow{r}{\partial_{n,2}'}  \arrow{r} & \pi_{1}(I_{n}') \arrow[outer sep=4pt]{r}{(g_{n}'\circ j_{n}')_{\#1}} & \pi_{1}(F_{n-1}(\rp\setminus \brak{w_{n}})) \arrow{r}{(\iota_{n}')_{\#1}} & \pi_{1}(\prod_1^{n-1}\,\rp) \arrow{r} & 1,
\end{tikzcd}
\end{equation*}
and since $p'_{\#1}$
is an isomorphism, we obtain the following commutative diagram of short exact sequences:
\begin{equation*}
\begin{tikzcd}[column sep=large, cramped]
1 \arrow{r} &\pi_{2}(\prod_1^{n-1}\,\St) \arrow{r}{\partial_{n,2}''} \arrow{d}[swap]{\cong}{p_{\#2}}  & \pi_{2}(I_{n}'') \arrow{r}{(h_{n}'')_{\#1}} \arrow[swap]{d}{\xi_{\#1}} & \pi_{1}(\orbconf[n-1]) \arrow{r} \arrow{d}[swap]{\cong}{p'_{\#1}} & 1  \\
1 \arrow{r} & \pi_{2}(\prod_1^{n-1}\,\rp) \arrow{r}{\partial_{n,2}'}  \arrow{r} & \pi_{1}(I_{n}') \arrow{r}{(g_{n}'\circ j_{n}')_{\#1}} & \ker{(\iota_{n}')_{\#1}}   \arrow{r}& 1.
\end{tikzcd}
\end{equation*}
The $5$-Lemma implies that $\xi_{\#1}$ 
is an isomorphism. We conclude that $\xi_{\#k}$ 
is an isomorphism for all $k\geq 0$. Now by~\cite[pages~204--5]{A} or~\cite[page~247]{M}, the homotopy fibre $I_{f}$ of a map $\map{f}{X}[Y]$ between two topological spaces $X$ and $Y$ is the standard homotopy pullback of the spaces and maps $\ast \to Y \stackrel{f}{\longleftarrow} X$,
and by~\cite[Lemma~36]{M}, if $X$ and $Y$ each have the homotopy type of a CW-complex, then $I_{f}$ has the homotopy type of a CW-complex. Applying this to the maps $\iota_{n}'$ and $\iota_{n}''$, we conclude that $I_{n}'$ and $I_{n}''$ each have the homotopy type of a CW-complex.  
Whitehead's theorem then implies that they have the same homotopy type, which proves the claim.

It remains to determine the homotopy type of $I_{n}''$. Let $1\leq i\leq n-1$, and consider the following commutative diagram:
\begin{equation*}
\begin{tikzcd}[cramped]
\orbconf[n-1] \arrow{r}{\iota_{n}''} \arrow[swap]{d}{p_{i}\bigl\lvert_{\orbconf[n-1]}\bigr.} & \prod_1^{n-1}\,\St \arrow{d}{\widetilde{p}_{i}} \\
C \arrow[hook]{r}{\iota_{2}''}  & \St,
\end{tikzcd}
\end{equation*}
where $p_{i}$ and $\widetilde{p}_{i}$ are the projections of \resec{mapiota}. The composition $\iota_{2}''\circ p_{i}\bigl\lvert_{\orbconf[n-1]}\bigr.$
is null homotopic because $C$ has the homotopy type of a circle and $\St$ is simply connected, so the composition $\widetilde{p}_{i}\circ \iota_{n}''$ is null homotopic, from which it follows that $\iota_{n}''$ is also null homotopic. So $I_{n}''$ has the homotopy type of $\orbconf[n-1] \times \Omega\bigl(\prod_1^{n-1} \, \St\bigr)$ by \rerems{homfibre}(\ref{it:homfibreb}), which proves the first part of the statement of part~(\ref{it:prop5b}). The second part then follows  from \relem{equifn}.\qedhere
\end{enumerate}
\end{proof}

\begin{rem}\label{rem:In2null}
In the course of the proof of \reth{prop5}, we showed that:
\begin{enumerate}[(a)]
\item $I_{n}'$ and $I_{n}''$ have the same homotopy type.
\item $\iota_{n}''$ is null homotopic (recall that $\iota_{n}''=\iota_{n}'$ if $M=\St$). From this, it follows that:
\begin{enumerate}[(i)]
\item the upper right-hand square of the diagram~\reqref{bigcommdiagfib} commutes up to homotopy, and so the same is true for the whole diagram. 
\item $I_{n}''$ and $I_{c}$ have the same homotopy type by~\cite[Proposition~3.3.17]{A}. 
\end{enumerate}
\end{enumerate}
In \relem{fundhomoequiv}, we shall exhibit explicit homotopy equivalences between $I_{c}$, $I_{n}'$ and $I_{n}''$.
\end{rem}


We now prove \reco{lesfibs2}.


\begin{proof}[Proof of \reco{lesfibs2}]
We make use of the notation of \resec{generalities}. Let $M=\St$ or $\rp$, and let $n\geq 2$ and $k\geq 1$. Taking the long exact sequence in homotopy of the lower fibration of~\reqref{commdiagfib}, and using the fact that $\widehat{\iota}_{n}\circ f_{n}=\iota_{n}$, we obtain the following long exact sequence:
\begin{equation}\label{eq:lesiota}
\ldots \to \pi_{k}(I_{n}) \xrightarrow{(g_{n}\circ j_{n})_{\#k}} \pi_{k}(F_{n}(M)) \xrightarrow{(\iota_{n})_{\#k}} \pi_k\left(\prod_1^n \, M\right) \xrightarrow{\partial_{n,k}} \pi_{k-1}(I_{n}) \to \ldots. 
\end{equation}

\begin{enumerate}[(a)] 
\item Suppose that $k\geq 3$ and $n\geq 2$ (resp.\ $M=\St$ and $n=k=2$). Then $(\iota_{n})_{\#k}$ and $(\iota_{n})_{\#(k-1)}$ are injective by \relem{iotainject}(\ref{it:iotainjectb}) (resp.\ $(\iota_{2})_{\#2}$ is injective by \relem{iotainject}(\ref{it:iotainjectb}) and $P_{2}(\St)$ is trivial), from which we obtain the short exact sequence~\reqref{splitrp2}, where $\pi_{k-1}(I_{n})$ has been replaced by $\pi_{k-1}\bigl(\Omega\bigl(\prod_1^{n-1} \, \St\bigr)\bigl)$ via the homotopy equivalences of \reth{prop5}. We have also used the fact that $F_{n-1}(\dt)$ and $\orbconf[n-1]$ are Eilenberg-Mac~Lane spaces of type $K(\pi,1)$ by~\cite{FaN} and \relem{equifn} (resp.\ $P_{1}(\dt)$ is trivial). \repr{diagS2}(\ref{it:diagb}) and \repr{diagRP2} (resp.\ \repr{diagS2}(\ref{it:diaga})) show that $(\iota_{n})_{\#k}$ is diagonal (resp.\ anti-diagonal). Further, if $i\in \brak{1,\ldots,n}$, by~\reqref{compproj} and \relem{iotainject}(\ref{it:iotainjecta}),  $\map{(\widetilde{p}_{i})_{\#k}\circ (\iota_{n})_{\#k}}{\pi_{k}(F_{n}(M))}[\pi_{k}(M)]$ is an isomorphism, and taking $A=\pi_{k}(F_{n}(M))$, $B=\pi_{k}(M)$, $\Theta=(\iota_{n})_{\#k}$ and $\pi_{i}=(\widetilde{p}_{i})_{\#k}$ in \relem{diagonal}, it follows that the short exact sequence~\reqref{splitrp2} splits. This proves the first part of the statement in each case. The application of \reth{homotopytypefns} and standard facts about homotopy groups yields the second short exact sequence of the statement of part~(\ref{it:lesfibs2a}).
  
\item Let $k=2$ and $n\geq 3$. Since $\pi_2(F_n(M))$ is trivial by \reth{homotopytypefns}, we obtain the exact sequence~\reqref{sespi1rp2} by applying once more \reth{prop5} to \req{lesiota}. If $M=\St$, the sequence becomes short exact, but does not split because the middle group is torsion free and the quotient has torsion. The remaining exact sequences in the statement then follow using standard results about homotopy groups.\qedhere 
\end{enumerate}
\end{proof}

\section{A presentation of  $G_{n}$}\label{sec:presorb}

Let $n\geq 1$. Once more, we make use of the notation of \resec{generalities}.
In this section, we shall exhibit a presentation of the fundamental group $G_{n}$ of the orbit configuration space $\orbconf$. We identify $C$ with the open annulus $(0,1)\times \St[1]$ in the complex plane $\mathbb{C}$, where we identify $\widetilde{z}_{0}$ with the origin in $\mathbb{C}$,  and up to this identification, if $z=re^{i\theta}\in C$ then we take its antipodal point $z'$ to be $(1-r)e^{i(\theta+\pi)}$. Let 
$v_{1},\ldots,v_{n}$ be base points in $C$ that lie on the negative part of the $x$-axis, and let $v_{1}',\ldots,v_{n}'$ be the corresponding antipodal points (see Figure~\ref{fig:gensorb}).
\begin{figure}[t]
\hfill
\begin{tikzpicture}[scale=0.7]
\draw[line width=1.5pt,dotted] (0,0) circle(6.8);
\draw[line width=1.3pt,densely dotted] (0,0) circle(0.15);

\draw[very thick, color=gray, decoration={markings,mark=at position 0.5 with {\arrow{stealth'}}},postaction={decorate}] (-3.1,0)  .. controls (-1,2.7) and (0.7,-0.3)  .. (1.15,-0.3);
\draw[very thick, color=gray] (-3.1,0)  .. controls (-1,4.2) and (2,0.4)  .. (1.15,-0.3);

\draw[very thick, color=gray, decoration={markings,mark=at position 0.5 with {\arrow{stealth'}}},postaction={decorate}] (-3.1,0)  .. controls (-0.5,-1) and (0.2,-0.5)  .. (0.3,0);
\draw[very thick, color=gray] (-3.1,0)  .. controls (-0.5,1) and (0.2,0.5)  .. (0.3,0);

\draw[very thick, color=gray, decoration={markings,mark=at position 0.5 with {\arrow{stealth'}}},postaction={decorate}] (-3.1,0)  .. controls (-1,5) and (2,0)  .. (3.15,-0.3);
\draw[very thick, color=gray] (-3.1,0)  .. controls (-1,6) and (4.5,0)  .. (3.15,-0.3);

\draw[very thick, color=gray, decoration={markings,mark=at position 0.5 with {\arrow{stealth'}}},postaction={decorate}] (-3.1,0)  .. controls (-5,2.5) and (-5.9,0.1)  .. (-5.1,-0.3);
\draw[very thick, color=gray] (-3.1,0)  .. controls (-4,1) and (-4.5,-0.1)  .. (-5.1,-0.3);

\foreach \k in {-6.1,-5.1,...,-1.1, 1.1, 2.1,...,6.1}
{\draw[fill] (\k,0) circle [radius=0.08];};

\node at (-5.65,0) {$\cdots$};
\node at (-3.55,0) {$\cdots$};
\node at (1.7,0) {$\cdots$};
\node at (4.65,0) {$\cdots$};
\node at (-1.55,0) {$\cdots$};

\node at (0,1.2) {$\rho_{j,j}$};
\node at (3,2) {$\rho_{j,2j-2}$};
\node at (-4.9,1.5) {$\rho_{j,i}$};
\node at (0,-0.9) {$\rho_{j,0}$};

\node at (-3, -1.2) {$v_{j}$};
\node at (-5,-1.2) {$v_{i}$};
\node at (-6,-1.2) {$v_{1}$};
\node at (-1,-1.2) {$v_{n}$};
\node at (3.15, -1.2) {$v_{j-1}'$};
\node at (6.15,-1.2) {$v_{n}'$};
\node at (1.15,-1.2) {$v_{1}'$};
\end{tikzpicture}
\hspace*{\fill}
\caption{Generators of $\pi_{1}\bigl(C \setminus\bigl\{v_{1},v_{1}',\ldots, v_{j-1},v_{j-1}'\bigr\}, v_{j}\bigr)$.}
\label{fig:gensorb}
\end{figure}
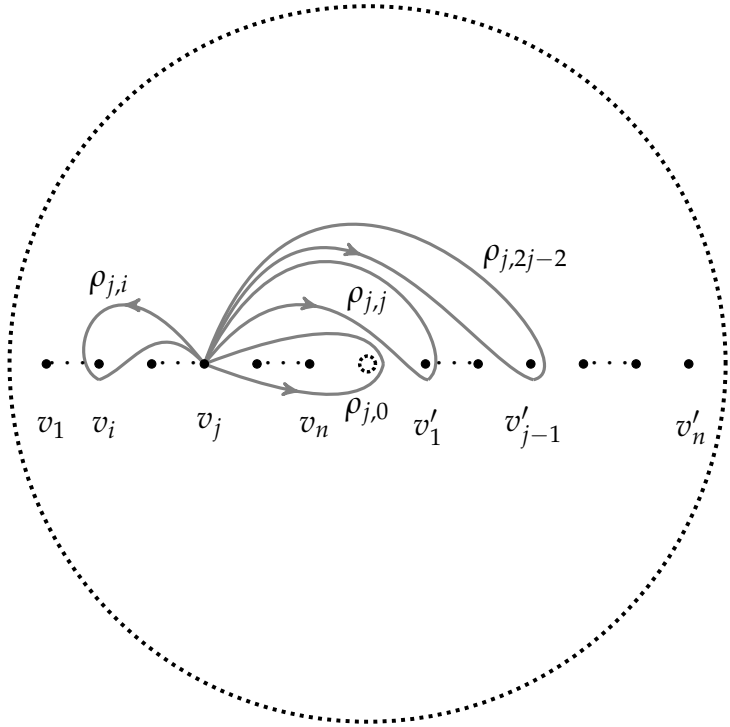
As we saw in the proof of \relem{equifn}, the map $\map{q_{n+1}}{\orbconf[n+1]}[\orbconf]$ given by forgetting the last point is a locally-trivial fibration that admits a section. The section is not unique: as in the proof of \relem{equifn}, it will be convenient in what follows to define $\map{s_{n+1}}{\orbconf}[{\orbconf[n+1]}]$ by $s_{n+1}(y_{1},\ldots,y_{n})= (y_{1},\ldots,y_{n},y_{n+1})$, where $y_{n+1}$ is a point sufficiently close to $\widetilde{z}_{0}$ 
and lying on the negative $x$-axis (see for example the proof of~\cite[Theorem~6]{GG1}). Then the corresponding antipodal point $y_{n+1}'$ lies on the positive $x$-axis close to the outer boundary of $C$, and the induced homomorphism $\map{(s_{n+1})_{\#}}{G_{n}}[G_{n+1}]$ is a section for $(q_{n+1})_{\#}$. Another useful choice of section will be described in \rerem{secondsec}.
%

If $j\in \N$, the free group $\pi_{1}\bigl(C \setminus\bigl\{v_{1},v_{1}',\ldots, v_{j-1},v_{j-1}'\bigr\}, v_{j}\bigr)$ of rank $2j-1$ admits a basis $\setl{\rho_{j,i}}{0\leq i\leq 2j-2}$, where for $0\leq i\leq 2j-2$, we represent the generators $\rho_{j,i}$ by the loops depicted in Figure~\ref{fig:gensorb}. If $j\leq i\leq 2j-2$ then $\rho_{j,i}$ is represented  by a loop based at $v_{j}$ that wraps round the point $v_{i-j+1}'$. If $j=n$ then via the identification of the fibre of $q_{n+1}$ over $(v_{1},\ldots,v_{n})$ with $C \setminus\bigl\{v_{1},v_{1}',\ldots, v_{n},v_{n}'\bigr\}$, this group is isomorphic to $\ker{(q_{n+1})_{\#}}$. 
However, we have to take into account the fact that we are working in the orbit configuration space. If $(y_{1},\ldots,y_{n+1})\in \orbconf[n+1]$ then for all $1\leq i<j\leq n+1$, $y_{i}$ should not only be different from $y_{j}$, but it must also avoid $y_{j}'$. We thus interpret the fibre as an `honest' subspace of $\orbconf[n+1]$, and for each $0\leq i\leq 2j-2$, we will represent the element $\rho_{j,i}$ by a pair of antipodal loops (see Figures~\ref{fig:rhoj0}--\ref{fig:rhojigreat}).
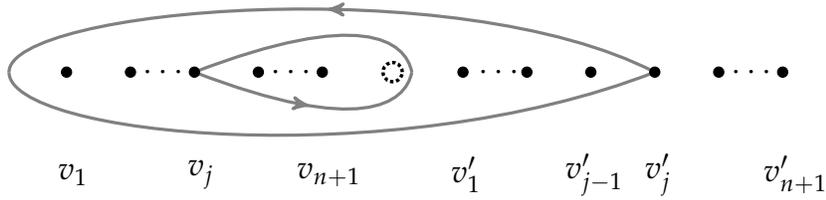
\begin{figure}[!pt]
\hfill
\begin{tikzpicture}[scale=0.85]
\draw[line width=1.3pt,densely dotted] (0,0) circle(0.15);

\draw[very thick, color=gray, decoration={markings,mark=at position 0.5 with {\arrow{stealth'}}},postaction={decorate}] (-3.1,0)  .. controls (-0.5,-1) and (0.2,-0.5)  .. (0.3,0);
\draw[very thick, color=gray] (-3.1,0)  .. controls (-0.5,1) and (0.2,0.5)  .. (0.3,0);

\draw[very thick, color=gray, decoration={markings,mark=at position 0.5 with {\arrow{stealth'}}},postaction={decorate}] (4.1,0)  .. controls (1,1.6) and (-6,1)  .. (-6,0);
\draw[very thick, color=gray] (4.1,0)  .. controls (0,-1.6) and (-6,-1)  .. (-6,0);

\foreach \k in {-5.1,-4.1,...,-1.1, 1.1, 2.1,...,6.1}
{\draw[fill] (\k,0) circle [radius=0.08];};

\node at (-3.55,0) {$\cdots$};
\node at (1.7,0) {$\cdots$};
\node at (5.65,0) {$\cdots$};
\node at (-1.55,0) {$\cdots$};

\node at (-3, -1.6) {$v_{j}$};

\node at (-5,-1.6) {$v_{1}$};
\node at (-1,-1.6) {$v_{n+1}$};
\node at (3.15, -1.6) {$v_{j-1}'$};
\node at (4.15, -1.6) {$v_{j}'$};
\node at (6.3,-1.6) {$v_{n+1}'$};
\node at (1.15,-1.6) {$v_{1}'$};
\end{tikzpicture}
\hspace*{\fill}
\caption{The generator $\rho_{j,0}$.}
\label{fig:rhoj0}
\end{figure}
\begin{figure}[!pt]
\hfill
\begin{tikzpicture}[scale=0.85]
\draw[line width=1.3pt,densely dotted] (0,0) circle(0.15);

\draw[very thick, color=gray, decoration={markings,mark=at position 0.5 with {\arrow{stealth'}}},postaction={decorate}] (-3.1,0)  .. controls (-4,1.5) and (-6,1)  .. (-5.2,-0.3);

\draw[very thick, color=gray] (-3.1,0)  .. controls (-5.2,1.5) and (-4,-1)  .. (-5.2,-0.3);

\draw[very thick, color=gray, decoration={markings,mark=at position 0.5 with {\arrow{stealth'}}},postaction={decorate}] (4.1,0)  .. controls (3,-1.5) and (1.2,-1)  .. (2,0.3);

\draw[very thick, color=gray] (4.1,0)  .. controls (2,-1.5) and (3,1)  .. (2,0.3);

\foreach \k in {-6.1,-5.1,...,-1.1, 1.1, 2.1,...,6.1}
{\draw[fill] (\k,0) circle [radius=0.08];};

\node at (-5.65,0) {$\cdots$};
\node at (-3.55,0) {$\cdots$};
\node at (1.55,0) {$\cdots$};
\node at (5.65,0) {$\cdots$};
\node at (-1.55,0) {$\cdots$};
\node at (3.65,0) {$\cdots$};

\node at (-3, -1.5) {$v_{j}$};
\node at (-5,-1.5) {$v_{i}$};
\node at (-6,-1.5) {$v_{1}$};
\node at (-1,-1.5) {$v_{n+1}$};
\node at (2.15, -1.5) {$v_{i}'$};
\node at (4.15, -1.5) {$v_{j}'$};
\node at (6.4,-1.5) {$v_{n+1}'$};
\node at (1.15,-1.5) {$v_{1}'$};
\end{tikzpicture}
\hspace*{\fill}
\caption{The generator $\rho_{j,i}$, $1\leq i<j$.}
\label{fig:rhoji}
\end{figure}
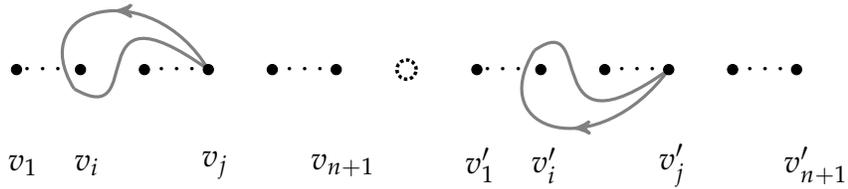
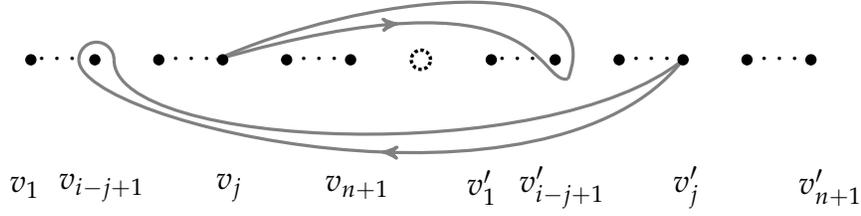
\begin{figure}[!pt]
\hfill
\begin{tikzpicture}[scale=0.85]
\draw[line width=1.3pt,densely dotted] (0,0) circle(0.15);

\draw[very thick, color=gray, decoration={markings,mark=at position 0.5 with {\arrow{stealth'}}},postaction={decorate}] (4.1,0)  .. controls (2,-2.7) and (-5.35,-1)  .. (-5.35,0);

\draw[very thick, color=gray] (-4.8,0)  .. controls (-4.8,-1.4) and (2,-1.7)  .. (4.1,0);

\draw[very thick, color=gray] (2.3,-0.3)  .. controls (3,1.5) and (-1,1)  .. (-3.2,0);

\draw[very thick, color=gray, decoration={markings,mark=at position 0.5 with {\arrow{stealth'}}},postaction={decorate}] 
(-3.2,0) .. controls (2,1.5) and (1.8,-0.4) .. (2.3,-0.3);


\draw[very thick, color=gray] (-4.8,0) arc [start angle=0, end angle=180, radius=0.275];

\foreach \k in {-6.1,-5.1,...,-1.1, 1.1, 2.1,...,6.1}
{\draw[fill] (\k,0) circle [radius=0.08];};

\node at (-5.65,0) {$\cdots$};
\node at (-3.55,0) {$\cdots$};
\node at (1.55,0) {$\cdots$};
\node at (3.65,0) {$\cdots$};
\node at (5.65,0) {$\cdots$};
\node at (-1.55,0) {$\cdots$};

\node at (-3, -2) {$v_{j}$};
\node at (-5,-2) {$v_{i-j+1}$};
\node at (-6.2,-2) {$v_{1}$};
\node at (-1,-2) {$v_{n+1}$};
\node at (2.2, -2) {$v_{i-j+1}'$};
\node at (4.15, -2) {$v_{j}'$};
\node at (6.4,-2) {$v_{n+1}'$};
\node at (0.95,-2) {$v_{1}'$};
\end{tikzpicture}
\hspace*{\fill}
\caption{The generator $\rho_{j,i}$, $j\leq i\leq 2j-2$.}
\label{fig:rhojigreat}
\end{figure}
With our choice of section, if $j\leq n$, the image of $\rho_{j,i}$ ($\rho_{j,i}$ being considered as an element of $G_{j}$) by the homomorphism $(s_{n}\circ \cdots \circ s_{j+1})_{\#}$ is the element $\rho_{j,i}$ of $G_{n}$. Using induction on $n$, it follows from this and the fact that $G_{n}$ is isomorphic to the internal semi-direct product $\ker{(q_{n})_{\#}}\rtimes (s_{n})_{\#}(G_{n-1})$ that the group $G_{n}$ is generated by the set $\bigcup_{j=1}^{n} \setl{\rho_{j,i}}{0\leq i\leq 2j-2}$. We will shortly give a presentation of $G_{n}$, but we first define several elements that will serve to simplify the statement: 
\begin{enumerate}[(a)]
\item for all $1\leq j<k\leq n$, $D_{j,k}=\rho_{k,j}\rho_{k,j+1}\cdots \rho_{k,k-1}$ (see Figure~\ref{fig:Djk}).

\item for all $1\leq j<k\leq n$, $C_{k,j}=\rho_{k,0}^{-1}D_{j+1,k}^{-1} \rho_{k,j} D_{j+1,k}\rho_{k,0}$ ($D_{j+1,k}^{-1} \rho_{k,j} D_{j+1,k}$ is the mirror image of $\rho_{k,j}^{-1}$ with respect to the horizontal axis, see Figure~\ref{fig:Ckj}).

\item for all $1\leq k\leq n$ and $k\leq m < q\leq 2k-2$, $E_{k,m,q}=\rho_{k,m}\rho_{k,m+1}\cdots \rho_{k,q}$. The element $E_{k,m,q}$ is represented by the loop (and its antipode) based at $v_{k}$ that winds successively around the points $v_{m-k+1}', v_{m-k+2}', \ldots, v_{q-k+1}'$ (see Figure~\ref{fig:Ekmq}).
\end{enumerate}
\begin{figure}[!th]
\hfill
\begin{tikzpicture}[scale=0.85]
\draw[line width=1.3pt,densely dotted] (0,0) circle(0.15);

\draw[very thick, color=gray, decoration={markings,mark=at position 0.5 with {\arrow{stealth'}}},postaction={decorate}] (-3.1,0)  .. controls (-3.8,0.8) and (-5.3,0.5)  .. (-5.3,0);
\draw[very thick, color=gray] (-3.1,0)  .. controls (-3.8,-0.5) and (-5.3,-0.6)  .. (-5.3,0);

\draw[very thick, color=gray] (4.1,0)  .. controls (3.2,0.8) and (1.9,0.5)  .. (1.9,0);
\draw[very thick, color=gray, decoration={markings,mark=at position 0.5 with {\arrow{stealth'}}},postaction={decorate}] (4.1,0)  .. controls (3.2,-0.5) and (1.9,-0.6)  .. (1.9,0);

\foreach \k in {-6.1,-5.1,...,-1.1, 1.1, 2.1,...,6.1}
{\draw[fill] (\k,0) circle [radius=0.08];};

\node at (-5.65,0) {$\cdots$};
\node at (-3.55,0) {$\cdots$};
\node at (1.55,0) {$\cdots$};
\node at (5.65,0) {$\cdots$};
\node at (-1.55,0) {$\cdots$};
\node at (3.65,0) {$\cdots$};

\node at (-3, -1) {$v_{k}$};
\node at (-5,-1) {$v_{j}$};
\node at (-6,-1) {$v_{1}$};
\node at (-1,-1) {$v_{n}$};
\node at (2.15, -1) {$v_{j}'$};
\node at (4.15, -1) {$v_{k}'$};
\node at (6.4,-1) {$v_{n}'$};
\node at (1.15,-1) {$v_{1}'$};
\end{tikzpicture}
\hspace*{\fill}
\caption{The element $D_{j,k}$, $1\leq j<k\leq n$.}
\label{fig:Djk}
\end{figure}
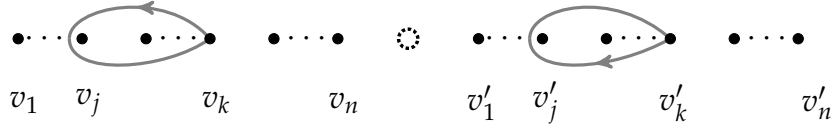
\begin{figure}[!th]
\hfill
\begin{tikzpicture}[scale=0.85]
\draw[line width=1.3pt,densely dotted] (0,0) circle(0.15);

\draw[very thick, color=gray, decoration={markings,mark=at position 0.5 with {\arrow{stealth'}}},postaction={decorate}] (-3.1,0)  .. controls (-3.1,0.5) and (0.6,1)  .. (0.4,0);
\draw[very thick, color=gray] (0.4,0)  .. controls (0.4,-0.7) and (-4.8,-0.5)  .. (-4.8,0);
\draw[very thick, color=gray] (-5.3,0)  .. controls (-5.3,0.4) and (-4.8,0.4)  .. (-4.8,0);
\draw[very thick, color=gray] (0.7,0)  .. controls (0.7,-1) and (-5.3,-0.7)  .. (-5.3,0);
\draw[very thick, color=gray] (-3.1,0)  .. controls (-3.1,1) and (0.7,1)  .. (0.7,0);

\draw[very thick, color=gray, decoration={markings,mark=at position 0.5 with {\arrow{stealth'}}},postaction={decorate}](4,0) .. controls (4,-1.4) and (-6.6,-1.4) .. (-6.6,0);
\draw[very thick, color=gray](4,0) .. controls (4,-1.1) and (-6.4,-1.1) .. (-6.4,0);

\draw[very thick, color=gray](1.9,0) .. controls (1.9,1.35) and (-6.4,1.35) .. (-6.4,0);
\draw[very thick, color=gray](2.3,0) .. controls (2.3,1.5) and (-6.6,1.5) .. (-6.6,0);
\draw[very thick, color=gray](2.3,0) .. controls (2.3,-0.4) and (1.9,-0.4) .. (1.9,0);

\foreach \k in {-6.1,-5.1,...,-1.1, 1.1, 2.1,...,6.1}
{\draw[fill] (\k,0) circle [radius=0.08];};

\node at (-5.65,0) {$\cdots$};
\node at (-3.55,0) {$\cdots$};
\node at (1.55,0) {$\cdots$};
\node at (5.65,0) {$\cdots$};
\node at (-1.55,0) {$\cdots$};
\node at (3.65,0) {$\cdots$};

\node at (-3, -1.5) {$v_{k}$};
\node at (-5,-1.5) {$v_{j}$};
\node at (-6,-1.5) {$v_{1}$};
\node at (-1,-1.5) {$v_{n}$};
\node at (2.15, -1.5) {$v_{j}'$};
\node at (4.15, -1.5) {$v_{k}'$};
\node at (6.4,-1.5) {$v_{n}'$};
\node at (1.15,-1.5) {$v_{1}'$};
\end{tikzpicture}
\hspace*{\fill}
\caption{The element $C_{k,j}$, $1\leq j<k\leq n$.}
\label{fig:Ckj}
\end{figure}
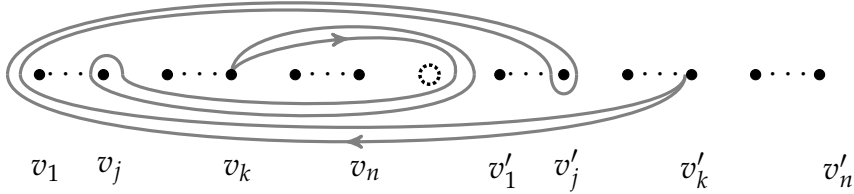
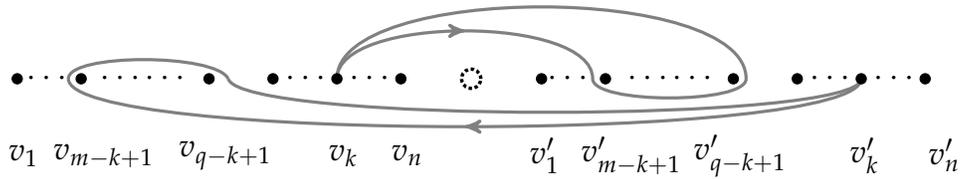
\begin{figure}[!th]
\hfill
\begin{tikzpicture}[scale=0.85]
\draw[line width=1.3pt,densely dotted] (0,0) circle(0.15);

\draw[very thick, color=gray, decoration={markings,mark=at position 0.5 with {\arrow{stealth'}}},postaction={decorate}](6.1,0) .. controls (6.1,-1) and (-6.3,-1) .. (-6.3,0);
\draw[very thick, color=gray](6.1,0) .. controls (6.1,-0.7) and (-3.6,-0.7) .. (-3.8,0);
\draw[very thick, color=gray](-3.8,0) .. controls (-4,0.4) and (-6.3,0.4) .. (-6.3,0);

\draw[very thick, color=gray, decoration={markings,mark=at position 0.5 with {\arrow{stealth'}}},postaction={decorate}](-2.1,0) .. controls (-2.1,1)  and 
(1.9,1) .. (1.9,0);
\draw[very thick, color=gray](4.3,0) .. controls (4.3,1.5) and (-2.1,1.5) .. (-2.1,0);
\draw[very thick, color=gray](1.9,0) .. controls (1.9,-0.4) and (4.3,-0.4) .. (4.3,0);

\foreach \k in {-7.1,-6.1,-4.1,-3.1,-2.1,-1.1, 1.1, 2.1,4.1,5.1,6.1,7.1}
{\draw[fill] (\k,0) circle [radius=0.08];};

\node at (-6.6,0) {$\cdots$};
\node at (-5.45,0) {$\cdots$};
\node at (-4.75,0) {$\cdots$};
\node at (1.55,0) {$\cdots$};
\node at (5.65,0) {$\cdots$};
\node at (6.65,0) {$\cdots$};
\node at (-1.55,0) {$\cdots$};
\node at (-2.55,0) {$\cdots$};
\node at (2.8,0) {$\cdots$};
\node at (3.5,0) {$\cdots$};

\node at (-2, -1.2) {$v_{k}$};
\node at (-5.75,-1.2) {$v_{m-k+1}$};
\node at (-3.85,-1.2) {$v_{q-k+1}$};
\node at (-7,-1.2) {$v_{1}$};
\node at (-1,-1.2) {$v_{n}$};
\node at (2.5, -1.2) {$v_{m-k+1}'$};
\node at (4.2, -1.2) {$v_{q-k+1}'$};
\node at (6.15, -1.2) {$v_{k}'$};
\node at (7.4,-1.2) {$v_{n}'$};
\node at (1.15,-1.2) {$v_{1}'$};
\end{tikzpicture}
\hspace*{\fill}
\caption{The element $E_{k,m,q}$, $1\leq k\leq n$ and $k\leq m < q\leq 2k-2$.}
\label{fig:Ekmq}
\end{figure}

\begin{prop}\label{prop:presgn}
Let $n\in \N$. The following constitutes a presentation of the group $G_{n}$:
\begin{enumerate}
\item[\textbf{generating set:}] $\bigcup_{j=1}^{n} \setl{\rho_{j,i}}{0\leq i\leq 2j-2}$.
\item[\textbf{relations:}] let $1\leq j<k\leq n$. Then: 
\begin{enumerate}[(I)]
\item\label{it:presI}
\begin{equation*}
\rho_{j,0}\rho_{k,l}\rho_{j,0}^{-1}= \begin{cases}
\rho_{k,l} & \text{if $l=0$ or $j<l<k$}\\
C_{k,j}  \rho_{k,l}C_{k,j}^{-1} &\text{if $1\leq l<j$}\\
& \text{or if $k\leq l\leq k+j-2$}\\
C_{k,j}& \text{if $l=j$}\\
C_{k,j}E_{k,k,k+j-2}^{-1} \rho_{k,0}^{-1} D_{1,k}^{-1} \rho_{k,l} D_{1,k} \rho_{k,0} E_{k,k,k+j-2} C_{k,j}^{-1}& \text{if $l=k+j-1$}\\
C_{k,j}\rho_{k,k+j-1} \rho_{k,l} \rho_{k,k+j-1}^{-1} C_{k,j}^{-1}& \text{if $k+j\leq l\leq 2k-2$.}
\end{cases}
\end{equation*}

\item\label{it:presII} for all $1\leq i<j$, $\rho_{j,i}\rho_{k,l}\rho_{j,i}^{-1}= \rho_{k,l}$ if $0\leq l\leq i-1$, $j+1\leq l\leq k+i-2$, $k+i\leq l\leq k+j-2$ or $k+j\leq l\leq 2k-2$, and
\begin{equation*}
\!\rho_{j,i}\rho_{k,l}\rho_{j,i}^{-1}= \begin{cases}
\rho_{k,j}^{-1} \rho_{k,i} \rho_{k,j} & \!\!\text{if $l=i$}\\
[\rho_{k,j}^{-1}, \rho_{k,i}^{-1}] \rho_{k,l} [\rho_{k,i}^{-1}, \rho_{k,j}^{-1}] & \!\!\text{if $i<l<j$}\\
\rho_{k,j}^{-1} \rho_{k,i}^{-1} \rho_{k,j} \rho_{k,i} \rho_{k,j} & \!\!\text{if $l=j$}\\
E_{k,k+i-1,k+j-1} E_{k,k+i,k+j-2}^{-1} \rho_{k,k+i-1} E_{k,k+i,k+j-2} E_{k,k+i-1,k+j-1}^{-1} & \!\!\text{if $l=k+i-1$}\\
E_{k,k+i,k+j-2}^{-1} E_{k,k+i-1,k+j-2} \rho_{k,k+j-1} E_{k,k+i-1,k+j-2}^{-1} E_{k,k+i,k+j-2} & \!\!\text{if $l=k+j-1$}
\end{cases}
\end{equation*}
in the remaining cases.

\item\label{it:presIII} for all $j\leq i\leq 2j-2$, $\rho_{j,i}\rho_{k,l}\rho_{j,i}^{-1}=[\rho_{k,j}^{-1}, \rho_{k,k+i-j}^{-1}] \rho_{k,l} [\rho_{k,k+i-j}^{-1},\rho_{k,j}^{-1}]$ if $1\leq l\leq i-j$, $i-j+2\leq l\leq j-1$, $k+i-j+1\leq l\leq k+j-2$ or $k+j\leq l\leq 2k-2$,
\begin{align*}
\rho_{j,i}\rho_{k,i-j+1}\rho_{j,i}^{-1}=& [\rho_{k,j}^{-1}, \rho_{k,k+i-j}^{-1}] D_{i-j+1,k} \rho_{k,0} E_{k,k,k+j-1} E_{k,k,k+j-2}^{-1}  C_{k,i-j+1}E_{k,k,k+j-2}\cdot\\
& E_{k,k,k+j-1}^{-1} \rho_{k,0}^{-1} D_{i-j+1,k}^{-1} [\rho_{k,k+i-j}^{-1},\rho_{k,j}^{-1}],
\end{align*}
which is the case $l=i-j+1$,
\begin{align*}
\rho_{j,i}\rho_{k,k+j-1}\rho_{j,i}^{-1}=& [\rho_{k,j}^{-1}, \rho_{k,i-j+1}^{-1}] E_{k,k,k+j-2}^{-1} C_{k,i-j+1}E_{k,k,k+j-2} \rho_{k,k+j-1} E_{k,k,k+j-2}^{-1} C_{k,i-j+1}^{-1}\cdot\\
& E_{k,k,k+j-2} [\rho_{k,i-j+1}^{-1},\rho_{k,j}^{-1}],
\end{align*}
which is the case $l=k+j-1$, and
\begin{equation*}
\rho_{j,i}\rho_{k,l}\rho_{j,i}^{-1}= \begin{cases}
\rho_{k,l} & \text{if $l=0$ or $j+1\leq l\leq k+i-j-1$}\\ 
%
%
\rho_{k,j}^{-1} \rho_{k,k+i-j}^{-1} \rho_{k,j} \rho_{k,k+i-j} \rho_{k,j} & \text{if $l=j$}\\
\rho_{k,j}^{-1} \rho_{k,k+i-j} \rho_{k,j} & \text{if $l=k+i-j$,}
\end{cases}
\end{equation*}
in the remaining cases.
%
%
%
%
\end{enumerate} 
\end{enumerate}
\end{prop}

\begin{proof}
The proof is by induction on $n$. If $n=1$ then $G_{1}=\ang{\rho_{1,0}}\cong \Z$, and the given presentation is valid. So suppose that the result holds for some $n\geq 1$. As we saw above, $G_{n+1}$ is generated by $\bigcup_{j=1}^{n+1} \setl{\rho_{j,i}}{0\leq i\leq 2j-2}$, and applying standard techniques to obtain a presentation of the group extension $G_{n+1}$ using the short exact sequence of the form~\reqref{sesequi} whose kernel is a free group~\cite[Proposition~1, page~139]{Jo}, the defining relations are of two types:
\begin{enumerate}[(i)]
\item\label{it:relns1} the images under $(s_{n+1})_{\#}$ of the relators of the presentation of $G_{n}$, rewritten as elements of $\ker{(p_{n+1})_{\#}}$.
\item\label{it:relns2} the elements of the form $yxy^{-1}$, rewritten as elements of $\ker{(p_{n+1})_{\#}}$, where 
\begin{equation*}
\text{$x\in \setl{\rho_{n+1,l}}{0\leq l\leq 2n}$ and $y\in \bigcup_{j=1}^{n} \setl{\rho_{j,i}}{0\leq i\leq 2j-2}$.}
\end{equation*}
\end{enumerate}
Since $(s_{n+1})_{\#}(\rho_{j,i})=\rho_{j,i}$ for all $1\leq j\leq n$ and $0\leq i\leq 2j-2$, the relators of $G_{n}$ are sent to relators of $G_{n+1}$ under $(s_{n+1})_{\#}$, and so the relations of~(\ref{it:relns1}) above are those of $G_{n}$, but considered as relations in $G_{n+1}$. This gives rise to all of the relations~(\ref{it:presI})--(\ref{it:presIII}) for $G_{n+1}$ with $1\leq j<k\leq n$. It remains to analyse the relations emanating from~(\ref{it:relns2}), which correspond to the relations~(\ref{it:presI})--(\ref{it:presIII}) for $G_{n+1}$ with $1\leq j\leq n$ and $k=n+1$. We obtain the following types of relation:
\begin{enumerate}[(I)]
\item conjugation of $\rho_{n+1,l}$ by the generator $\rho_{j,0}$ of Figure~\ref{fig:rhoj0}, for $1\leq j\leq n$ and $0\leq l\leq 2n$.
\item conjugation of $\rho_{n+1,l}$ by the generator $\rho_{j,i}$ of Figure~\ref{fig:rhoji}, for $1\leq i< j\leq n$ and $0\leq l\leq 2n$.
\item conjugation of $\rho_{n+1,l}$ by the generator $\rho_{j,i}$ of Figure~\ref{fig:rhojigreat}, for $1\leq j\leq n$ and $j\leq i\leq 2j-2$ and $0\leq l\leq 2n$.
\end{enumerate}
Using these figures, for all $0\leq l\leq 2n$ and $0\leq i\leq 2j-2$,  the conjugate of $\rho_{n+1,l}$ by $\rho_{j,i}$ may be rewritten as an element of $\ker{(p_{n+1})_{\#}}$. Although $\rho_{n+1,l}$ is also represented by two loops, the intersections between the two pairs of loops are symmetric, and geometrically it suffices to consider (and remove) the intersections between the two representative loops of $\rho_{j,i}$ with one representative loop of $\rho_{n+1,l}$. The verifications are left to the reader.
\end{proof}

If $n\in \N$, let $\Theta_{n}=\rho_{1,0}\cdots \rho_{n,0}$. We now show that the centre $Z(G_{n})$ of $G_{n}$ is infinite cyclic, generated by $\Theta_{n}$.

\begin{prop}\label{prop:gncentre}
Let $n\in \N$. Then $Z(G_{n})=\ang{\Theta_{n}}$.
\end{prop}

\begin{proof}
Let $n\in \N$. In what follows, (\ref{it:presI})--(\ref{it:presIII}) refer to the corresponding relations of \repr{presgn}. Let $\Theta=\Theta_{n}$. We start by showing that $\Theta$ commutes with each generator $\rho_{k,l}$ of $G_{n}$, where $1\leq k\leq n$ and $0\leq l\leq 2k-2$, which will imply that $\Theta\in Z(G_{n})$. We consider three cases.
\begin{enumerate}[(a)]
\item Suppose that $l=0$. By~(\ref{it:presI}), the generators of the form $\rho_{m,0}$, $1\leq m\leq n$, commute pairwise, so $\rho_{k,0}$ commutes with $\Theta$ for all $1\leq k\leq n$, which proves the result in this case. 
\item Suppose that $1\leq l<k$. By~(\ref{it:presI}) and~(\ref{it:presII}), $\rho_{k,l}$ and $\rho_{m,0}$ commute if $m<l$ or if $k<m$, so:
\begin{align}
\Theta \rho_{k,l}\Theta^{-1} &= \rho_{1,0}\cdots \rho_{k,0} \rho_{k,l} \rho_{k,0}^{-1} \cdots \rho_{1,0}^{-1}\notag\\
&= \rho_{k,0}\rho_{l,0}\cdots \rho_{k-1,0}\rho_{1,0}\cdots \rho_{l-1,0} \rho_{k,l} \rho_{l-1,0}^{-1}\cdots \rho_{1,0}^{-1} \rho_{k-1,0}^{-1}\cdots \rho_{l,0}^{-1} \rho_{k,0}^{-1}\notag\\
&= \rho_{k,0}\rho_{l,0}\cdots \rho_{k-1,0}\rho_{k,l} \rho_{k-1,0}^{-1}\cdots \rho_{l,0}^{-1} \rho_{k,0}^{-1},\label{eq:thetacomm}
\end{align}
using the fact that the $\rho_{m,0}$, $1\leq m\leq n$, commute pairwise. By~(\ref{it:presI}),
\begin{equation*}
\rho_{j,0} \rho_{k,l} \rho_{j,0}^{-1}=\begin{cases}
C_{k,j} \rho_{k,l} C_{k,j}^{-1} & \text{if $1\leq l<j$}\\
C_{k,l} & \text{if $l=j$,}
\end{cases}
\end{equation*}
and so
\begin{align}
\rho_{l,0}\cdots \rho_{k-1,0}\rho_{k,l} \rho_{k-1,0}^{-1}\cdots \rho_{l,0}^{-1}&= \rho_{l,0}\cdots \rho_{k-2,0} C_{k,k-1} \rho_{k,l} C_{k,k-1}^{-1} \rho_{k-2,0}^{-1}\cdots \rho_{l,0}^{-1}\notag\\
&= C_{k,k-1} \rho_{l,0}\cdots \rho_{k-2,0} \rho_{k,l} \rho_{k-2,0}^{-1}\cdots \rho_{l,0}^{-1} C_{k,k-1}^{-1}\notag\\
&= C_{k,k-1} \cdots C_{k,l+1}\rho_{l,0} \rho_{k,l} \rho_{l,0}^{-1} C_{k,l+1}^{-1} \ldots C_{k,k-1}^{-1}\notag\\
&= C_{k,k-1} \cdots C_{k,l+1}C_{k,l} C_{k,l+1}^{-1} \ldots C_{k,k-1}^{-1},\label{eq:rhoklcomm}
\end{align}
because $C_{k,m}$ commutes with $\rho_{q,0}$ for all $1\leq q<m<k$. Since $\rho_{k,m} \cdots \rho_{k,k-1}=D_{m,k}$, for all $1\leq m<k$, we have:
\begin{equation}\label{eq:conjrhokm}
\rho_{k,0}^{-1} \rho_{k,m} \cdots \rho_{k,k-1} \rho_{k,0}= \rho_{k,0}^{-1} D_{m,k} \rho_{k,0}.
\end{equation}
Let us prove by reverse induction on $m$ that for all $1\leq m<k$,
\begin{equation}\label{eq:prodcij}
C_{k,k-1} \cdots C_{k,m+1}C_{k,m} 
= \rho_{k,0}^{-1} D_{m,k} \rho_{k,0}.
\end{equation}
If $m=k-1$ then $C_{k,k-1}=\rho_{k,0}^{-1} \rho_{k,k-1} \rho_{k,0}=\rho_{k,0}^{-1} D_{k,k-1} \rho_{k,0}$ from the definition of $C_{k,k-1}$ and~\reqref{conjrhokm}, hence~\reqref{prodcij} is valid. So suppose that~\reqref{prodcij} holds for some $2\leq m<k$. Then by induction and the definitions of $C_{k,m-1}$ and $D_{m,k}$, it follows that:
\begin{align*}
C_{k,k-1} \cdots C_{k,m+1}C_{k,m} C_{k,m-1}&=\rho_{k,0}^{-1} D_{m,k} \rho_{k,0} \ldotp \rho_{k,0}^{-1} D_{m,k}^{-1} \rho_{k,m-1} D_{m,k} \rho_{k,0}\\
&= \rho_{k,0}^{-1} \rho_{k,m-1} D_{m,k} \rho_{k}= \rho_{k,0}^{-1} D_{m-1,k} \rho_{k,0},
\end{align*}
which proves~\reqref{prodcij}. Combining equations~\reqref{thetacomm},~\reqref{rhoklcomm},~\reqref{conjrhokm} and~\reqref{prodcij}, we see that $\Theta$ commutes with $\rho_{k,l}$ if $1\leq l<k$. 

\item Suppose that $k\leq l\leq 2k-2$. By~(\ref{it:presIII}), $\rho_{k,l}$ commutes with $\rho_{j,0}$ for all $k+1\leq j\leq n$ and by~(\ref{it:presI}), the $\rho_{m,0}$, $1\leq m\leq n$, commute pairwise, so:
\begin{multline}\label{eq:Thetakl}
\Theta \rho_{k,l} \Theta^{-1} = \rho_{1,0} \cdots \rho_{n,0} \rho_{k,l} \rho_{n,0}^{-1} \cdots \rho_{1,0}^{-1}= \rho_{1,0} \cdots \rho_{k,0} \rho_{k,l} \rho_{k,0}^{-1} \cdots \rho_{1,0}^{-1}\\
= \rho_{1,0} \cdots\rho_{l-k+1,0} \rho_{k,0}\bigl(\rho_{l-k+2,0} \cdots\rho_{k-1,0} \rho_{k,l} \rho_{k-1,0}^{-1} \cdots \rho_{l-k+2,0}^{-1} \bigr) \rho_{k,0}^{-1} \rho_{l-k+1,0}^{-1}\cdots\rho_{1,0}^{-1}.
\end{multline}
By~(\ref{it:presI}), we have: 
\begin{equation}\label{eq:rhom0kl}
\rho_{m,0} \rho_{k,l} \rho_{m,0}^{-1}= C_{k,m} \rho_{k,l} C_{k,m}^{-1}\; \text{for all $l-k+2\leq m\leq k-1$.}
\end{equation}
By definition, $C_{k,m}= \rho_{k,0}^{-1}\rho_{k,k-1}^{-1} \cdots\rho_{k,m+1}^{-1} \rho_{k,m} \rho_{k,m+1}\cdots \rho_{k,k-1}\rho_{k,0}$, so
\begin{equation}\label{eq:commckm}
C_{k,m}\rho_{s,0}=\rho_{s,0}C_{k,m} \quad \text{for all $1\leq s<m$ by~(\ref{it:presI}),}
\end{equation}
and hence:
\begin{align*}
\rho_{l-k+2,0} \cdots\rho_{k-1,0} \rho_{k,l} \rho_{k-1,0}^{-1} \cdots \rho_{l-k+2,0}^{-1}&= C_{k,k-1}\cdots C_{k,l-k+2} \rho_{k,l} C_{k,l-k+2}^{-1}\cdots C_{k,k-1}^{-1}\\
&= \rho_{k,0}^{-1} D_{l-k+2,k} \rho_{k,0} \rho_{k,l} \rho_{k,0}^{-1} D_{l-k+2,k}^{-1}  \rho_{k,0}
\end{align*}
by equations~\reqref{prodcij},~\reqref{rhom0kl} and~\reqref{commckm}. Therefore: 
\begin{align}\label{eq:conjThetakl}
\hspace*{-3mm}\Theta \rho_{k,l} \Theta^{-1} & \!=\! \rho_{1,0} \cdots\rho_{l-k+1,0} D_{l-k+2,k} \rho_{k,0} \rho_{k,l} \rho_{k,0}^{-1} D_{l-k+2,k}^{-1} \rho_{l-k+1,0}^{-1}\cdots\rho_{1,0}^{-1}\notag\\
&\!=\! \rho_{l-k+1,0}\bigl( \rho_{1,0} \cdots\rho_{l-k,0} D_{l-k+2,k} \rho_{k,0} \rho_{k,l} \rho_{k,0}^{-1} D_{l-k+2,k}^{-1} \rho_{l-k,0}^{-1}\cdots\rho_{1,0}^{-1} \bigr) \rho_{l-k+1,0}^{-1}\notag\\
&\!=\! \rho_{l-k+1,0} D_{l-k+2,k} \rho_{k,0} \bigl( \rho_{1,0} \cdots\rho_{l-k,0} \rho_{k,l} \rho_{l-k,0}^{-1}\cdots\rho_{1,0}^{-1} \bigr) \rho_{k,0}^{-1} D_{l-k+2,k}^{-1} \rho_{l-k+1,0}^{-1},
\end{align}
by~\reqref{Thetakl}, since $D_{l-k+2,k}=\rho_{k,l-k+2}\cdots \rho_{k,k-1}$, so it commutes with $\rho_{s,0}$ for all $1\leq s\leq l-k+1$ by~(\ref{it:presI}). 
Now from~\req{prodcij}, we obtain: 
\begin{align}
C_{k,l-k} \cdots C_{k,1}&= (C_{k,k-1} \cdots C_{k,l-k+1})^{-1} \ldotp C_{k,k-1}\cdots C_{k,l-k+1}\ldotp C_{k,l-k} \cdots C_{k,1}\notag\\
&= \rho_{k,0}^{-1} D_{l-k+1,k}^{-1} D_{1,k} \rho_{k,0},\label{eq:Cklkk1}
\end{align}
and by~(\ref{it:presI}), we have: 
\begin{equation}\label{eq:conjkq}
\rho_{m,0} \rho_{k,q} \rho_{m,0}^{-1} =C_{k,m}\rho_{k,k+m-1} \rho_{k,q} \rho_{k,k+m-1}^{-1} C_{k,m}^{-1}\quad \text{if $1\leq m\leq q-k$ and $q\leq 2k-2$.}
\end{equation}
So for all $1\leq m\leq l-k$,
\begin{multline}\label{eq:conjrhom0kmk}
\rho_{m,0}(\rho_{k,m+k} \rho_{k,m+k+1}\cdots \rho_{k,l-1}\rho_{k,l} \rho_{k,l-1}^{-1}\cdots \rho_{k,m+k+1}^{-1}\rho_{k,m+k}^{-1}) \rho_{m,0}=\\ C_{k,m}\rho_{k,m+k-1}\rho_{k,m+k}
\rho_{k,m+k+1}\cdots \rho_{k,l-1}\rho_{k,l} \rho_{k,l-1}^{-1}\cdots \rho_{k,m+k+1}^{-1}\rho_{k,m+k}^{-1} \rho_{k,m+k-1}^{-1} C_{k,m}^{-1}.
\end{multline}
Using equations~\reqref{commckm},~\reqref{Cklkk1} and~\reqref{conjrhom0kmk}, we see that:
\begin{align}\label{eq:prodrhoconj}
\rho_{1,0} \cdots\rho_{l-k,0} \rho_{k,l} \rho_{l-k,0}^{-1}\cdots\rho_{1,0}^{-1}=& \rho_{1,0} \cdots\rho_{l-k-1,0} C_{k,l-k} \rho_{k,l-1} \rho_{k,l} \rho_{k,l-1}^{-1} C_{k,l-k}^{-1} \rho_{l-k-1,0}^{-1}\cdots\rho_{1,0}^{-1}\notag\\
=& C_{k,l-k}\rho_{1,0} \cdots \rho_{l-k-2,0}(\rho_{l-k-1,0} \rho_{k,l-1} \rho_{k,l} \rho_{k,l-1}^{-1} \rho_{l-k-1,0}^{-1})\ldotp\notag\\
&\rho_{l-k-2,0}^{-1}\cdots\rho_{1,0}^{-1}C_{k,l-k}^{-1}\notag\\
=& C_{k,l-k} \cdots C_{k,1} E_{k,k,l-1} \rho_{k,l} E_{k,k,l-1}^{-1} C_{k,1}^{-1} \cdots C_{k,l-k}^{-1}\notag\\
=& \rho_{k,0}^{-1} D_{l-k+1,k}^{-1} D_{1,k} \rho_{k,0} E_{k,k,l-1} \rho_{k,l} E_{k,k,l-1}^{-1} \rho_{k,0}^{-1} D_{1,k}^{-1} \ldotp\notag\\
& D_{l-k+1,k} \rho_{k,0}.
\end{align}
Combining~\reqref{conjThetakl} and~\reqref{prodrhoconj}, and using the relation $D_{l-k+2,k} D_{l-k+1,k}^{-1}=\rho_{k,l-k+1}^{-1}$, we obtain: 
\begin{equation}\label{eq:endtheta}
\Theta \rho_{k,l} \Theta^{-1}= \rho_{l-k+1,0} (\rho_{k,l-k+1}^{-1} D_{1,k} \rho_{k,0} E_{k,k,l-1} \rho_{k,l} E_{k,k,l-1}^{-1} \rho_{k,0}^{-1} D_{1,k}^{-1} \rho_{k,l-k+1}) \rho_{l-k+1,0}^{-1}.
\end{equation}
In~\reqref{endtheta}, it remains to conjugate by $\rho_{l-k+1,0}$. Using relations~(\ref{it:presI}), we have: 
\begin{equation}
\!\!\rho_{l-k+1,0} \rho_{k,m} \rho_{l-k+1,0}^{-1}= \begin{cases}
\rho_{k,m} & \text{if $m=0$ or if $l-k+2 \leq m\leq k-1$}\\
C_{k,l-k+1} \rho_{k,m} C_{k,l-k+1}^{-1} & \text{if $1\leq m\leq l-k$ or if $k\leq m\leq l-1$}\\
C_{k,l-k+1} & \text{if $m=l-k+1$,}
\end{cases}\label{eq:conjrhokm1}
\end{equation}
and
\begin{equation}
\!\!\rho_{l-k+1,0} \rho_{k,m} \rho_{l-k+1,0}^{-1}= C_{k,l-k+1} E_{k,k,l-1}^{-1} \rho_{k,0}^{-1} D_{1,k}^{-1} \rho_{k,l} D_{1,k} \rho_{k,0} E_{k,k,l-1} C_{k,l-k+1}^{-1} \; \text{if $m=l$.}\label{eq:conjrhokm2}
\end{equation}
From the definition of $C_{k,l-k+1}$, we have $\rho_{k,0}  C_{k,l-k+1} \rho_{k,0}^{-1}=D_{l-k+2,k}^{-1} \rho_{k,l-k+1}D_{l-k+2,k}$, and together with~\reqref{endtheta},~\reqref{conjrhokm1} and~\reqref{conjrhokm2}  it  follows that:
\begin{align*}
\Theta \rho_{k,l} \Theta^{-1}=& C_{k,l-k+1}^{-1} \ldotp C_{k,l-k+1} \rho_{k,1}\cdots \rho_{k,l-k} C_{k,l-k+1}^{-1}\ldotp C_{k,l-k+1} \ldotp \rho_{k,l-k+2} \cdots \rho_{k,k-1}\ldotp \rho_{k,0} \ldotp\\
& C_{k,l-k+1} E_{k,k,l-1} C_{k,l-k+1}^{-1} \ldotp C_{k,l-k+1} E_{k,k,l-1}^{-1} \rho_{k,0}^{-1} D_{1,k}^{-1} \rho_{k,l} D_{1,k} \rho_{k,0} E_{k,k,l-1} C_{k,l-k+1}^{-1} \ldotp\\
& C_{k,l-k+1} E_{k,k,l-1}^{-1} C_{k,l-k+1}^{-1}\ldotp \rho_{k,0}^{-1}\ldotp \rho_{k,k-1}^{-1}\cdots \rho_{k,l-k+2}^{-1} \ldotp C_{k,l-k+1}^{-1}\ldotp\\
&  C_{k,l-k+1} \rho_{k,l-k}^{-1} \cdots \rho_{k,1}^{-1} C_{k,l-k+1}^{-1} \ldotp C_{k,l-k+1}\\
=& \rho_{k,1}\cdots \rho_{k,l-k}   D_{l-k+2,k} (\rho_{k,0}  C_{k,l-k+1} \rho_{k,0}^{-1}) D_{1,k}^{-1} \rho_{k,l} D_{1,k} (\rho_{k,0} C_{k,l-k+1}^{-1} \rho_{k,0}^{-1})\ldotp\\
& D_{l-k+2,k}^{-1}  \rho_{k,l-k}^{-1} \cdots \rho_{k,1}^{-1}\\
=& \rho_{k,1}\cdots \rho_{k,l-k} D_{l-k+2,k} (D_{l-k+2,k}^{-1} \rho_{k,l-k+1}D_{l-k+2,k}) D_{1,k}^{-1} \rho_{k,l} D_{1,k}\ldotp\\
& (D_{l-k+2,k}^{-1} \rho_{k,l-k+1}^{-1}D_{l-k+2,k}) D_{l-k+2,k}^{-1}  \rho_{k,l-k}^{-1} \cdots \rho_{k,1}^{-1}\\
=& (\rho_{k,1}\cdots \rho_{k,l-k} \rho_{k,l-k+1}D_{l-k+2,k}) D_{1,k}^{-1} \rho_{k,l} D_{1,k} (D_{l-k+2,k}^{-1} \rho_{k,l-k+1}^{-1} \rho_{k,l-k}^{-1} \cdots \rho_{k,1}^{-1})\\
=&  D_{1,k}\ldotp D_{1,k}^{-1} \rho_{k,l} D_{1,k} \ldotp D_{1,k}^{-1}=\rho_{k,l}.
\end{align*}
We have thus shown that $\Theta$ commutes with each of the generators $\rho_{k,l}$, of $G_{n}$, where $1\leq k\leq n$ and $0\leq l\leq 2k-2$, and so we conclude that $\Theta\in Z(G_{n})$ for all $n\in \N$. 
\end{enumerate}

It remains to prove that $Z(G_{n})= \ang{\Theta_{n}}$. The proof is similar to that for $P_{n}$. If $n=1$ then $G_{1}=\ang{\Theta_{1}}$ from \repr{presgn}, so the result holds. Assume by induction that $Z(G_{n-1})= \ang{\Theta_{n-1}}$ for some $n\geq 2$, let $x\in Z(G_{n})$, and consider the short exact sequence~\reqref{sesequi}. Since $(q_{n})_{\#}$ is surjective, $(q_{n})_{\#}(x)\in Z(G_{n-1})$, and so by induction, there exists $l\in \Z$ such that $(q_{n})_{\#}(x)=\Theta_{n-1}^{l}$. From the definition of $q_{n}$, we have $(q_{n})_{\#}(\rho_{i,0})=\rho_{i,0}$ if $1\leq i\leq n-1$ and $(q_{n})_{\#}(\rho_{n,0})=1$, thus $(q_{n})_{\#}(\Theta_{n})=\Theta_{n-1}$. Since $\Theta_{n}\in Z(G_{n})$, $x \Theta_{n}^{-l}\in \ker{(q_{n})_{\#}}\cap Z(G_{n})$, so $x \Theta_{n}^{-l}\in Z(\ker{(q_{n})_{\#}})$. Thus $x \Theta_{n}^{-l}$ is trivial because $\ker{(q_{n})_{\#}}$ is a free group of rank $2n-1$. It follows that $x \in \ang{\Theta_{n}}$, and this completes the proof of the proposition.
\end{proof}

\begin{rem}\label{rem:secondsec}
As we already mentioned, the section $\map{(s_{n+1})_{\#}}{G_{n}}[G_{n+1}]$ gives rise to a decomposition of $G_{n}$ as an iterated semi-direct product of the form $\F[2n-1]\rtimes (\F[2n-3] \rtimes (\cdots\rtimes(\F[3]\rtimes \Z)\cdots))$. We may choose a different section for $(q_{n+1})_{\#}$ so that $\ang{\Theta_{n}}$ appears as a direct factor of $G_{n}$ as follows. Let $\map{s_{n+1}'}{\orbconf}[{\orbconf[n+1]}]$ be the map defined by $s_{n+1}'(y_{1},\ldots,y_{n})= (y_{1},\ldots,y_{n},y_{n+1})$, where the point $y_{n+1}\in C$ belongs to the segment between $y_{n}$ and $\widetilde{z}_{0}$ and is sufficiently close to $y_{n}$. The induced homomorphism $\map{(s_{n+1}')_{\#}}{G_{n}}[G_{n+1}]$ sends $\rho_{n,0}$ to $\rho_{n,0}\rho_{n+1,0}$, and $\Theta_{n}$ to $\Theta_{n+1}$. As above, $G_{n}$ may be decomposed as an iterated semi-direct product of free groups, but by \repr{gncentre}, it is of the form: 
\begin{equation}\label{eq:gndecomp}
G_{n}\cong \F[2n-1]\rtimes (\F[2n-3] \rtimes (\cdots\rtimes(\F[5]\rtimes \F[3])\cdots))\times \Z,
\end{equation}
where the $\Z$-factor is generated by $\Theta_{n}$. We will make use of this section to prove \repr{sesnngen} at the end of \resec{general}. 
\end{rem}

A decomposition of $G_{n}$ as a direct sum one of whose factors is $\Z$ may also be obtained without reference to a section by using the following lemma.

\begin{lem}\label{lem:directsum}
Let $F,G$ and $H$ be groups, let $\map{\alpha}{F}[G\oplus H]$ be a surjective homomorphism, and let $\widetilde{G}$ be a subgroup of $Z(F)$ such that the restriction of $\alpha$ to $\widetilde{G}$ is an isomorphism onto $G$. Then $F=\widetilde{G} \oplus \alpha^{-1}(H)$.
\end{lem}

\begin{proof}
First of all, we claim that $F$ is generated by $\widetilde{G} \cup \alpha^{-1}(H)$. To see this, let $x\in F$. Then there exist $g\in G$ and $h\in H$ such that $\alpha(x)=gh$. Thus there exists a (unique) $\widetilde{g}\in \widetilde{G}$ such that $\alpha(\widetilde{g})=g$. Thus $\alpha(\widetilde{g}^{-1}x)=h$, and hence $\widetilde{g}^{-1}x\in \alpha^{-1}(H)$. So there exists $\widetilde{h}\in \alpha^{-1}(H)$ such that $x=\widetilde{g}\,\widetilde{h}$, which proves the result. This decomposition is unique, since if there exist $\widetilde{g}_{1}\in \widetilde{G}$ and $\widetilde{h}_{1}\in \alpha^{-1}(H)$ such that $x=\widetilde{g}_{1}\ldotp\widetilde{h}_{1}$ then $\alpha(x)=\alpha(\widetilde{g})\,\alpha(\widetilde{h})=\alpha(\widetilde{g}_{1})\,\alpha(\widetilde{h}_{1})$, and since $\alpha(\widetilde{g}),\alpha(\widetilde{g}_{1}) \in G$ and $\alpha(\widetilde{h}),\alpha(\widetilde{h}_{1}) \in H$, it follows that $\alpha(\widetilde{g})=\alpha(\widetilde{g}_{1})$, and so $\widetilde{g}=\widetilde{g}_{1}$ by the bijectivity of $\alpha\left\lvert_{\widetilde{G}}\right.$, from which we deduce also that $\widetilde{h}=\widetilde{h}_{1}$. In particular, we have $\widetilde{G}\cap \alpha^{-1}(H)=\brak{e}$. Now $\widetilde{G}$ is normal in $F$ since $\widetilde{G}$ is a subgroup of $Z(F)$, and $\alpha^{-1}(H)$ is normal in $F$ because $H$ is normal in $G\oplus H$ and $\alpha$ is surjective. We conclude that $F=\widetilde{G} \oplus \alpha^{-1}(H)$.
\end{proof}

\begin{prop}
For all $n\in \N$, there exists a subgroup $H_{n}$ of $G_{n}$ such that $G_{n}=H_{n} \oplus \ang{\Theta_{n}}$.
\end{prop}

\begin{proof}
If $n=1$ then $G_{1}=\ang{\rho_{1,0}}=\ang{\Theta_{1}}$, and it suffices to take $H_{1}$ to be the trivial group. So suppose that the result holds for $n-1$, where $n\geq 2$, so there exists a subgroup $H_{n-1}$ of $G_{n-1}$ for which $G_{n-1}=H_{n-1}\oplus \ang{\Theta_{n-1}}$. The homomorphism $\map{(q_{n})_{\#}}{G_{n}}[G_{n-1}]$ of \req{sesequi} is surjective, and as we saw at the end of the proof of \repr{gncentre}, $(q_{n})_{\#}(\Theta_{n})=\Theta_{n-1}$, so $\map{(q_{n})_{\#}\left\lvert_{\ang{\Theta_{n}}}\right.}{\ang{\Theta_{n}}}[\ang{\Theta_{n-1}}]$ is an isomorphism. Further, $\ang{\Theta_{n}}=Z(G_{n})$ by \repr{gncentre}, and the result follows from \relem{directsum} by taking $\widetilde{G}=\ang{\Theta_{n}}$ and $H_{n}=((q_{n})_{\#})^{-1}(H_{n-1})$.
\end{proof}

\section{The boundary homomorphism for $\St$ and $\rp$}\label{sec:boundhomo}

In all of this section, $M$ will be $\St$ or $\rp$. Set $n_{0}=3$ (resp.\ $n_{0}=2$) if $M=\St$ (resp.\ $M=\rp$), and let $n\geq n_{0}$. By \reth{homotopytypefns}, $\pi_{2}(F_{n}(M))$ is trivial.
Considering the tail of the exact sequence~\reqref{lesiota} and using the homotopy equivalence $\map{g_{n}}{E_{n}}[F_{n}(M)]$, we obtain the following exact sequence: 
\begin{equation}\label{eq:lespartialn}
1 \to \pi_2\left(\prod_1^n\, M\right) \stackrel{\partial_{n}}{\to} \pi_{1}(I_{n}) \xrightarrow{(g_{n}\circ j_{n})_{\#}} P_{n}(M) \xrightarrow{(\iota_{n})_{\#}} \pi_1\left(\prod_1^n\, M\right) \to 1,
\end{equation}
where $\partial_{n}=\partial_{n,2}$ is the associated boundary homomorphism. If $M=\St$ then this sequence is short exact. The aim of this section is to describe completely $\partial_{n}$. 
In \resec{geomrepsS2}, we first describe explicit generators of $\pi_{1}(\Omega (\St))$ and $\pi_{1}(\Omega (\rp))$ that will be used in the rest of the paper. In \resec{notation}, we exhibit homotopy equivalences between $I_{c}$ and $I_{n}''$, and $I_{n}''$ and $I_{n}'$ in \relem{fundhomoequiv}. In \resec{n3S2}, we analyse the case $n=n_{0}$, the main results being \reth{brrelnS2}, which describes the relation between certain elements of $\pi_{1}(I_{n_{0}})$, and \reco{sesnn0}, which makes explicit the boundary homomorphism $\partial_{n_{0}}$. In \resec{general}, we then use the results of \resec{n3S2} to study the case $n>n_{0}$, and to prove \reth{tauhatsquareS2}. \reco{sesnngen} describes the boundary homomorphism $\partial_{n}$, and generalises \reco{sesnn0}. We shall make use of the notation defined in \resec{generalities}. Recall that $\map{\pi}{\St}[\rp]$ denotes the universal covering, and that $\map{\tau}{\St}$ is the antipodal map.

\subsection{Geometric representatives of generators of $\pi_{1}(\Omega (M))$}\label{sec:geomrepsS2}

We follow the notation of \resec{generalities}. For $i\in\brak{1,2,3}$, let $\map{p_{i}}{\R^{3}}[\R]$ denote projection onto the $i$\textsu{th} coordinate. Taking $n=n_{0}$, the basepoint $W_{n_{0}}$ is $(\widetilde{x}_{0}, \widetilde{z}_{0},-\widetilde{z}_{0})$ (resp.\ $(x_{0},z_{0})$) 
if $M=\St$ (resp.\ $M=\rp$). If $u\in \brak{\widetilde{x}_{0},\pm\widetilde{z}_{0}}$, we shall identify $\pi_{1}(\Omega(\St),c_{u})$ with $\pi_{2}(\St,u)$, and $\pi_{1}(\Omega(\rp),c_{\pi(u)})$ with $\pi_{2}(\rp,\pi(u))$ in the usual way. In this section, we define explicit geometric representatives of generators of $\pi_{1}(\Omega(\St),u)$ and of $\pi_{1}(\Omega(\rp),c_{\pi(u)})$, and we prove two lemmas that shall be used in what follows.

Let $\Pi_0$ be the tangent plane to $\St$ at $\widetilde{x}_{0}$, and let $D$ be the straight line of equation $\begin{cases}
x=1\\
y=0
\end{cases}$ contained in $\Pi_0$. For $t\in I$, let $\Pi_t$ be the plane given by rotating $\Pi_0$ about $D$ through an angle $t\pi$, and let $\widetilde{\omega}_{\widetilde{x}_{0},t}=\St\cap \Pi_t$. We choose the orientation and parametrisation so that:
\begin{enumerate}\renewcommand\theenumi{X\alph{enumi}}
\item\label{it:xa} for all $s\in (0,1)$, $p_{2}(\widetilde{\omega}_{\widetilde{x}_{0},t}(s))<0$ if $t\in \left(0,\frac{1}{2}\right)$, and $p_{2}(\widetilde{\omega}_{\widetilde{x}_{0},t}(s))>0$ if $t\in \left(\frac{1}{2},1\right)$.
\item\label{it:xb} for all $t\in (0,1)$, $p_{3}(\widetilde{\omega}_{\widetilde{x}_{0},t}(s))<0$ if $s\in \left(0,\frac{1}{2}\right)$ and $p_{3}(\widetilde{\omega}_{\widetilde{x}_{0},t}(s))>0$ if $s\in \left(\frac{1}{2},1\right)$.
\item\label{it:xc}  for all $t,s\in I$, $p_{i}(\widetilde{\omega}_{\widetilde{x}_{0},t}(s))=p_{i}(\widetilde{\omega}_{\widetilde{x}_{0},t}(1-s))$ for $i=1,2$, and $p_{3}(\widetilde{\omega}_{\widetilde{x}_{0},t}(s))=-p_{3}(\widetilde{\omega}_{\widetilde{x}_{0},t}(1-s))$.

\item\label{it:xd} for all $t,s\in I$, $p_{2}(\widetilde{\omega}_{\widetilde{x}_{0},t}(s))=-p_{2}(\widetilde{\omega}_{\widetilde{x}_{0},1-t}(s))$, and $p_{i}(\widetilde{\omega}_{\widetilde{x}_{0},t}(s))=p_{i}(\widetilde{\omega}_{\widetilde{x}_{0},1-t}(s))$ for $i=1,3$.

\item\label{it:xe} for all $t\in I$, $p_{1}(\widetilde{\omega}_{\widetilde{x}_{0},t}(s))>0$ if $s\in \left[0,\frac{1}{4}\right)$, and $p_{1}(\widetilde{\omega}_{\widetilde{x}_{0},\frac{1}{2}}(s))<0$ if $s\in \left(\frac{1}{4},\frac{1}{2}\right]$.
\end{enumerate}
For each $t\in I$, $\widetilde{\omega}_{\widetilde{x}_{0},t}$ is a loop based at $\widetilde{x}_{0}$, $\widetilde{\omega}_{\widetilde{x}_{0},0}=\widetilde{\omega}_{\widetilde{x}_{0},1}=c_{\widetilde{x}_{0}}$, and for each $\widetilde{w}\in \St\setminus \brak{\widetilde{x}_{0}}$, there exist unique values $t,s\in (0,1)$ for which $\widetilde{\omega}_{\widetilde{x}_{0},t}(s)=\widetilde{w}$. Observe also that:
\begin{enumerate}[\textbullet]
\item from condition~(\ref{it:xb}), $p_{3}(\widetilde{\omega}_{\widetilde{x}_{0},t}(\frac{1}{2}))=0$ for all $t\in I$, so the loop $(\widetilde{\omega}_{\widetilde{x}_{0},t}(\frac{1}{2}))_{t\in I}$ is the great circle lying in the $xy$-plane. 
\item by conditions~(\ref{it:xc}) and~(\ref{it:xe}), for all $t\in I$, $p_{1}(\widetilde{\omega}_{\widetilde{x}_{0},t}(s))>0$ if $s\in \left[0,\frac{1}{4}\right) \cup \left(\frac{3}{4}, 1\right]$, and $p_{1}(\widetilde{\omega}_{\widetilde{x}_{0},\frac{1}{2}}(s))<0$ if $s\in \left(\frac{1}{4},\frac{3}{4}\right)$.
\end{enumerate}

Consider the rotation $R$ about the straight line of equation $x=z$ and $y=0$ by an angle $\pi$. The matrix of $R$ in the standard basis of $\R^{3}$ is $\left(\begin{smallmatrix}
0 & 0 & 1\\
0 & -1 & 0\\
1 & 0 & 0
\end{smallmatrix}\right)$. For all $t\in I$, let $\widetilde{\omega}_{\widetilde{z}_{0},t}= R(\widetilde{\omega}_{\widetilde{x}_{0},t})$. It follows from the above conditions on $\widetilde{\omega}_{\widetilde{x}_{0},t}$ that:
\begin{enumerate}\renewcommand\theenumi{Z\alph{enumi}}
\item\label{it:za} for all $s\in (0,1)$, $p_{2}(\widetilde{\omega}_{\widetilde{z}_{0},t}(s))>0$ if $t\in \left(0,\frac{1}{2}\right)$, and $p_{2}(\widetilde{\omega}_{\widetilde{z}_{0},t}(s))<0$ if $t\in \left(\frac{1}{2},1\right)$.
\item\label{it:zb} for all $t\in (0,1)$, $p_{1}(\widetilde{\omega}_{\widetilde{z}_{0},t}(s))<0$ if $s\in \left(0,\frac{1}{2}\right)$ and $p_{1}(\widetilde{\omega}_{\widetilde{z}_{0},t}(s))>0$ if $s\in \left(\frac{1}{2},1\right)$.
\item\label{it:zc} for all $t,s\in I$, $p_{i}(\widetilde{\omega}_{\widetilde{z}_{0},t}(s))= p_{i}(\widetilde{\omega}_{\widetilde{z}_{0},t}(1-s))$ for $i=2,3$, and $p_{1}(\widetilde{\omega}_{\widetilde{z}_{0},t}(s))=-p_{1}(\widetilde{\omega}_{\widetilde{z}_{0},t}(1-s))$.

\item\label{it:zd} for all $t,s\in I$, $p_{2}(\widetilde{\omega}_{\widetilde{z}_{0},t}(s))=-p_{2}(\widetilde{\omega}_{\widetilde{z}_{0},1-t}(s))$, and $p_{i}(\widetilde{\omega}_{\widetilde{z}_{0},t}(s))=p_{i}(\widetilde{\omega}_{\widetilde{z}_{0},1-t}(s))$ for $i=1,3$.

\item\label{it:ze} for all $t\in I$, $p_{3}(\widetilde{\omega}_{\widetilde{z}_{0},t}(s))>0$ if $s\in \left[0,\frac{1}{4}\right)$, and $p_{3}(\widetilde{\omega}_{\widetilde{z}_{0},\frac{1}{2}}(s))<0$ if $s\in \left(\frac{1}{4},\frac{1}{2}\right]$.
\end{enumerate}
In particular:
\begin{enumerate}[\textbullet]
\item by condition~(\ref{it:zb}), $p_{1}(\widetilde{\omega}_{\widetilde{z}_{0},t}(\frac{1}{2}))=0$, so the loop $(\widetilde{\omega}_{\widetilde{z}_{0},t}(\frac{1}{2}))_{t\in I}$ is the great circle lying in the $yz$-plane. 
\item by conditions~(\ref{it:zc}) and~(\ref{it:ze}), for all $t\in I$, $p_{3}(\widetilde{\omega}_{\widetilde{z}_{0},t}(s))>0$ if $s\in \left[0,\frac{1}{4}\right) \cup \left(\frac{3}{4}, 1\right]$, and $p_{3}(\widetilde{\omega}_{\widetilde{z}_{0},\frac{1}{2}}(s))<0$ if $s\in \left(\frac{1}{4},\frac{3}{4}\right)$.
\end{enumerate}
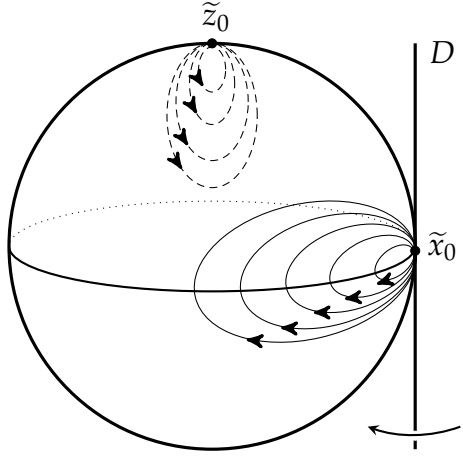
\begin{figure}[t]
\hfill 
\begin{tikzpicture}[scale=0.6, >=stealth]
\draw[very thick] (0,0) circle(4.5);
\draw[dotted] (4.5,0) arc (0:180:4.5 and 1);
\draw[thick] (-4.5,0) arc (180:360:4.5 and 1);

\draw[rotate around={38:(4.5,0)}, middlearrow=0.65] (3.92,0) ellipse [x radius = 0.5, y radius = 0.3];
\draw[rotate around={25:(4.5,0)}, middlearrow=0.65] (3.45,0) ellipse [x radius = 1, y radius = 0.6];
\draw[rotate around={19:(4.5,0)}, middlearrow=0.65] (3,0) ellipse [x radius = 1.5, y radius = 0.9];
\draw[rotate around={16:(4.5,0)}, middlearrow=0.65] (2.5,0) ellipse [x radius = 2, y radius = 1.2];
\draw[rotate around={13:(4.5,0)}, middlearrow=0.65] (2,0) ellipse [x radius = 2.5, y radius = 1.5];

\draw[densely dashed, middlearrowrev=0.65] (0,3.92) ellipse [x radius = 0.3, y radius = 0.5];
\draw[densely dashed, middlearrowrev=0.65] (0,3.6) ellipse [x radius = 0.5, y radius = 0.9];
\draw[densely dashed, middlearrowrev=0.65] (0,3.2) ellipse [x radius = 0.8, y radius = 1.3];
\draw[densely dashed, middlearrowrev=0.65] (0,2.9) ellipse [x radius = 1, y radius = 1.6];

\draw[fill] (4.5,-0.1) circle(0.1);
\draw[fill] (0,4.5) circle(0.1);
\node at (5.1,0){$\widetilde{x}_{0}$};
\node at (0.1,5.1){$\widetilde{z}_{0}$};
\node at (5.1,4.3){$D$};

\draw[very thick] (4.5,-4.5) -- (4.5,4.5);
\draw[fill, white] (4.5,-4.2) circle(0.1);
\draw [thick, ->] (5.5,-4) arc [radius=3, start angle=290, end angle= 250];
\end{tikzpicture}
\hspace*{\fill}
\caption{The loops $\widetilde{\omega}_{\widetilde{x}_{0},t}$ and $\widetilde{\omega}_{\widetilde{z}_{0},t}$ for small values of $t$. The dotted lines indicate that the loops are on the `hidden' side of the sphere.}
\label{fig:smallt}
\end{figure}
The loop $\widetilde{\omega}_{u,t}$ is illustrated in Figure~\ref{fig:smallt} for $u\in \brak{\widetilde{x}_{0},\widetilde{z}_{0}}$ and small values of $t$. We also set:
\begin{equation}\label{eq:minusomega}
\widetilde{\omega}_{-u,t}= -\widetilde{\omega}_{u,1-t}.
\end{equation}
If $u\in \brak{\widetilde{x}_{0},\widetilde{z}_{0}, -\widetilde{z}_{0}}$, then $(\widetilde{\omega}_{u,t})_{t\in I}$ is a geometric representative of a generator, which we shall denote by $\widetilde{\lambda}_{u}$, of $\pi_{1}(\Omega(\St),c_{u})$. The fact that the covering map $\pi$ induces an isomorphism on the $\pi_{2}$-level implies that $(\pi\circ \widetilde{\omega}_{u,t})_{t\in I}$ is a geometric representative of a generator of $\pi_{1}(\Omega(\rp), c_{\pi(u)})$ that we denote by $\lambda_{\pi(u)}$. Note that:
\begin{equation}\label{eq:minusz0}
\lambda_{\pi(-\widetilde{z}_{0})}= [(\pi\circ \widetilde{\omega}_{-\widetilde{z}_{0},t})_{t\in I}]= [(\pi\circ (-\widetilde{\omega}_{\widetilde{z}_{0},1-t}))_{t\in I}]= [(\pi\circ \widetilde{\omega}_{\widetilde{z}_{0},t}^{-1})_{t\in I}]= -\lambda_{\pi(\widetilde{z}_{0})}
\end{equation}
using~\reqref{minusomega}.
If $M=\St$ (resp.\ $M=\rp$), each direct factor of $\pi_1(\prod_1^{n_{0}}\, \Omega (M), W_{n_{0}})$ will be identified with the corresponding group 
$\pi_{1}(\Omega(\St),c_{u})$ (resp.\ $\pi_{1}(\Omega(\rp), c_{\pi(u)})$).
If $u\in \brak{\widetilde{x}_{0},\widetilde{z}_{0}, -\widetilde{z}_{0}}$ (resp.\ $u\in \brak{\widetilde{x}_{0},\widetilde{z}_{0}}$), let $\widetilde{\delta}_{u}= \partial_{n_{0}}(\widetilde{\lambda}_{u})$ (resp.\ $\delta_{\pi(u)}= \partial_{n_{0}}(\lambda_{\pi(u)})$ 
in $\pi_{1}(I_{n_{0}})$. 

\begin{lem}\label{lem:boundaryS2}
With the above notation:
\begin{enumerate}[(a)]
\item\label{it:boundaryS2a} if $M=\St$, the loop $\bigl(\widetilde{x}_0, \widetilde{z}_0, -\widetilde{z}_0, \widetilde{\omega}_{\widetilde{x}_{0}, t}, c_{\widetilde{z}_0}, c_{-\widetilde{z}_0}\bigr)_{t\in I}$  (resp.\ $\bigl(\widetilde{x}_0, \widetilde{z}_0, -\widetilde{z}_0, c_{\widetilde{x}_0}, \widetilde{\omega}_{\widetilde{z}_{0}, t}, c_{-\widetilde{z}_0}\bigr)_{t\in I}$, $\bigl(\widetilde{x}_0, \widetilde{z}_0, -\widetilde{z}_0, c_{\widetilde{x}_0}, c_{\widetilde{z}_0}, \widetilde{\omega}_{-\widetilde{z}_{0}, t}\bigr)_{t\in I}$) is a geometric representative in $I_{n_{0}}$ of $\widetilde{\delta}_{\widetilde{x}_{0}}$ (resp.\ of $\widetilde{\delta}_{\widetilde{z}_{0}}$, $\widetilde{\delta}_{-\widetilde{z}_{0}}$).
\item\label{it:boundaryS2b} if $M=\rp$, the loop $(x_{0}, z_{0}, \pi\circ \widetilde{\omega}_{\widetilde{x}_{0},t},c_{z_{0}})_{t\in I}$ (resp.\ $(x_{0}, z_{0}, c_{x_{0}}, \pi\circ \widetilde{\omega}_{\pm\widetilde{z}_{0},t})_{t\in I}$) is a geometric representative in $I_{n_{0}}$ of $\delta_{x_{0}}$ (resp.\ of $\pm\delta_{z_{0}}$). 
\end{enumerate}
\end{lem}

\begin{proof}
The result is a consequence of the application of the construction given in the proof of \relem{boundary} to the lower fibration of~\reqref{commdiagfib}, and \req{minusz0}.
\end{proof}

\begin{lem}\label{lem:omegaoneminust}
Let $u\in \brak{\widetilde{x}_{0},\widetilde{z}_{0}}$. Then in $\pi_{1}(\Omega(\St), c_{u})$, $[(\widetilde{\omega}_{u,t})_{t\in I}]=[(\widetilde{\omega}_{u,1-t}^{-1})_{t\in I}]$.
\end{lem}

\begin{proof}
First suppose that $u=\widetilde{x}_{0}$, and for $\theta\in [0,\pi]$, let $R_{\theta}'=\left( \begin{smallmatrix}
1 & 0 & 0\\
0 & \cos \theta & -\sin \theta\\
0 & \sin \theta & \cos \theta
\end{smallmatrix}
\right)$. Then for all $t\in I$, $R_{0}'(\widetilde{\omega}_{\widetilde{x}_{0},t})=\widetilde{\omega}_{\widetilde{x}_{0},t}$, and using conditions~(\ref{it:xc}) and~(\ref{it:xd}), for all $t,s\in I$, we have:
\begin{align*}
R_{\pi}'(\widetilde{\omega}_{\widetilde{x}_{0},t}(s)) &=\begin{pmatrix}
p_{1}(\widetilde{\omega}_{\widetilde{x}_{0},t}(s))\\
-p_{2}(\widetilde{\omega}_{\widetilde{x}_{0},t}(s))\\
-p_{3}(\widetilde{\omega}_{\widetilde{x}_{0},t}(s))
\end{pmatrix}=\begin{pmatrix}
p_{1}(\widetilde{\omega}_{\widetilde{x}_{0},1-t}(s))\\
p_{2}(\widetilde{\omega}_{\widetilde{x}_{0},1-t}(s))\\
-p_{3}(\widetilde{\omega}_{\widetilde{x}_{0},1-t}(s))
\end{pmatrix}=\begin{pmatrix}
p_{1}(\widetilde{\omega}_{\widetilde{x}_{0},1-t}(1-s))\\
p_{2}(\widetilde{\omega}_{\widetilde{x}_{0},1-t}(1-s))\\
p_{3}(\widetilde{\omega}_{\widetilde{x}_{0},1-t}(1-s))
\end{pmatrix}\\
&= \widetilde{\omega}_{\widetilde{x}_{0},1-t}(1-s)= \widetilde{\omega}_{\widetilde{x}_{0},1-t}^{-1}(s).
\end{align*}
So $R_{\pi}'(\widetilde{\omega}_{\widetilde{x}_{0},t})=\widetilde{\omega}_{\widetilde{x}_{0},1-t}^{-1}$, from which we deduce the result for $\widetilde{x}_{0}$ using the homotopy $(R_{\theta}')_{\theta \in [0,\pi]}$. If $u=\widetilde{z}_{0}$, by the first part, $R\circ R_{0}'(\widetilde{\omega}_{\widetilde{x}_{0},t})=R(\widetilde{\omega}_{\widetilde{x}_{0},t})= \widetilde{\omega}_{\widetilde{z}_{0},t}$ and $R\circ R_{\pi}'(\widetilde{\omega}_{\widetilde{x}_{0},t})= R(\widetilde{\omega}_{\widetilde{x}_{0},1-t}^{-1})= \widetilde{\omega}_{\widetilde{z}_{0},1-t}^{-1}$ for all $t\in I$,  and the result then follows as in the case of $\widetilde{x}_{0}$.
\end{proof}

\subsection{Construction of intermediate homotopy fibres}\label{sec:notation}

Once more, the notation will be that defined in \resec{generalities}. We start by introducing some other notation that will be used in the rest of the paper. Let $n\geq n_{0}$, and let $j\in \brak{0,1,\ldots,n-n_{0}}$.   Let $\map{J_{j}}{\St}$ be the orientation-preserving homeomorphism defined by:
\begin{equation}\label{eq:defJj}
J_{j}(\theta, \phi)= \begin{cases}
\left(\theta, \phi+(\frac{\pi}{2}-\phi)\left(\frac{j}{j+1}\right)\right) & \text{if $\phi\geq 0$}\\
\left(\theta, \phi+(\frac{\pi}{2}+\phi)\left(\frac{j}{j+1}\right) \right) & \text{if $\phi< 0$.}
\end{cases}
\end{equation}
Note that $J_{0}=\id_{\St}$, $J_{j}(\widetilde{x}_{0})=\widetilde{x}_{j}$, $J_{j}$ sends the equator ($\phi=0$) to the circle of equation $\phi=\frac{\pi}{2}\left(\frac{j}{j+1}\right)$, leaves the longitudes invariant, and fixes the poles (see Figure~\ref{fig:mapJj0}).
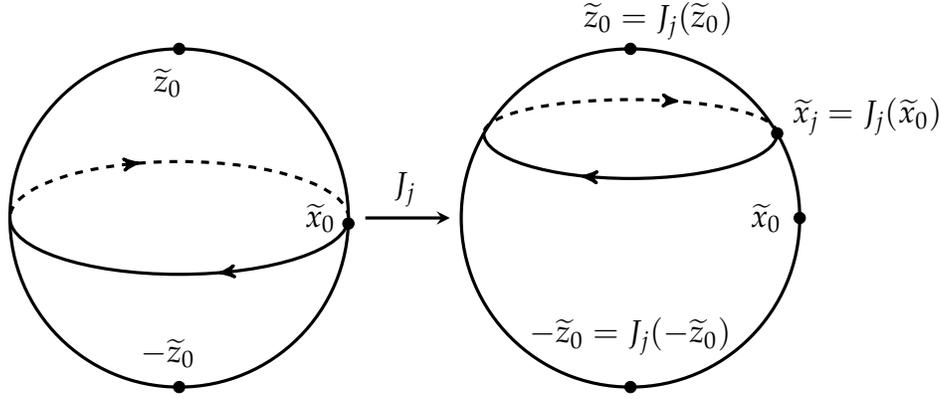
\begin{figure}[t]
\hspace*{\fill}
\begin{tikzpicture}[scale=0.75, >=stealth]

\draw[very thick] (0,0) circle(3);
\draw[dashed, very thick, middlearrow=0.65] (3,0) arc (0:180:3 and 1);
\draw[very thick, middlearrow=0.65] (-3,0) arc (180:360:3 and 1);

\draw[fill] (3,-0.1) circle(0.1);
\draw[fill] (0,3) circle(0.1);
\draw[fill] (0,-3) circle(0.1);
\node at (2.5,0){$\widetilde{x}_{0}$};
\node at (-0.2,2.4){$\widetilde{z}_{0}$};
\node at (-0.2,-2.4){$-\widetilde{z}_{0}$};


\draw[very thick] (8,0) circle(3);
\draw[dashed, very thick, middlearrow=0.4] (10.6,1.5) arc (0:180:2.6 and 0.6); 
\draw[very thick, middlearrow=0.4] (5.4,1.5) arc (180:360:2.6 and 0.8); 

\draw[fill] (11,0) circle(0.1);
\draw[fill] (8,3) circle(0.1);
\draw[fill] (8,-3) circle(0.1);
\draw[fill] (10.6,1.5) circle(0.1);

\node at (10.4,0){$\widetilde{x}_{0}$};
\node at (8.5,3.5){$\widetilde{z}_{0}=J_{j}(\widetilde{z}_{0})$};
\node at (8,-2.1){$-\widetilde{z}_{0}=J_{j}(-\widetilde{z}_{0})$};
\node at (12.2,1.8){$\widetilde{x}_{j}=J_{j}(\widetilde{x}_{0})$};

\draw[very thick, ->] (3.3,0) -- (4.8,0);
\node at (4,0.5){$J_{j}$};

%
%

\end{tikzpicture}
\hspace*{\fill}
\caption{The homeomorphism $J_{j}$.}
\label{fig:mapJj0}
\end{figure}
If $t\in I$,  let $\map{h_{j,t}}{I}[\St]$ be the path defined by:
\begin{equation}\label{eq:hxjt}
\text{$h_{j,t}(r)=\widetilde{\omega}_{\widetilde{x}_{j},t}\Bigl(\frac{1+r}{2}\Bigr)$ for all $r\in I$, where $\widetilde{\omega}_{\widetilde{x}_{j},t}= J_{j}\circ \widetilde{\omega}_{\widetilde{x}_{0},t}$.}
\end{equation}
The subspace $\mathbb{M}$ of $\St$ defined in \resec{generalities} is contractible. Let us define two explicit contractions
$\map{H_{1}}{\mathbb{M}\times I}[\St]$ and $\map{H_{2}}{\dd\times I}[\St]$ as follows. If $v\in \dd$, let $v=(\theta,\phi)$.
\begin{enumerate}[(a)]
\item To define $H_{1}$, we separate the cases $\phi\geq 0$ and $\phi<0$:
\begin{enumerate}[(i)]
\item suppose that $\phi\geq 0$. If $v=\widetilde{x}_{0}$ then set $H_{1}(v,r)=\widetilde{x}_{0}$ for all $r\in I$. If $v\neq \widetilde{x}_{0}$ then there exist $t\in (0,1)$ and $s\in [\frac{1}{2},1)$, both unique, such that $\widetilde{\omega}_{\widetilde{x}_{0},t}(s)=v$. We then set:
\begin{equation}\label{eq:defH1pos}
\text{$H_{1}(v,r)= \widetilde{\omega}_{\widetilde{x}_{0},t}(s(1-r)+r)$ for all $r\in I$.}
\end{equation}
\item suppose that $\phi<0$. Notice that $(\theta,0)=\widetilde{\omega}_{\widetilde{x}_{0},(2\pi-\theta)/2\pi}(\frac{1}{2})$, and that $\min\left(\frac{-2\phi}{\pi},\frac{1}{2}\right)\in (0,1)$. We set:
\begin{equation*}
H_{1}(v,r)=\begin{cases}\displaystyle
\biggl(\theta, \phi\biggl(1-\frac{r}{\min\bigl(\frac{-2\phi}{\pi},\frac{1}{2}\bigr)}\biggr)
\biggr) & \text{if $0\leq r\leq \min\left(\frac{-2\phi}{\pi},\frac{1}{2}\right)$}\\
\displaystyle\widetilde{\omega}_{\widetilde{x}_{0},(2\pi-\theta)/2\pi}\Biggl( \frac{1+r-2\min\bigl(\frac{-2\phi}{\pi},\frac{1}{2}\bigr)}{2\bigl( 1-\min\bigl(\frac{-2\phi}{\pi},\frac{1}{2}\bigr)\bigr)}\Biggr) & \text{if $\min\left(\frac{-2\phi}{\pi},\frac{1}{2}\right)<r\leq 1$.}
\end{cases}
\end{equation*}
\end{enumerate}
One may check that $H_{1}$ is continuous, and that $H_{1}(v,0)=v$ and $H_{1}(v,1)=\widetilde{x}_{0}$ for all $v\in \mathbb{M}$. 

Geometrically, $H_{1}$ contracts $\mathbb{M}$ onto $\widetilde{x}_{0}$, first along the longitudes in the lower hemisphere, then along the paths $(\widetilde{\omega}_{\widetilde{x}_{0}, t}(s))_{s\in [\frac{1}{2},1]}$ in the upper hemisphere.

\item $H_{2}$ is defined by $H_{2}(v,r)=H_{2}((\theta,\phi),r) = \left(\theta, (1-r)\phi +\frac{r\pi}{2} \right)$. Geometrically, $H_{2}$ contracts $\dd$ onto $\widetilde{z}_{0}$ along the longitudes.
\end{enumerate}
If $v\in \mathbb{M}$ (resp.\ $v\in \dd$), we define $\map{h_{v}^{(1)}}{I}[\St]$ (resp.\ $\map{h_{v}^{(2)}}{I}[\St]$) by:
\begin{equation}\label{eq:defh1}
\text{$h_{v}^{(1)}(r)=H_{1}(v,r)$ and $h_{v}^{(2)}(r)=H_{2}(v,r)$.}
\end{equation}
Observe that $h_{v}^{(1)}$ (resp.\ $h_{v}^{(2)}$) is a path that joins $v$ to $\widetilde{x}_{0}$ (resp.\ to $\widetilde{z}_{0}$) via the homotopy $H_{1}$ (resp.\ $H_{2}$). Let $\map{\Psi_{j}}{\mathbb{M}\times I}[\St]$ be the homotopy defined by $\Psi_{j}(\widetilde{x},t)= J_{j}(H_{1}(J_{j}^{-1}(\widetilde{x}),t))$ using~\reqref{defJj}, and for $\widetilde{x}\in \mathbb{M}$, let $\map{\psi_{\widetilde{x}}^{(j)}}{I}[\St]$ be the path defined by $\psi_{\widetilde{x}}^{(j)}(r)=\Psi_{j}(\widetilde{x},r)$ that joins $\widetilde{x}$ to $\widetilde{x}_{j}$. Note that $\psi_{\widetilde{x}}^{(0)}= h_{\widetilde{x}}^{(1)}$, and that for all $r\in I$:
\begin{align}\label{eq:psihone}
\psi^{(j)}_{\widetilde{\omega}_{\widetilde{x}_{j},t}(\frac{1}{2})}(r) &= \textstyle\Psi_{j}(\widetilde{\omega}_{\widetilde{x}_{j},t}(\frac{1}{2}),r)=  J_{j}(H_{1}(J_{j}^{-1}(\widetilde{\omega}_{\widetilde{x}_{j},t}(\frac{1}{2})),r))\notag\\
&= \textstyle J_{j}(H_{1}(\widetilde{\omega}_{\widetilde{x}_{0},t}(\frac{1}{2}),r))=\textstyle J_{j}(\widetilde{\omega}_{\widetilde{x}_{0},t}(\frac{1+r}{2}))\; \mbox{by equations~\reqref{hxjt} and~\reqref{defH1pos}}\notag\\
& =h_{j,t}(r)\; \mbox{by \req{hxjt}}.
\end{align}
Let $\map{\alpha_{c}}{I_{c}}[I_{n}'']$ be the pointed map defined by: 
\begin{equation}\label{eq:defalphac}
\alpha_{c}(y_{1},\ldots, y_{n-1}, \omega_{1},\ldots, \omega_{n-1})= \bigl(y_{1}, \ldots, y_{n-1}, \widehat{\psi}_{y_{1}}^{(1)}\ast \omega_{1}, \ldots, \widehat{\psi}_{y_{n-1}}^{(n-1)}\ast \omega_{n-1} \bigr),
\end{equation}
where:
\begin{equation}\label{eq:psihat}
\widehat{\psi}_{y_{i}}^{(i)} =\begin{cases}
h_{y_{n-1}}^{(2)} & \text{if $M=\St$ and $i=n-1$}\\
\psi_{y_{i}}^{(i-1)} & \text{otherwise.}
\end{cases}
\end{equation}
Note that $\alpha_{c}$ satisfies $h_{n}''\circ \alpha_{c}(u,\mu)=u$ for all $(u,\mu)\in I_{c}$, and so  diagram~\reqref{bigcommdiagfib} is commutative up to homotopy using \rerem{In2null}. 
From the definition of $\map{\alpha_{n}'}{I_{n}'}[I_{n}]$,~\reqref{defalphapi} and~\reqref{defalphac}, it follows that:
\begin{multline}\label{eq:compalphas}
\alpha_{n}'\circ \alpha_{\pi} \circ \alpha_{c}(y_{1},\ldots, y_{n-1}, \omega_{1},\ldots, \omega_{n-1})=\\
\bigl(\widehat{\pi}(y_{1}),\ldots, \widehat{\pi}(y_{n-1}), w_{n}, \widehat{\pi}\circ(\widehat{\psi}_{y_{1}}^{(1)} \ast \omega_{1}), \ldots, \widehat{\pi}\circ(\widehat{\psi}_{y_{n-1}}^{(n-1)}\ast \omega_{n-1}),c_{w_{n}}\bigr).
\end{multline}

The homomorphism $\map{J_{j}}{\St}$ defined by \req{defJj} and the covering map $\map{\pi}{\St}[\rp]$ induce isomorphisms on the $\pi_{2}$-level, so by \req{hxjt}, 
$(\widehat{\pi}\circ \widetilde{\omega}_{\widetilde{x}_{j},t})_{t\in I}$ is a geometric representative of a generator of $\pi_{1}(\Omega(\St),c_{\widetilde{x}_{j}})$ (resp.\ of $\pi_{1}(\Omega(\rp),c_{x_{j}})$) that we denote by $\widetilde{\lambda}_{\widetilde{x}_{j}}$ (resp.\ by $\lambda_{x_{j}}$). 
Identifying $\pi_{1}(\Omega(\St),c_{\widetilde{x}_{j}})$ (resp.\ $\pi_{1}(\Omega(\rp),c_{x_{j}})$) with $\pi_{2}(\St,\widetilde{x}_{j})$ (resp.\ $\pi_{2}(\rp,x_{j})$) as usual, we may thus take:
\begin{equation}\label{eq:basispi2}
\text{$\bigl(\widetilde{\lambda}_{\widetilde{x}_0}, \widetilde{\lambda}_{\widetilde{x}_1},\ldots, \widetilde{\lambda}_{\widetilde{x}_{n-3}}, \widetilde{\lambda}_{\widetilde{z}_0}, \widetilde{\lambda}_{-\widetilde{z}_0}\bigr)$ (resp.\ $\bigl(\lambda_{x_0}, \lambda_{x_1},\ldots, \lambda_{x_{n-2}}, \lambda_{z_0}\bigr)$)}
\end{equation}
to be a basis of the free Abelian group $\pi_{2}(\prod_{1}^{n}\, \St)$ (resp.\ of $\pi_{2}(\prod_{1}^{n}\, \rp)$). For each $v\in \brak{\widetilde{x}_{0},\widetilde{x}_{1}, \ldots, \widetilde{x}_{n-3}, \widetilde{z}_{0},-\widetilde{z}_{0}}$ (resp.\ $v\in \brak{x_{0},x_{1}, \ldots, x_{n-2}, z_{0}}$), let
\begin{equation}\label{eq:defdeltav}
\text{$\widetilde{\delta}_{v}=\partial_{n}(\widetilde{\lambda}_{v})$ (resp.\ $\delta_{v}=\partial_{n}(\lambda_{v})$).}
\end{equation}

By \rerem{In2null} and the proof of \reth{prop5}, the homotopy fibres $I_{c}$, $I_{n}''$ and $I_{n}'$ have the same homotopy type. We are now able to exhibit explicit homotopy equivalences between them, which will allow us to describe the images of certain elements of $\pi_{1}(I_{c})$ by the induced isomorphisms.

\begin{lem}\mbox{}\label{lem:fundhomoequiv}
\begin{enumerate}[(a)]
\item\label{it:fundhomoequiva} The map $\map{\alpha_{c}}{I_{c}}[I_{n}'']$ defined in~\reqref{defalphac} is a homotopy equivalence.
\item\label{it:fundhomoequivb} The map $\map{\alpha_{\pi}}{I_{n}''}[I_{n}']$ defined in~\reqref{defalphapi} is a homotopy equivalence.

\item\label{it:fundhomoequivc} With the notation of~\reqref{compalphas}, 
\begin{equation*}
g_{n}\circ j_{n}\circ \alpha_{n}'\circ \alpha_{\pi} \circ \alpha_{c}(y_{1},\ldots, y_{n-1}, \omega_{1},\ldots, \omega_{n-1})=\\
\bigl(\widehat{\pi}(y_{1}),\ldots, \widehat{\pi}(y_{n-1}), w_{n}).
\end{equation*}
\item\label{it:fundhomoequivd} For all $i=1,\ldots,n-1$, $(\alpha_{n}'\circ \alpha_{\pi} \circ \alpha_{c})_{\#}(\widetilde{\lambda}_{w_{i}''})= \begin{cases}
\widetilde{\delta}_{w_{i}}=\partial_{n}(\widetilde{\lambda}_{w_{i}})  & \text{if $M=\St$}\\
\delta_{w_{i}} =\partial_{n}(\lambda_{w_{i}})  & \text{if $M=\rp$.}\end{cases}$
\end{enumerate}
\end{lem}

\begin{proof}\mbox{}
\begin{enumerate}[(a)]
\item Let $\map{\beta_{c}}{I_{n}''}[I_{c}]$ be defined by:
\begin{equation*}
\beta_{c}(v_{1},\ldots,v_{n-1}, \gamma_{1},\ldots,\gamma_{n-1}) =\bigl(v_{1},\ldots,v_{n-1}, \bigl(\widehat{\psi}_{v_{1}}^{(1)}\bigr)^{-1}\ast \gamma_{1}, \ldots,  \bigl(\widehat{\psi}_{v_{n-1}}^{(n-1)}\bigr)^{-1}\ast \gamma_{n-1}\bigr).
\end{equation*}
Then $\alpha_{c}\circ \beta_{c}\simeq \id_{I_{n}''}$ and $\beta_{c}\circ \alpha_{c} \simeq \id_{I_{c}}$, so $\alpha_{c}$ and $\beta_{c}$ are mutual homotopy inverses between $I_{c}$ and $I_{n}''$. 

\item If $M=\St$ then $\alpha_{\pi}$ is the identity, and the result is clear. So suppose that $M=\rp$. Define $\map{\beta_{\pi}}{I_{n}'}[I_{n}'']$ by:
\begin{equation*}
\beta_{\pi}(v_{1},\ldots,v_{n-1},l_{1},\ldots,l_{n-1})=(\widetilde{v}_{1},\ldots,\widetilde{v}_{n-1}, \widetilde{l}_{1}, \ldots, \widetilde{l}_{n-1}),
\end{equation*}
where for $i=1,\ldots, n-1$, $\map{\widetilde{l}_{i}}{I}[\St]$ is the lift of the path $\map{l_{i}}{I}[\rp]$ by the covering map $\pi$ such that $\widetilde{l}_{i}(1)= w_{i}''$, and $\widetilde{v}_{i}$ is given by $\widetilde{v_{i}}=\widetilde{l}_{i}(0)$. The map $\beta_{\pi}$ is well defined, since if $(v_{1},\ldots,v_{n-1})\in F_{n-1}(\rp\setminus \brak{z_{0}})$ then $(\widetilde{v}_{1},\ldots,\widetilde{v}_{n-1})\in \orbconf[n-1]$. The fact that $\pi$ is a covering map implies that $\alpha_{\pi}\circ \beta_{\pi}=\id_{I_{n}'}$ and $\beta_{\pi}\circ\alpha_{\pi} =\id_{I_{n}''}$, so $\alpha_{\pi}$ and $\beta_{\pi}$ are mutual homotopy inverses between $I_{n}''$ and $I_{n}'$.

\item This follows directly from~\reqref{compalphas} and the definition of the maps $g_{n}$ and $j_{n}$ given in \resec{generalities}.

\item Let $i\in \brak{1,\ldots, n-1}$. Then:
\begin{equation}\label{eq:imlambdaprime2}
\widetilde{\lambda}_{w_{i}''}=\bigl[\bigl(w_{1}'',\ldots, w_{n-1}'', c_{w_{1}''}, \ldots, c_{w_{i-1}''}, \widetilde{\omega}_{w_{i}'',t}, c_{w_{i+1}''}, \ldots, c_{w_{n-1}''}\bigr)_{t\in I}\bigr]
\end{equation}
in $\pi_{1}(I_{c})$. For all $j=1,\ldots,n-1$, $\widehat{\pi}(w_{j}'')=w_{j}$, and $\widehat{\psi}_{w_{j}''}^{j-1}=c_{w_{j}}$ using~\reqref{psihat} and the definition of $\Psi_{j-1}$. It follows from~\reqref{defdeltav},~\reqref{imlambdaprime2} and \relem{boundary} that:
\begin{align*}
(\alpha_{n}'\circ \alpha_{\pi} \circ \alpha_{c})_{\#}(\widetilde{\lambda}_{w_{i}''})&=\bigl[\bigl(w_{1},\ldots, w_{n-1}, c_{w_{1}}, \ldots, c_{w_{i-1}}, \widehat{\pi} \circ \widetilde{\omega}_{w_{i}'',t}, c_{w_{i+1}}, \ldots, c_{w_{n-1}}\bigr)_{t\in I}\bigr]\\
&= \begin{cases}
\partial_{n}(\widetilde{\lambda}_{w_{i}})=\widetilde{\delta}_{w_{i}}  & \text{if $M=\St$}\\
\partial_{n}(\lambda_{w_{i}})=\delta_{w_{i}}  & \text{if $M=\rp$,}
\end{cases}
\end{align*}
as required.\qedhere
\end{enumerate}
\end{proof}

The following proposition relates the elements that appear in~\reqref{basispi2} to an element $\widehat{\tau}_{n}$ of $\pi_{1}(I_{n})$ that projects to the full twist braid under the homomorphism $(g_{n}\circ j_{n})_{\#}$.

\begin{prop}\label{prop:deftaunhat}
Let $M=\St$ or $\rp$, let $n\geq n_{0}$,
let $\map{\tau_{n}}{I}[I_{c}]$ be the loop in $I_{c}$ based at $W_{c}$ defined by:
\begin{equation}\label{eq:deftaun}
\!\!\!\tau_{n}(t) \!=\!\begin{cases}
\textstyle\bigl(\widetilde{\omega}_{\widetilde{x}_{0},t}(\frac{1}{2}), \widetilde{\omega}_{\widetilde{x}_{1},t}(\frac{1}{2}), \ldots, \widetilde{\omega}_{\widetilde{x}_{n-3},t}(\frac{1}{2}), \widetilde{z}_{0}, c_{\widetilde{x}_{0}},c_{\widetilde{x}_{1}},\ldots, c_{\widetilde{x}_{n-3}}, c_{\widetilde{z}_{0}}\bigr) & \!\text{if $M=\St$}\\
\textstyle\bigl(\widetilde{\omega}_{\widetilde{x}_{0},t}(\frac{1}{2}), \widetilde{\omega}_{\widetilde{x}_{1},t}(\frac{1}{2}), \ldots, \widetilde{\omega}_{\widetilde{x}_{n-2},t}(\frac{1}{2}), c_{\widetilde{x}_{0}}, c_{\widetilde{x}_{1}},\ldots, c_{\widetilde{x}_{n-2}}\bigr) & \!\text{if $M=\rp$,}
\end{cases}
\end{equation}
and let $\widehat{\tau}_{n}$ be the element of $\pi_{1}(I_{n})$ defined by $\widehat{\tau}_{n}=[\alpha_{n}'\circ \alpha_{\pi}\circ \alpha_{c}(\tau_{n})]$. Then 
$(g_{n}\circ j_{n})_{\#}(\widehat{\tau}_{n})=\ft$, the full-twist braid of $P_{n}(M)$, and there exist unique integers $m_{0},m_{1},\ldots,m_{n-1}\in \Z$ such that:
\begin{equation}\label{eq:commdeltanS2}
\widehat{\tau}_{n}^{2}=\begin{cases}
\partial_{n}(m_{0}\widetilde{\lambda}_{\widetilde{x}_0}+\cdots+m_{n-3} \widetilde{\lambda}_{\widetilde{x}_{n-3}} +m_{n-2} \widetilde{\lambda}_{\widetilde{z}_{0}}+m_{n-1} \widetilde{\lambda}_{-\widetilde{z}_{0}}) &\text{if $M=\St$}\\
\partial_{n}(m_{0}\lambda_{x_0}+\cdots+m_{n-2} \lambda_{x_{n-2}} +m_{n-1} \lambda_{z_{0}}) & \text{if $M=\rp$.}
\end{cases}
\end{equation}
\end{prop}

\begin{proof}
Let $\tau_{n}$ be as defined in \req{deftaun}, and for $t\in I$, let:
\begin{equation}\label{eq:defyns}
(y_{1}(t),\ldots, y_{n-1}(t))=\begin{cases}
\bigl(\widetilde{\omega}_{\widetilde{x}_{0},t}(\frac{1}{2}), \widetilde{\omega}_{\widetilde{x}_{1},t}(\frac{1}{2}), \ldots, \widetilde{\omega}_{\widetilde{x}_{n-3},t}(\frac{1}{2}), \widetilde{z}_{0}\bigr) & \text{if $M=\St$}\\
\bigl(\widetilde{\omega}_{\widetilde{x}_{0},t}(\frac{1}{2}), \widetilde{\omega}_{\widetilde{x}_{1},t}(\frac{1}{2}), \ldots, \widetilde{\omega}_{\widetilde{x}_{n-2},t}(\frac{1}{2})\bigr) & \text{if $M=\rp$.}
\end{cases}
\end{equation}
Then $\tau_{n}(t)=(y_{1}(t),\ldots, y_{n-1}(t), c_{y_{1}(0)}, \ldots, c_{y_{n-1}(0)})$, and using~\reqref{compalphas}, we have: 
\begin{align}
\widehat{\tau}_{n}&= [\alpha_{n}'\circ \alpha_{\pi}\circ \alpha_{c}\circ \tau_{n}]\notag\\
&= 
\bigl[\bigl(\widehat{\pi}(y_{1}(t)),\ldots, \widehat{\pi}(y_{n-1}(t)), w_{n}, \widehat{\pi}\circ\widehat{\psi}_{y_{1}(t)}^{(1)}, \ldots, \widehat{\pi}\circ\widehat{\psi}_{y_{n-1}(t)}^{(n-1)},c_{w_{n}}\bigr)_{t\in I}\bigr].\label{eq:tauhat}
\end{align}
Using~\reqref{psihat},~\reqref{tauhat}, \relem{fundhomoequiv}(\ref{it:fundhomoequivc}) and the fact that $h_{\widetilde{z}_{0}}^{(2)}=c_{\widetilde{z}_{0}}$, we see that:
\begin{equation*}
\!\!\!(g_{n}\circ j_{n})_{\#} (\widehat{\tau}_{n})
\!=\!
\begin{cases}
\textstyle\left[\left(\widetilde{\omega}_{\widetilde{x}_{0},t}(\frac{1}{2}), \widetilde{\omega}_{\widetilde{x}_{1},t}(\frac{1}{2}), \ldots, \widetilde{\omega}_{\widetilde{x}_{n-3},t}(\frac{1}{2}), \widetilde{z}_{0}, -\widetilde{z}_{0}\right)_{t\in I} \right] & \hspace*{-3.5mm}\text{if $M\!=\!\St$}\\
\left[\left(\pi(\widetilde{\omega}_{\widetilde{x}_{0},t}(\textstyle\frac{1}{2})), \pi(\widetilde{\omega}_{\widetilde{x}_{1},t}(\frac{1}{2})), \ldots, \pi(\widetilde{\omega}_{\widetilde{x}_{n-2},t}(\frac{1}{2})), z_{0}\right)_{t\in I}\right] & \hspace*{-3.5mm}\text{if $M\!=\!\rp$.}
\end{cases} 
\end{equation*}
Using the constructions of \resec{geomrepsS2} and~\req{hxjt}, this element may be seen to be a geometric representative of $\ft$ in $P_{n}(M)$ (see Figure~\ref{fig:fulltwist} for the case $M=\St$). 
\begin{figure}[t]
\hfill
\begin{tikzpicture}[scale=0.6, >=stealth]
\draw[very thick] (0,0) circle(4.5);
\draw[dotted,very thick, middlearrow=0.35] (4.5,0) arc (0:180:4.5 and 1);
\draw[very thick, middlearrow=0.35] (-4.5,0) arc (180:360:4.5 and 1);
\draw[fill] (4.5,-0.1) circle(0.1);
\node at (5.1,0){$\widetilde{x}_{0}$};

\begin{scope}[shift={(0,4.5*sin(15))},scale=cos(15)]
\draw[dotted,very thick, middlearrow=0.35] (4.5,0) arc (0:180:4.5 and 1);
\draw[very thick, middlearrow=0.35] (-4.5,0) arc (180:360:4.5 and 1);
\node at (5.2,0){$\widetilde{x}_{1}$};
\draw[fill] (4.5,-0.1) circle(0.1);
\end{scope}

\begin{scope}[shift={(0,4.5*sin(27))},scale=cos(27)]
\draw[dotted,very thick, middlearrow=0.35] (4.5,0) arc (0:180:4.5 and 1);
\draw[very thick, middlearrow=0.35] (-4.5,0) arc (180:360:4.5 and 1);
\node at (5.2,0){$\widetilde{x}_{2}$};
\draw[fill] (4.52,-0.1) circle(0.11);
\end{scope}

\begin{scope}[shift={(0,4.5*sin(55))},scale=cos(55)]
\draw[dotted,very thick, middlearrow=0.35] (4.5,0) arc (0:180:4.5 and 0.6);
\draw[very thick, middlearrow=0.35] (-4.5,0) arc (180:360:4.5 and 1);
\node at (6.4,0){$\widetilde{x}_{n-3}$};
\draw[fill] (4.55,-0.1) circle(0.17);
\end{scope}

\draw[fill] (0,4.5) circle(0.1);
\draw[fill] (0,-4.5) circle(0.1);
\node at (0.6,5){$\widetilde{z}_{0}$};
\node at (0.6,-5){$-\widetilde{z}_{0}$};
\node[rotate=-20] at (4.2,3){$\ddots$};
\node at (1.3,-1.5){$t$};
\end{tikzpicture}
\hspace*{\fill}
\caption{The full twist braid $\ft$ of $P_{n}(\St)$, which is equal to $(g_{n}\circ j_{n})_{\#}(\widehat{\tau}_{n})$.} 
\label{fig:fulltwist}
\end{figure}
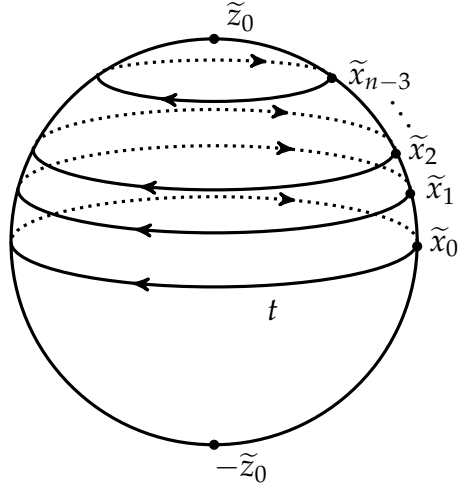
In particular, $\widehat{\tau}_{n}$ is non trivial, and $\widehat{\tau}_{n}^{2}\in \im{\partial_{n}}$ by exactness of~\reqref{lespartialn}. With respect to the basis of $\pi_{2}(\prod_{1}^{n}\, M)$ given by~\reqref{basispi2},
the existence and uniqueness of the integers $m_{0},m_{1},\ldots, m_{n-1}$ given in~\reqref{commdeltanS2} also follows from the exactness of~\reqref{lespartialn}.
\end{proof}

\subsection{The boundary homomorphism in the case $n=n_{0}$}\label{sec:n3S2}

In this section, we assume that $n=n_{0}$, and we use of the notation of Sections~\ref{sec:generalities} and~\ref{sec:notation}. The aim is to describe the boundary homomorphism $\map{\partial_{n_{0}}}{\pi_{2}(\prod_{i=1}^{n_{0}}\, M)}[\pi_{1}(I_{n_{0}})]$ that arises in the exact sequence~\reqref{lespartialn}. 
It follows from \reth{prop5} and the fact that $\orbconf[1]=F_{1}(C)$ that $\pi_{1}(I_{n_{0}})\cong \Z^{n_{0}}$. In particular, $\pi_{1}(I_{n_{0}})$ is Abelian, which justifies the use of additive notation for this group in what follows.
As we mentioned just after~\reqref{defInprime2}, if $M=\St$, $\iota_{n_{0}}''=\iota_{n_{0}}'$ and $I_{n_{0}}''=I_{n_{0}}'$. If $M=\rp$, since $n_{0}-1=1$, we have $\orbconf[n_{0}-1]=F_{n_{0}-1}(C)$. So in both cases, $F_{n_{0}-1}^{G}(\mathbb{M})=F_{n_{0}-1}(\mathbb{M})$. If $M=\St$ (resp.\ $M=\rp$), in \resec{geomreprS2}, we use the explicit homotopy equivalences of \relem{fundhomoequiv} and \reco{alphanprime} between $I_{c},I_{n_{0}}''$, $I_{n_{0}}'$ and $I_{n_{0}}$ to determine a generating set of 
$\pi_{1}(I_{n_{0}})$. The main result of this section is \reth{brrelnS2} that describes the relation between the images under the boundary operator $\partial_{n_{0}}$ of the elements $\widetilde{\lambda}_{u}$ (resp.\ $\lambda_{u}$), where $u\in \brak{\widetilde{x}_{0}, \widetilde{z}_{0},-\widetilde{z}_{0}}$ (resp.\ $u\in \brak{x_{0}, z_{0}}$), and the element $\widehat{\tau}_{n_{0}}$ defined in \repr{deftaunhat}. The proof of \reth{brrelnS2} 
is given in \resec{n3S2a} (resp.~\resec{rp2nequals2}), and the proof for $\rp$ will make use of that for $\St$.




\subsubsection{A basis of $\pi_{1}(I_{n_{0}})$}\label{sec:geomreprS2}




In this section, we start by exhibiting a basis for $\pi_{1}(I_{n_{0}})$ in \repr{genspi1S2} using the structure of $I_{c}$ and the homotopy equivalences of \relem{fundhomoequiv} and \reco{alphanprime}. We then state \reth{brrelnS2} that describes the relations between the elements of $\pi_{1}(I_{n_{0}})$ given by~\reqref{defdeltav}. If $M=\St$ (resp.\ $\rp$), it will suffice to determine $\widetilde{\delta}_{-\widetilde{z}_{0}}$ (resp.\ $\delta_{z_{0}}$) in terms of the elements of this basis. This will enable us to understand more fully the exact sequence~\reqref{lespartialn}.

\begin{prop}\label{prop:genspi1S2} 
If $M=\St$ (resp.\ $\rp$), a basis of $\pi_{1}(I_{n_{0}})$ is given by $\bigl\langle\widetilde{\delta}_{\widetilde{x}_{0}},\widetilde{\delta}_{\widetilde{z}_{0}}, \br{\widetilde{x}_{0}, \widetilde{z}_{0}, -\widetilde{z}_{0}} \bigr\rangle$ (resp.\ by $\ang{\delta_{x_{0}}, \br{x_{0},{z}_{0}}}$).
\end{prop}

\begin{proof}
As we saw above, $F_{n_{0}-1}^{G}(\mathbb{M})=F_{n_{0}-1}(\mathbb{M})$, so $I_{c}=F_{n_{0}-1}(\mathbb{M}) \times \prod_{i=1}^{n_{0}-1}\, \Omega(\St)$ by~\rerems{homfibre}(\ref{it:homfibreb}), and a basis of $\pi_{1}(I_{c})\cong \Z^{n_{0}}$ is given by $\brak{[\tau_{n_{0}}]} \cup \bigl\{\widetilde{\lambda}_{w_{i}''}\bigr\}_{i=1}^{n_{0}-1}$. The result follows by applying \relem{fundhomoequiv}(\ref{it:fundhomoequivd}) and \repr{deftaunhat} to this basis, 
and using the fact that the homomorphism $\map{(\alpha_{n_{0}}'\circ \alpha_{\pi}\circ \alpha_{c})_{\#}}{\pi_{1}(I_{c})}[\pi_{1}(I_{n_{0}})]$ is an isomorphism by \relem{fundhomoequiv} and \reco{alphanprime}. 
\end{proof}

Before stating \reth{brrelnS2}, we introduce some notation that will be useful in what follows. For $u\in \brak{\pm\widetilde{x}_{0},\pm\widetilde{z}_{0}}$ and for all $r,s,t\in I$, let $\map{\Omega_{u,t,s}, \Omega_{u,t,s}'}{I}[\St]$ be the families of arcs defined by:
\begin{equation}\label{eq:defOmega}
\Omega_{u,t,s}(r)= \widetilde{\omega}_{u,t}(s+r(1-s))\quad\text{and}\quad \Omega_{u,t,s}'(r)= \widetilde{\omega}_{u,t}(s(1-r)).
\end{equation}
The arcs $\Omega_{\widetilde{x}_{0},t,s}$ and $\Omega_{\widetilde{x}_{0},t,s}'$ are illustrated in Figure~\ref{fig:omega}.
\begin{figure}[t]
\hfill
\begin{tikzpicture}[scale=0.5, >=stealth]
\draw (0,0) circle(4.5);
\draw[dotted] (4.5,0) arc (0:180:4.5 and 1);
\draw (-4.5,0) arc (180:360:4.5 and 1);
\draw[very thick, middlearrow=0.35] (-2.5,-0.8) .. controls (-2.25,-3.75) and (4.2,-3.6) .. (4.45,0);
\draw[very thick, middlearrowrev=0.65] (-1.5,1.6) .. controls (-1,2.5) and (4.2,3.6) .. (4.45,0);
\draw[very thick] (-2.5,-0.8) .. controls (-2.5,0.5) and (-2,1.2) .. (-1.5,1.6);

\draw[fill] (4.5,-0.1) circle(0.15);
\draw[fill] (0,4.5) circle(0.15);
\draw[fill] (0,-4.5) circle(0.15);
\draw[fill] (-1.5,1.6) circle(0.15);
\node at (5.2,0){$\widetilde{x}_{0}$};
\node at (0.6,5){$\widetilde{z}_{0}$};
\node at (0.6,-5){$-\widetilde{z}_{0}$};
\node at (-1.2,-3.3){$\widetilde{\omega}_{\widetilde{x}_{0},t}$};
\node at (0.8,1.6){$\widetilde{\omega}_{\widetilde{x}_{0},t}(s)$};
\end{tikzpicture}
\hspace*{\fill}
\begin{tikzpicture}[scale=0.5, >=stealth]
\draw (0,0) circle(4.5);
\draw[dotted] (4.5,0) arc (0:180:4.5 and 1);
\draw (-4.5,0) arc (180:360:4.5 and 1);
\draw (-2.5,-0.8) .. controls (-2.25,-3.75) and (4.2,-3.6) .. (4.45,0);
\draw[very thick, middlearrowrev=0.65] (-1.5,1.6) .. controls (-1,2.5) and (4.2,3.6) .. (4.45,0);
\draw (-2.5,-0.8) .. controls (-2.5,0.5) and (-2,1.2) .. (-1.5,1.6);

\draw[fill] (4.5,-0.1) circle(0.15);
\draw[fill] (0,4.5) circle(0.15);
\draw[fill] (0,-4.5) circle(0.15);
\draw[fill] (-1.5,1.6) circle(0.15);
\node at (5.2,0){$\widetilde{x}_{0}$};
\node at (0.6,5){$\widetilde{z}_{0}$};
\node at (0.6,-5){$-\widetilde{z}_{0}$};
\node at (-0.5,3){$\Omega_{\widetilde{x}_{0},t,s}$};
\end{tikzpicture}
\hspace*{\fill}
\begin{tikzpicture}[scale=0.5, >=stealth]
\draw (0,0) circle(4.5);
\draw[dotted] (4.5,0) arc (0:180:4.5 and 1);
\draw (-4.5,0) arc (180:360:4.5 and 1);
\draw[very thick, middlearrowrev=0.35] (-2.5,-0.8) .. controls (-2.25,-3.75) and (4.2,-3.6) .. (4.45,0);

\draw (-1.5,1.6) .. controls (-1,2.5) and (4.2,3.6) .. (4.45,0);
\draw[very thick] (-2.5,-0.8) .. controls (-2.5,0.5) and (-2,1.2) .. (-1.5,1.6);

\draw[fill] (4.5,-0.1) circle(0.15);
\draw[fill] (0,4.5) circle(0.15);
\draw[fill] (0,-4.5) circle(0.15);
\draw[fill] (-1.5,1.6) circle(0.15);
\node at (5.2,0){$\widetilde{x}_{0}$};
\node at (0.6,5){$\widetilde{z}_{0}$};
\node at (0.6,-5){$-\widetilde{z}_{0}$};
\node at (1,-2){$\Omega_{\widetilde{x}_{0},t,s}'$};
\end{tikzpicture}
\hspace*{\fill}
\caption{The arcs $\widetilde{\omega}_{\widetilde{x}_{0},t}$, $\Omega_{\widetilde{x}_{0},t,s}$ and $\Omega_{\widetilde{x}_{0},t,s}'$.}
\label{fig:omega}
\end{figure}
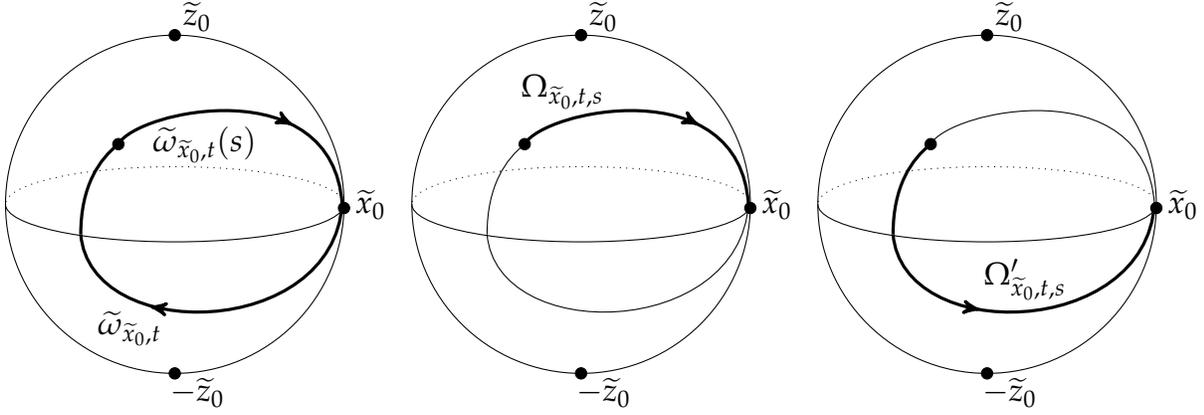
Then $\Omega_{u,t,0}=\widetilde{\omega}_{u,t}$, $\Omega'_{u,t,1}=\widetilde{\omega}_{u,t}^{-1}$, $\Omega_{u,t,1}= \Omega_{u,t,0}'=c_{u}$, and for all $s,t\in I$, $\Omega_{u,t,s}$ and $\Omega_{u,t,s}'$ are subarcs of $\widetilde{\omega}_{u,t}$ that join $\widetilde{\omega}_{u,t}(s)$ to $u$. Furthermore, 
equations~\reqref{minusomega},~\reqref{hxjt},~\reqref{defH1pos}, and~\reqref{defh1} imply that:
\begin{align}
\Omega_{\widetilde{x}_{0},t,\frac{1}{2}}&= h_{\widetilde{\omega}_{\widetilde{x}_{0},t}\left(\frac{1}{2}\right)}^{(1)}= h_{0,t}\label{eq:Omegah}\\
\Omega_{u,t,s}&=-\Omega_{-u,1-t,s} \quad\text{and}\quad \Omega_{u,t,s}'=-\Omega_{-u,1-t,s}'.\label{eq:minusOmega}
\end{align}
Using~\reqref{defyns},~\reqref{tauhat} and the fact that $h_{\widetilde{z}_{0}}^{(2)}=c_{\widetilde{z}_{0}}$,
we have:
\begin{equation}
\widehat{\tau}_{n_{0}}
= \begin{cases}
\bigl[\bigl(\widetilde{\omega}_{\widetilde{x}_{0},t}({\textstyle\frac{1}{2}}), \widetilde{z}_{0}, -\widetilde{z}_{0}, \Omega_{\widetilde{x}_{0},t,\frac{1}{2}}, c_{\widetilde{z}_{0}}, c_{-\widetilde{z}_{0}}\bigr)_{t\in I}\bigr]& \text{if $M=\St$, and}\\
\bigl[\bigl(\pi\circ \widetilde{\omega}_{\widetilde{x}_{0},t}\left({\textstyle\frac{1}{2}}\right), {z}_{0}, \pi\circ \Omega_{\widetilde{x}_{0},t,\frac{1}{2}}, c_{{z}_{0}}\bigr)_{t\in I}\bigr]& \text{if $M=\rp$.}
\end{cases}\label{eq:omegah}
\end{equation}
By \relem{boundaryS2}(\ref{it:boundaryS2a}) and \repr{genspi1S2}, if $M=\St$, the loop $\alpha_{n_{0}}'\circ \alpha_{\pi} \circ \alpha_{c}\circ \tau_{n_{0}}$, illustrated in Figure~\ref{fig:looptauS2}, and the images by $\alpha_{n_{0}}'\circ \alpha_{\pi} \circ \alpha_{c}$ of the loops $\bigl(\widetilde{x}_0, \widetilde{z}_0, \widetilde{\omega}_{\widetilde{x}_{0}, t}, c_{\widetilde{z}_0}\bigr)_{t\in I}$ and $\bigl(\widetilde{x}_0, \widetilde{z}_0, c_{\widetilde{x}_0}, \widetilde{\omega}_{\widetilde{z}_{0}, t}\bigr)_{t\in I}$
are geometric representatives of $\widetilde{\tau}_{n}$, $\widetilde{\delta}_{\widetilde{x}_{0}}$ and $\widetilde{\delta}_{\widetilde{z}_{0}}$ respectively. If $M=\rp$, geometric representatives of the basis elements $(\alpha_{c})_{\#}(\widetilde{\lambda}_{\widetilde{x}_{0}})$ and $ [\alpha_{c} \circ \tau_{n_{0}}]$ of $\pi_{1}(I_{n_{0}}'')$ are similar in nature to those 
for the case $M=\St$,
except that the markings on $\widetilde{z}_{0}$ should be forgotten. We now state the main result of \resec{boundhomo}.
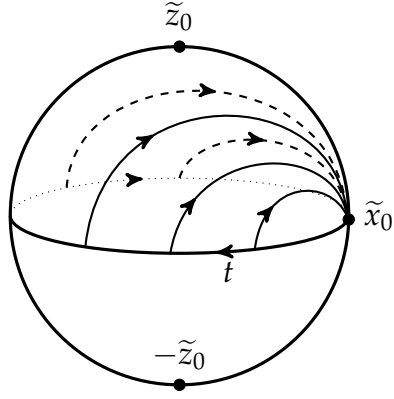
\begin{figure}[t]
\hfill
%
%
\begin{tikzpicture}[scale=0.5, >=stealth]
\draw[very thick] (0,0) circle(4.5);
\draw[dotted,middlearrow=0.65] (4.5,0) arc (0:180:4.5 and 1);
\draw[very thick, middlearrow=0.65] (-4.5,0) arc (180:360:4.5 and 1);
\draw[thick, middlearrowrev=0.35] (2,-0.9) .. controls (2.25,1) and (4,1) .. (4.45,0);
\draw[thick, middlearrowrev=0.25] (-0.25,-1.03) .. controls (0.25,2) and (4.15,2) .. (4.45,0);
\draw[thick, middlearrowrev=0.35] (-2.5,-0.8) .. controls (-2,3.75) and (4.2,3.6) .. (4.45,0);
\draw[thick, dashed, middlearrowrev=0.5] (-3,0.75) .. controls (-2.75,4.3) and (4.2,4.3) .. (4.45,0);
\draw[thick, dashed, middlearrowrev=0.45] (0,1) .. controls (0.4,2.5) and (4.2,2.5) .. (4.45,0);

\draw[fill] (4.5,-0.1) circle(0.15);
\draw[fill] (0,4.5) circle(0.15);
\draw[fill] (0,-4.5) circle(0.15);
\node at (5.3,0){$\widetilde{x}_{0}$};
\node at (0,5.3){$\widetilde{z}_{0}$};
\node at (0,-3.7){$-\widetilde{z}_{0}$};
\node at (1.3,-1.5){$t$};
\end{tikzpicture}
\hspace*{\fill}
\caption{The loop $\alpha_{n_{0}}'\circ \alpha_{\pi} \circ \alpha_{c}\circ \tau_{n_{0}}$ in $I_{n_{0}}$ in the case $M=\St$}
\label{fig:looptauS2}
\end{figure}

\begin{thm}\label{th:brrelnS2}
In $\pi_{1}(I_{n_{0}})$, we have
$2 \widehat{\tau}_{n_{0}}=
\begin{cases}
\widetilde{\delta}_{\widetilde{x}_{0}}+\widetilde{\delta}_{\widetilde{z}_{0}} - \widetilde{\delta}_{-\widetilde{z}_{0}} & \text{if $M=\St$}\\
\delta_{x_{0}}+\delta_{z_{0}} & \text{if $M=\rp$.}
\end{cases}$
\end{thm}

The proof of \reth{brrelnS2} will occupy Sections~\ref{sec:n3S2a} and~\ref{sec:rp2nequals2}. We first state and prove a corollary as well as a lemma that we will require for that proof.

\begin{cor}\label{cor:sesnn0}
The boundary homomorphism $\map{\partial_{n_{0}}}{\pi_{2}(\prod_{i=1}^{n_{0}}\, M)}[\pi_{1}(I_{n_{0}})]$ of the exact sequence~\reqref{lespartialn} is given by: 
\begin{align*}
\partial_{n_{0}}(\widetilde{\lambda}_{u})& =
\left\{
\begin{aligned}
&\widetilde{\delta}_{u}\\
& \widetilde{\delta}_{\widetilde{x}_{0}}+\widetilde{\delta}_{\widetilde{z}_{0}}-2\br{\widetilde{x}_{0}, \widetilde{z}_{0}, -\widetilde{z}_{0}}
\end{aligned}\right.
&& \begin{aligned}
&\text{if $u\in \brak{\widetilde{x}_{0},\widetilde{z}_{0}}$}\\
&\text{if $u=-\widetilde{z}_{0}$}
\end{aligned}
&& \text{and if $M=\St$, and}\\
\partial_{n_{0}}(\lambda_{u}) &=
\left\{
\begin{aligned}
&\delta_{u}\\
&2\br{x_{0},z_{0}}-\delta_{x_{0}}
\end{aligned}\right.
&& \begin{aligned}
&\text{if $u=x_{0}$}\\
&\text{if $u=z_{0}$}
\end{aligned}
&& \text{and if $M=\rp$.}
\end{align*}
\end{cor}

\begin{proof}[Proof of \reco{sesnn0}]
With the basis of $\pi_{2}(\prod_{i=1}^{n_{0}}\, M)$ given by~\reqref{basispi2}, the result is a consequence of~\reqref{defdeltav} and \reth{brrelnS2}.
\end{proof}

To state the lemma that follows, we first give two definitions.
\begin{enumerate}[(a)]
\item Let $M=\St$, let $\widetilde{u}, \widetilde{v}\in \St$ be such that $(\widetilde{u}, \widetilde{v},-\widetilde{v})\in F_{3}(\St)$, and define:
\begin{equation}\label{eq:defI0uv}
I_{n_{0}}(\widetilde{u}, \widetilde{v})=\setl{(\widetilde{z}, \lambda)\in F_{3}(\St)\times \left(\prod_{1}^{3} \St \right)^{I}}{\text{$\lambda(0)=\widetilde{z}$ and $\lambda(1)=(\widetilde{u}, \widetilde{v}, -\widetilde{v})$}},
\end{equation}
where we equip $I_{n_{0}}(\widetilde{u}, \widetilde{v})$ with the basepoint $(\widetilde{u}, \widetilde{v},-\widetilde{v},c_{\widetilde{u}},c_{\widetilde{v}},c_{-\widetilde{v}})$. 

\item Let $M=\rp$, let $(u,v)\in F_{2}(\rp)$, and define:
\begin{equation}\label{eq:defI0uvrp2}
I_{n_{0}}(u,v)=\setl{(z, \mu)\in F_{2}(\rp)\times \left(\prod_{1}^{2} \rp \right)^{I}}{\text{$\mu(0)=z$ and $\mu(1)=(u,v)$}},
\end{equation}
where we equip $I_{n_{0}}(u,v)$ with the basepoint $(u,v,c_{u},c_{v})$.
\end{enumerate}
With the notation and choice of basepoints given in \resec{generalities}, if $M=\St$ (resp.\ $M=\rp$) then $I_{n_{0}}(\widetilde{x}_{0},\widetilde{z}_{0})$ (resp.\ $I_{n_{0}}(x_{0},z_{0})$) coincides with $I_{n_{0}}$.

\begin{lem}\label{lem:deformwS2}
Let $\widetilde{w}\in \St$ be such that $p_{3}(\widetilde{w})>0$, let $\map{\Gamma}{I\times I}[\St]$ be a map such that for all $s,t\in I$, $p_{3}(\Gamma(s,t))\geq 0$, $\Gamma(0,t)= \widetilde{\omega}_{\widetilde{x}_{0},t}\left(\textstyle\frac{1}{2}\right)$ and $\Gamma(1,t)=\Gamma(s,0)=\Gamma(s,1) =\widetilde{x}_{0}$, and let $\map{\gamma_{t}}{I}[\St]$ be the map defined by $\gamma_{t}(s)=\Gamma(s,t)$. 
\begin{enumerate}[(a)]
\item\label{it:deformwS2a} Let $M=\St$. Then: 
\begin{equation*}
\left[ \left( \gamma_{t}(0), \widetilde{w}, -\widetilde{w}, \gamma_{t}, c_{\widetilde{w}}, c_{-\widetilde{w}}\right)_{t\in I}\right]= \bigl[ \bigl( \widetilde{\omega}_{\widetilde{x}_{0},t}\left(\textstyle\frac{1}{2}\right), \widetilde{w}, -\widetilde{w}, \Omega_{\widetilde{x}_{0},t,\frac{1}{2}}, c_{\widetilde{w}}, c_{-\widetilde{w}}\bigr)_{t\in I}\bigr]
\end{equation*}
in $\pi_{1}(I_{n_{0}}(\widetilde{x}_{0}, \widetilde{w}))$.

\item\label{it:deformwS2b} Let $M=\rp$, and let $w=\pi(\widetilde{w})$. Then:
\begin{equation*}
\bigl[ \left( \pi(\gamma_{t}(0)), w, \pi\circ \gamma_{t}, c_{w}\right)_{t\in I}\bigl]= \bigl[ \bigl( \pi\circ\widetilde{\omega}_{\widetilde{x}_{0},t}\left(\textstyle\frac{1}{2}\right), w, \pi\circ \Omega_{\widetilde{x}_{0},t,\frac{1}{2}}, c_{w}\bigr)_{t\in I}\bigr]
\end{equation*}
in $\pi_{1}(I_{n_{0}}(x_{0},w))$.
\end{enumerate}
\end{lem}

\begin{proof}\mbox{}
\begin{enumerate}[(a)]
\item From the conditions on $\widetilde{w}$ and $\Gamma$, $\left( \gamma_{t}(0), \widetilde{w}, -\widetilde{w}, \gamma_{t}, c_{\widetilde{w}}, c_{-\widetilde{w}}\right)\in I_{n_{0}}(\widetilde{x}_{0},\widetilde{w})$ for all $t\in I$.
For all $r,s,t\in I$, define $\map{\Theta_{r}}{I\times I}[\St]$ by:
\begin{equation*}
\Theta_{r}(s,t)= \frac{r\Omega_{\widetilde{x}_{0},t,\frac{1}{2}}(s)+ (1-r)\Gamma(s,t)}{\bigl\lVert r\Omega_{\widetilde{x}_{0},t,\frac{1}{2}}(s)+ (1-r)\Gamma(s,t)\bigr\rVert},
\end{equation*}
and let $\map{\theta_{r,t}}{I}[\St]$ be defined by $\theta_{r,t}(s)=\Theta_{r}(s,t)$. We claim that the map $\Theta_{r}$ is well defined: using \req{defOmega}, if $s=0$ then $\Theta_{r}(0,t) =\widetilde{\omega}_{\widetilde{x}_{0},t}\left(\frac{1}{2}\right)$, and if $s=1$ or $t\in \brak{0,1}$ then $\Theta_{r}(s,t)=\widetilde{x}_{0}$. So assume that $(s,t)\in (0,1)\times (0,1)$. Then $p_{3}(\Gamma(s,t))\geq 0$ by hypothesis and $p_{3}(\Omega_{\widetilde{x}_{0},t,\frac{1}{2}}(s))>0$ from condition~(\ref{it:xb}) given in \resec{geomrepsS2}, so $\Gamma(s,t)$ and $\Omega_{\widetilde{x}_{0},t,\frac{1}{2}}(s)$ are not antipodal, which proves the claim. In particular, $\theta_{r,t}(0)= \Theta_{r}(0,t)\notin \brak{\widetilde{w},-\widetilde{w}}$ and $p_{3}(\Theta_{r}(s,t))\geq 0$ for all $r,s,t\in I$. For $r\in I$, let:
\begin{equation*}
\text{$\xi(r)=\left( \theta_{r,t}(0), \widetilde{w}, -\widetilde{w}, \theta_{r,t}, c_{\widetilde{w}}, c_{-\widetilde{w}}\right)_{t\in I}$ in $I_{n_{0}}(\widetilde{x}_{0}, \widetilde{w})$.}
\end{equation*}
Since $\left( \theta_{r,t}(0), \widetilde{w}, -\widetilde{w}, \theta_{r,t}, c_{\widetilde{w}}, c_{-\widetilde{w}}\right)=(\widetilde{x}_{0}, \widetilde{w},-\widetilde{w}, c_{\widetilde{x}_{0}},c_{\widetilde{w}}, c_{-\widetilde{w}})$ for $t\in \brak{0,1}$, $\xi(r)$ is a loop in $I_{n_{0}}(\widetilde{x}_{0}, \widetilde{w})$ for all $r\in I$, and $\xi$ defines a based homotopy  in $I_{n_{0}}(\widetilde{x}_{0}, \widetilde{w})$ between the loops $\xi(0)=\bigl( \gamma_{t}(0), \widetilde{w}, -\widetilde{w},  \gamma_{t}, c_{\widetilde{w}}, c_{-\widetilde{w}}\bigr)_{t\in I}$ and $\xi(1)=\bigl( \widetilde{\omega}_{\widetilde{x}_{0},t}\left(\textstyle\frac{1}{2}\right), \widetilde{w}, -\widetilde{w}, \Omega_{\widetilde{x}_{0},t,\frac{1}{2}}, c_{\widetilde{w}}, c_{-\widetilde{w}}\bigr)_{t\in I}$ as required.

\item For $r\in I$, let $\xi'(r)=\left( \pi(\theta_{r,t}(0)), w, \pi\circ \theta_{r,t}, c_{w}\right)_{t\in I}$. From the proof of part~(\ref{it:deformwS2a}), $\xi'$ defines a based homotopy in $I_{n_{0}}(x_{0},w)$ between the loops $\xi'(0)=\bigl( \pi(\gamma_{t}(0)), w, \pi\circ \gamma_{t}, c_{w}\bigr)_{t\in I}$ and $\xi'(1)=\bigl( \pi\circ\widetilde{\omega}_{\widetilde{x}_{0},t}\left(\textstyle\frac{1}{2}\right), w, \pi\circ \Omega_{\widetilde{x}_{0},t,\frac{1}{2}}, c_{w}\bigr)_{t\in I}$, and the result follows.\qedhere
\end{enumerate}
\end{proof}

\subsubsection{The proof of \reth{brrelnS2} in the case $M=\St$}\label{sec:n3S2a}

Let $M=\St$. This section is devoted to proving \reth{brrelnS2} that describes the relation in $I_{n_{0}}$ between $\br{\widetilde{x}_{0}, \widetilde{z}_{0}, -\widetilde{z}_{0}}$ and the $\widetilde{\delta}_{u}$, where $u\in \brak{\widetilde{x}_{0}, \widetilde{z}_{0}, -\widetilde{z}_{0}}$.

\begin{proof}[Proof of \reth{brrelnS2} in the case $M=\St$]
By taking their sum, it suffices to show that the following two equalities in $I_{n_{0}}$ hold: 
\begin{align}
\br{\widetilde{x}_{0}, \widetilde{z}_{0}, -\widetilde{z}_{0}}&= \widetilde{\delta}_{\widetilde{x}_{0}}+ \dl{\widetilde{z}_{0}, -\widetilde{z}_{0}}\label{eq:breq1}\\
\dl{\widetilde{z}_{0}, -\widetilde{z}_{0}}&=-\br{\widetilde{x}_{0}, \widetilde{z}_{0}, -\widetilde{z}_{0}}+\widetilde{\delta}_{\widetilde{z}_{0}} - \widetilde{\delta}_{-\widetilde{z}_{0}},\label{eq:breq2}
\end{align}
where:
\begin{equation}\label{eq:defdlS2}
\dl{\widetilde{z}_{0}, -\widetilde{z}_{0}}=\bigl[\bigl(\widetilde{x}_{0}, \widetilde{\omega}_{\widetilde{z}_{0},1-t}({\textstyle\frac{1}{2}}), \widetilde{\omega}_{-\widetilde{z}_{0},t}({\textstyle\frac{1}{2}}), c_{\widetilde{x}_{0}}, \Omega'_{\widetilde{z}_{0},1-t, \frac{1}{2}}, \Omega'_{-\widetilde{z}_{0},t, \frac{1}{2}}\bigr)_{t\in I}\bigr].
\end{equation}
We start by proving~\reqref{breq1}. For all $\alpha,t\in I$, let:
\begin{equation}\label{eq:defwtalpha}
w_{t,\alpha}=\bigl(\Omega_{\widetilde{x}_{0},t, \frac{1-\alpha}{2}}(0), \Omega'_{\widetilde{z}_{0},1-t, \frac{\alpha}{2}}(0), \Omega'_{-\widetilde{z}_{0}, t, \frac{\alpha}{2}}(0), \Omega_{\widetilde{x}_{0},t, \frac{1-\alpha}{2}}, \Omega'_{\widetilde{z}_{0},1-t, \frac{\alpha}{2}}, \Omega'_{-\widetilde{z}_{0},t, \frac{\alpha}{2}}\bigr).
\end{equation}
The three components of $w_{t,\alpha}$, namely the first and fourth coordinates, the second and fifth coordinates, and the third and sixth coordinates, as $t$ increases, are illustrated in Figure~\ref{fig:wtalpha}.
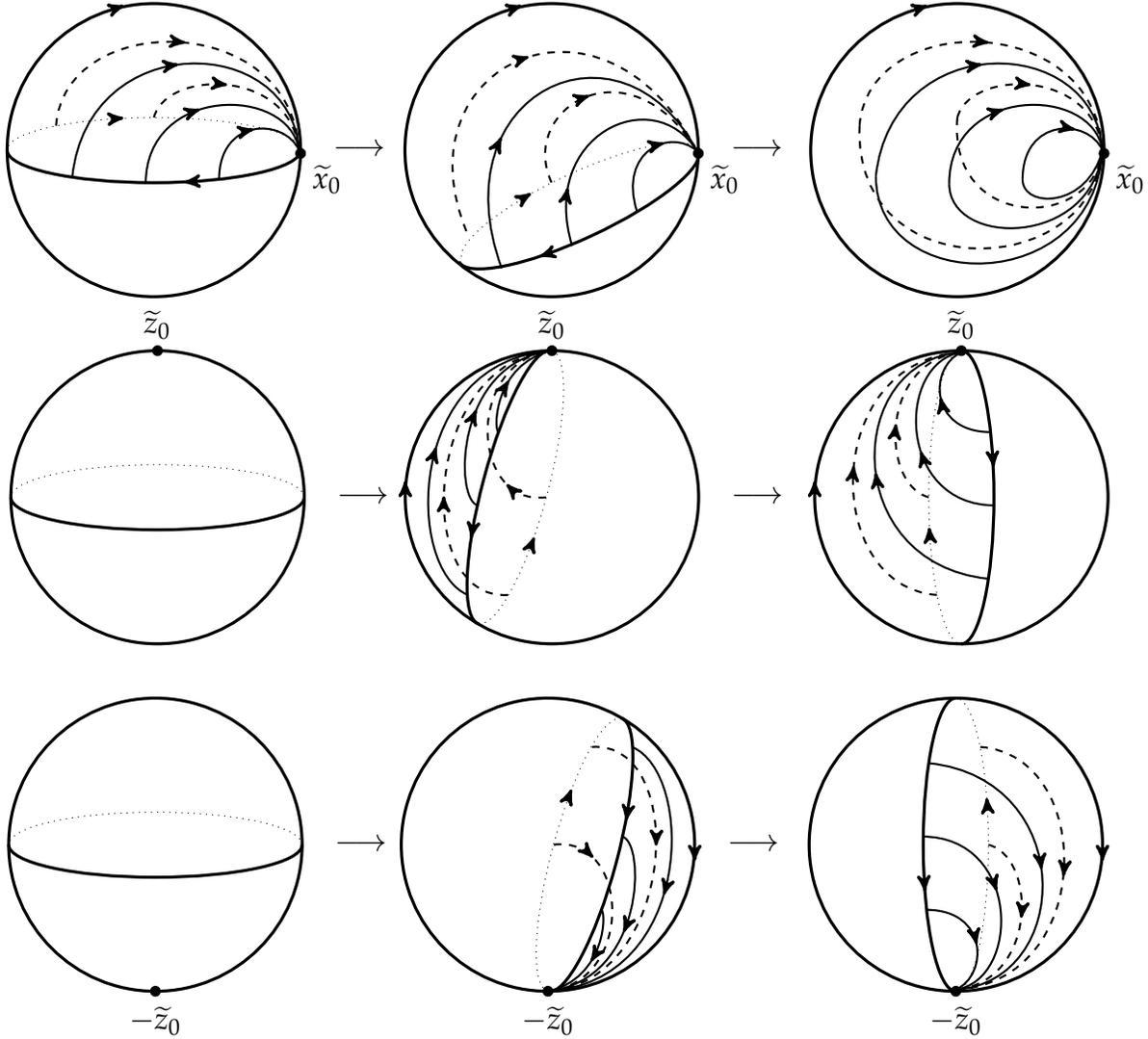
\begin{figure}[!h]
\hspace*{\fill}
\begin{tikzpicture}[scale=0.45, >=stealth]
\draw[very thick,middlearrow=0.3] (0,0) circle(4.5);
\draw[dotted,middlearrow=0.65] (4.5,0) arc (0:180:4.5 and 1);
\draw[very thick, middlearrow=0.65] (-4.5,0) arc (180:360:4.5 and 1);

\draw[thick, middlearrowrev=0.55] (2,-0.9) .. controls (2,1) and (4,1) .. (4.45,0);
\draw[thick, middlearrowrev=0.45] (-0.25,-1.03) .. controls (-0.25,2) and (4.15,2) .. (4.45,0);
\draw[thick, middlearrowrev=0.5] (-2.5,-0.8) .. controls (-2,3.75) and (4.2,3.6) .. (4.45,0);
\draw[thick, dashed, middlearrowrev=0.5] (-3,0.75) .. controls (-2.75,4.3) and (4.2,4.3) .. (4.45,0);
\draw[thick, dashed, middlearrowrev=0.45] (0,1) .. controls (0.4,2.5) and (4.2,2.5) .. (4.45,0);

%

\draw[fill] (4.5,-0.1) circle(0.15);
\node at (5.3,-0.9){$\widetilde{x}_{0}$};
\node at (6.3,0){$\to$};
\end{tikzpicture}\hfill
%
\begin{tikzpicture}[scale=0.45, >=stealth]
\draw[very thick,middlearrow=0.3] (0,0) circle(4.5);
\begin{scope}[rotate around={25:(4.5,0)}]
\draw[dotted,middlearrow=0.7] (4.5,0) arc (0:180:4 and 1);
\draw[very thick, middlearrowrev=0.7] (4.5,0) arc (360:180:4.05 and 1);

\draw[thick, middlearrowrev=0.7] (2,-0.9) .. controls (2.25,1) and (4,1) .. (4.45,0);
\draw[thick, middlearrowrev=0.25] (-0.25,-1.03) .. controls (0.25,2.5) and (4.15,2.5) .. (4.45,0);
\draw[thick, middlearrowrev=0.35] (-2.5,-0.74) .. controls (-2,5) and (4.2,5) .. (4.45,0);
\draw[thick, dashed, middlearrowrev=0.5] (-3,0.75) .. controls (-2.75,6.3) and (4.2,6.3) .. (4.45,0);
\draw[thick, dashed, middlearrowrev=0.45] (0,1) .. controls (0,3.5) and (4,4) .. (4.45,0);
\end{scope}

\draw[fill] (4.5,-0.1) circle(0.15);
\node at (5.3,-0.9){$\widetilde{x}_{0}$};
\node at (6.3,0){$\to$};
\end{tikzpicture}
\hspace*{\fill}
\begin{tikzpicture}[scale=0.45, >=stealth]
\draw[very thick,middlearrow=0.3] (0,0) circle(4.5);

\draw[thick, middlearrowrev=0.7] (2,-0.9) .. controls (2,1) and (4,1) .. (4.45,0);
\draw[thick, middlearrowrev=0.45] (-0.25,-1.03) .. controls (-0.25,2) and (4.15,2) .. (4.45,0);
\draw[thick, middlearrowrev=0.5] (-2.5,-0.8) .. controls (-2,3.75) and (4.2,3.6) .. (4.45,0);
\draw[thick, dashed, middlearrowrev=0.5] (-3,0.75) .. controls (-2.75,4.3) and (4.2,4.3) .. (4.45,0);
\draw[thick, dashed, middlearrowrev=0.45] (0,1) .. controls (0.4,2.5) and (4.2,2.5) .. (4.45,0);

\draw[thick] (2,-0.9) .. controls (2.,-1.8) and (4,-1.8) .. (4.45,0);
\draw[thick] (-0.25,-1.03) .. controls (-0.25,-3) and (4.15,-3) .. (4.45,0);
\draw[thick] (-2.5,-0.8) .. controls (-2.5,-4.5) and (4.2,-4.5) .. (4.45,0);
\draw[thick, dashed] (-3,0.75) .. controls (-3.2,-4.3) and (4.2,-4) .. (4.45,0);
\draw[thick, dashed] (0,1) .. controls (-0.3,-2.5) and (4.2,-2.5) .. (4.45,0);

\draw[fill] (4.5,-0.1) circle(0.15);
\node at (5.3,-0.9){$\widetilde{x}_{0}$};
\end{tikzpicture}
\hspace*{\fill}\linebreak
\noindent\hspace*{\fill}
\begin{tikzpicture}[scale=0.45, >=stealth]
\draw[very thick] (0,0) circle(4.5);
\draw[dotted] (4.5,0) arc (0:180:4.5 and 1);
\draw[very thick] (-4.5,0) arc (180:360:4.5 and 1);
%

\draw[fill] (0,4.5) circle(0.15);
\node at (5.3,-0.9){\phantom{$\widetilde{x}_{0}$}};
\node at (6.3,0){$\to$};
\node at (0,5.3){$\widetilde{z}_{0}$};
\end{tikzpicture}\hfill
\begin{tikzpicture}[scale=0.45, >=stealth]
\draw[very thick,middlearrow=0.5] (0,0) circle(4.5);

\begin{scope}[rotate around={-90:(0,0)}]
\begin{scope}[rotate around={-15:(-4.5,0)}]
\draw[dotted,middlearrow=0.7] (-4.5,0) arc (0:180:-4.33 and 1);
\draw[very thick, middlearrowrev=0.7] (-4.5,0) arc (360:180:-4.34 and 1);
\end{scope}


\draw[thick, middlearrowrev=0.5] (-2,-1.63) .. controls (-2,-1.8) and (-4.1,-1.8) .. (-4.5,0);
\draw[thick, middlearrowrev=0.5] (0.25,-2.27) .. controls (0.1,-3) and (-4.15,-3) .. (-4.5,0);
\draw[thick, middlearrowrev=0.5] (3,-2.58) .. controls (2.2,-4.5) and (-4.2,-4.5) .. (-4.5,0);
\draw[thick, dashed, middlearrowrev=0.4] (3,-1.35) .. controls (3.2,-4.3) and (-4.2,-4) .. (-4.5,0);
\draw[thick, dashed, middlearrowrev=0.2] (0,-0.18) .. controls (0.3,-2.5) and (-4.2,-2.7) .. (-4.5,0);
\end{scope}

\draw[fill] (0,4.5) circle(0.15);
\node at (5.3,-0.9){\phantom{$\widetilde{x}_{0}$}};
\node at (6.3,0){$\to$};
\node at (0,5.3){$\widetilde{z}_{0}$};
\end{tikzpicture}
\hspace*{\fill}
\begin{tikzpicture}[scale=0.45, >=stealth]
\begin{scope}[rotate around={90:(0,0)}]
\draw[very thick,middlearrow=0.25] (0,0) circle(4.5);
\draw[dotted,middlearrow=0.65] (4.5,0) arc (0:180:4.5 and 1);
\draw[very thick, middlearrow=0.65] (-4.5,0) arc (180:360:4.5 and 1);

\draw[thick, middlearrowrev=0.55] (2,-0.9) .. controls (2,1) and (4,1) .. (4.45,0);
\draw[thick, middlearrowrev=0.45] (-0.25,-1.03) .. controls (-0.25,2) and (4.15,2) .. (4.45,0);
\draw[thick, middlearrowrev=0.5] (-2.5,-0.8) .. controls (-2,3.75) and (4.2,3.6) .. (4.45,0);
\draw[thick, dashed, middlearrowrev=0.5] (-3,0.75) .. controls (-2.75,4.3) and (4.2,4.3) .. (4.45,0);
\draw[thick, dashed, middlearrowrev=0.45] (0,1) .. controls (0.4,2.5) and (4.2,2.5) .. (4.45,0);
\end{scope}

\draw[fill] (0,4.5) circle(0.15);

\node at (5.3,-0.9){\phantom{$\widetilde{x}_{0}$}};
\node at (0,5.3){$\widetilde{z}_{0}$};
\end{tikzpicture}\hspace*{\fill}
\linebreak
\noindent\hspace*{\fill}
\begin{tikzpicture}[scale=0.45, >=stealth]
\draw[very thick] (0,0) circle(4.5);
\draw[dotted] (4.5,0) arc (0:180:4.5 and 1);
\draw[very thick] (-4.5,0) arc (180:360:4.5 and 1);
%

\draw[fill] (0,-4.5) circle(0.15);
\node at (5.3,-0.9){\phantom{$\widetilde{x}_{0}$}};
\node at (6.3,0){$\to$};
\node at (0,5.3){\phantom{$\widetilde{z}_{0}$}};
\node at (0,-5.4){$-\widetilde{z}_{0}$};
\end{tikzpicture}\hfill
\begin{tikzpicture}[scale=0.45, >=stealth]
\draw[very thick,middlearrow=0] (0,0) circle(4.5);

\begin{scope}[rotate around={90:(0,0)}]
\begin{scope}[rotate around={-15:(-4.5,0)}]
\draw[dotted,middlearrowrev=0.7] (-4.5,0) arc (0:180:-4.33 and 1);
\draw[very thick, middlearrow=0.7] (-4.5,0) arc (360:180:-4.34 and 1);
\end{scope}


\draw[thick, middlearrowrev=0.5] (-2,-1.63) .. controls (-2,-1.8) and (-4.1,-1.8) .. (-4.5,0);
\draw[thick, middlearrowrev=0.5] (0.25,-2.27) .. controls (0.1,-3) and (-4.15,-3) .. (-4.5,0);
\draw[thick, middlearrowrev=0.5] (3,-2.58) .. controls (2.2,-4.5) and (-4.2,-4.5) .. (-4.5,0);
\draw[thick, dashed, middlearrowrev=0.4] (3,-1.35) .. controls (3.2,-4.3) and (-4.2,-4) .. (-4.5,0);
\draw[thick, dashed, middlearrowrev=0.2] (0,-0.18) .. controls (0.3,-2.5) and (-4.2,-2.7) .. (-4.5,0);
\end{scope}

\draw[fill] (0,-4.5) circle(0.15);
\node at (5.3,-0.9){\phantom{$\widetilde{x}_{0}$}};
\node at (6.3,0){$\to$};
\node at (0,5.3){\phantom{$\widetilde{z}_{0}$}};
\node at (0,-5.4){$-\widetilde{z}_{0}$};
\end{tikzpicture}
\hspace*{\fill}
\begin{tikzpicture}[scale=0.45, >=stealth]
\begin{scope}[rotate around={-90:(0,0)}]
\draw[very thick,middlearrow=0.25] (0,0) circle(4.5);
\draw[dotted,middlearrowrev=0.65] (4.5,0) arc (0:180:4.5 and 1);
\draw[very thick, middlearrowrev=0.65] (-4.5,0) arc (180:360:4.5 and 1);

\draw[thick, middlearrowrev=0.55] (2,-0.9) .. controls (2,1) and (4,1) .. (4.45,0);
\draw[thick, middlearrowrev=0.45] (-0.25,-1.03) .. controls (-0.25,2) and (4.15,2) .. (4.45,0);
\draw[thick, middlearrowrev=0.5] (-2.5,-0.8) .. controls (-2,3.75) and (4.2,3.6) .. (4.45,0);
\draw[thick, dashed, middlearrowrev=0.5] (-3,0.75) .. controls (-2.75,4.3) and (4.2,4.3) .. (4.45,0);
\draw[thick, dashed, middlearrowrev=0.45] (0,1) .. controls (0.4,2.5) and (4.2,2.5) .. (4.45,0);
\end{scope}

\draw[fill] (0,-4.5) circle(0.15);

\node at (5.3,-0.9){\phantom{$\widetilde{x}_{0}$}};
\node at (0,5.3){\phantom{$\widetilde{z}_{0}$}};
\node at (0,-5.4){$-\widetilde{z}_{0}$};
\end{tikzpicture}
\hspace*{\fill}
\caption{The three components of $w_{t,\alpha}$ based at $\widetilde{x}_{0}$ (the first row), $\widetilde{z}_{0}$ (the second row), and $-\widetilde{z}_{0}$ (the third row), as $t$ increases from $0$ to $1$. To obtain $w_{t,\alpha}$, the figures in each column should be superimposed. The left-hand column represents $\widehat{\tau}_{n_{0}}$, and the right-hand column represents $\widetilde{\delta}_{n_{0}}+\xi_{n_{0}}$.}
\label{fig:wtalpha}
\end{figure}
We claim that $w_{t,\alpha}\in I_{n_{0}}$. To see this, by~\reqref{defOmega}, the arc $(\Omega_{\widetilde{x}_{0},t, \frac{1-\alpha}{2}})_{t\in I}$ (resp.\ $(\Omega'_{\widetilde{z}_{0},1-t, \frac{\alpha}{2}})_{t\in I}$, $(\Omega'_{-\widetilde{z}_{0},t, \frac{\alpha}{2}})_{t\in I}$) joins $\Omega_{\widetilde{x}_{0},t, \frac{1-\alpha}{2}}(0)$ (resp.\ $\Omega'_{\widetilde{z}_{0},1-t, \frac{\alpha}{2}}(0)$, $\Omega'_{-\widetilde{z}_{0},t, \frac{\alpha}{2}}(0)$) to $\widetilde{x}_{0}$ (resp.\ to $\widetilde{z}_{0}$, $-\widetilde{z}_{0}$). So prove the claim, it remains to show that $(\Omega_{\widetilde{x}_{0},t, \frac{1-\alpha}{2}}(0), \Omega'_{\widetilde{z}_{0},1-t, \frac{\alpha}{2}}(0), \Omega'_{-\widetilde{z}_{0}, t, \frac{\alpha}{2}}(0))$ belongs to $F_{3}(\St)$. First, using \req{minusOmega}, $\Omega'_{\widetilde{z}_{0},1-t, \frac{\alpha}{2}}(0)=-\Omega'_{-\widetilde{z}_{0}, t, \frac{\alpha}{2}}(0)$ for all $t,\alpha\in I$, so $\Omega'_{\widetilde{z}_{0},1-t, \frac{\alpha}{2}}(0)\neq \Omega'_{-\widetilde{z}_{0}, t, \frac{\alpha}{2}}(0)$. We now prove that for all $t, \alpha \in I$, $\Omega_{\widetilde{x}_{0},t, \frac{1-\alpha}{2}}(0)\neq \Omega'_{-\widetilde{z}_{0}, t, \frac{\alpha}{2}}(0)$, or equivalently $\widetilde{\omega}_{\widetilde{x}_{0},t}({\textstyle\frac{1-\alpha}{2}}) \neq -\widetilde{\omega}_{\widetilde{z}_{0},1-t}({\textstyle\frac{\alpha}{2}})$, by considering the following cases: 
\begin{enumerate}[(a)]
\item if $t\in \brak{0,1}$ then $\widetilde{\omega}_{\widetilde{x}_{0},t}({\textstyle\frac{1-\alpha}{2}})=\widetilde{x}_{0}$ and $-\widetilde{\omega}_{\widetilde{z}_{0},1-t}({\textstyle\frac{\alpha}{2}})=-\widetilde{z}_{0}$ by conditions~(\ref{it:xa}) and~(\ref{it:za}).

\item Suppose that $t\in (0,1)$. If $\alpha=0$ then $\widetilde{\omega}_{\widetilde{x}_{0},t}({\textstyle\frac{1-\alpha}{2}})=\widetilde{\omega}_{\widetilde{x}_{0},t}({\textstyle\frac{1}{2}})$ and $-\widetilde{\omega}_{\widetilde{z}_{0},1-t}({\textstyle\frac{\alpha}{2}})=-\widetilde{z}_{0}$, but $\widetilde{\omega}_{\widetilde{x}_{0},t}({\textstyle\frac{1}{2}})\neq -\widetilde{z}_{0}$ by condition~(\ref{it:xb}). If $\alpha=1$ then $\widetilde{\omega}_{\widetilde{x}_{0},t}({\textstyle\frac{1-\alpha}{2}})=\widetilde{x}_{0}$ and $-\widetilde{\omega}_{\widetilde{z}_{0},1-t}({\textstyle\frac{\alpha}{2}})=-\widetilde{\omega}_{\widetilde{z}_{0},1-t}({\textstyle\frac{1}{2}})$, but $-\widetilde{\omega}_{\widetilde{z}_{0},1-t}({\textstyle\frac{1}{2}}) \neq \widetilde{x}_{0}$ by condition~(\ref{it:zb}).

\item suppose that $t\in (0,1)\setminus \brak{\frac{1}{2}}$ and $\alpha\in (0,1)$. By conditions~(\ref{it:xa}) and~(\ref{it:za}), we see that $p_{2}(\widetilde{\omega}_{\widetilde{x}_{0},t}({\textstyle\frac{1-\alpha}{2}})) \ldotp p_{2}(-\widetilde{\omega}_{\widetilde{z}_{0},1-t}({\textstyle\frac{\alpha}{2}}))<0$, in particular $\widetilde{\omega}_{\widetilde{x}_{0},t}({\textstyle\frac{1-\alpha}{2}}) \neq -\widetilde{\omega}_{\widetilde{z}_{0},1-t}({\textstyle\frac{\alpha}{2}})$.  

\item suppose that $t=\frac{1}{2}$  and $\alpha\in (0,1)$. 
\begin{enumerate}[(i)]
\item If $\alpha\in (0, \frac{1}{2})$ (resp.\ $\alpha\in (\frac{1}{2},1)$) then $p_{1}(\widetilde{\omega}_{\widetilde{x}_{0},\frac{1}{2}}({\textstyle\frac{1-\alpha}{2}}))<0$ and $p_{1}(-\widetilde{\omega}_{\widetilde{z}_{0},\frac{1}{2}}({\textstyle\frac{\alpha}{2}}))> 0$ by conditions~(\ref{it:xe}) and~(\ref{it:zb}) (resp.\ $p_{3}(\widetilde{\omega}_{\widetilde{x}_{0},\frac{1}{2}}({\textstyle\frac{1-\alpha}{2}})) < 0$ and $p_{3}(-\widetilde{\omega}_{\widetilde{z}_{0},\frac{1}{2}}({\textstyle\frac{\alpha}{2}}))>0$ by conditions~(\ref{it:xb}) and~(\ref{it:ze})), so $\widetilde{\omega}_{\widetilde{x}_{0},\frac{1}{2}}({\textstyle\frac{1-\alpha}{2}}) \neq -\widetilde{\omega}_{\widetilde{z}_{0},\frac{1}{2}}({\textstyle\frac{\alpha}{2}})$.


\item If $\alpha=\frac{1}{2}$ then $p_{1}(\widetilde{\omega}_{\widetilde{x}_{0},\frac{1}{2}}({\textstyle\frac{1}{4}})) = 0$ and $p_{1}(-\widetilde{\omega}_{\widetilde{z}_{0},\frac{1}{2}}({\textstyle\frac{1}{4}}))>0$ by conditions~(\ref{it:xe}) and~(\ref{it:zb}), so $\widetilde{\omega}_{\widetilde{x}_{0},\frac{1}{2}}({\textstyle\frac{1}{4}}) \neq -\widetilde{\omega}_{\widetilde{z}_{0},\frac{1}{2}}({\textstyle\frac{1}{4}})$.
\end{enumerate}
\end{enumerate}
Finally, let us show that $\Omega_{\widetilde{x}_{0},t, \frac{1-\alpha}{2}}(0)\neq \Omega'_{\widetilde{z}_{0}, 1-t, \frac{\alpha}{2}}(0)$ for all $t, \alpha \in I$, or equivalently that $\widetilde{\omega}_{\widetilde{x}_{0},t}({\textstyle\frac{1-\alpha}{2}}) \neq \widetilde{\omega}_{\widetilde{z}_{0},1-t}({\textstyle\frac{\alpha}{2}})$, by considering the following cases:
\begin{enumerate}[(a)]
\item if $t\in \brak{0,1}$ then $\widetilde{\omega}_{\widetilde{x}_{0},t}({\textstyle\frac{1-\alpha}{2}})=\widetilde{x}_{0}$ and $\widetilde{\omega}_{\widetilde{z}_{0},1-t}({\textstyle\frac{\alpha}{2}})=\widetilde{z}_{0}$.

\item if $t\in (0,1)$ and $\alpha=0$ then $\widetilde{z}_{0}=\widetilde{\omega}_{\widetilde{z}_{0},1-t}(0)\neq \widetilde{\omega}_{\widetilde{x}_{0},t}({\textstyle\frac{1}{2}})$ by condition~(\ref{it:xb}).

\item suppose that $t\in (0,1)$ and $\alpha\in (0,\frac{1}{2}]$ (resp.\ $\alpha\in (\frac{1}{2},1]$). Then $p_{3}(\widetilde{\omega}_{\widetilde{x}_{0},t}({\textstyle\frac{1-\alpha}{2}}))<0$ and $p_{3}(\widetilde{\omega}_{\widetilde{z}_{0},1-t}({\textstyle\frac{\alpha}{2}}))\geq 0$ by conditions~(\ref{it:xb}) and~(\ref{it:ze}) (resp.\ $p_{1}(\widetilde{\omega}_{\widetilde{x}_{0},t}({\textstyle\frac{1-\alpha}{2}}))>0$ and $p_{1}(\widetilde{\omega}_{\widetilde{z}_{0},1-t}({\textstyle\frac{\alpha}{2}}))\leq 0$ by conditions~(\ref{it:xe}) and~(\ref{it:zb})), so $\widetilde{\omega}_{\widetilde{x}_{0},t}({\textstyle\frac{1-\alpha}{2}}) \neq \widetilde{\omega}_{\widetilde{z}_{0},1-t}({\textstyle\frac{\alpha}{2}})$.

\end{enumerate}
This proves the claim that $w_{t,\alpha}$ lies in $I_{n_{0}}$. 

For all $\alpha,t\in I$, we define:
\begin{equation}\label{eq:defvtalpha}
\!\!v_{t,\alpha}\!=\!\bigl( \Omega_{\widetilde{x}_{0},1-t,\! \frac{2-\alpha}{2}}(0), \Omega'_{\widetilde{z}_{0},1-t, \frac{1+\alpha}{2}}(0), \Omega'_{-\widetilde{z}_{0},t, \frac{1+\alpha}{2}}(0), \Omega_{\widetilde{x}_{0},1-t, \frac{2-\alpha}{2}}, \Omega'_{\widetilde{z}_{0},1-t, \frac{1+\alpha}{2}}, \Omega'_{-\widetilde{z}_{0},t, \frac{1+\alpha}{2}} \bigr).
\end{equation}
The three components of $v_{t,\alpha}$, namely the first and fourth coordinates, the second and fifth coordinates, and the third and sixth coordinates, as $t$ increases, are illustrated in Figure~\ref{fig:vtalpha}.
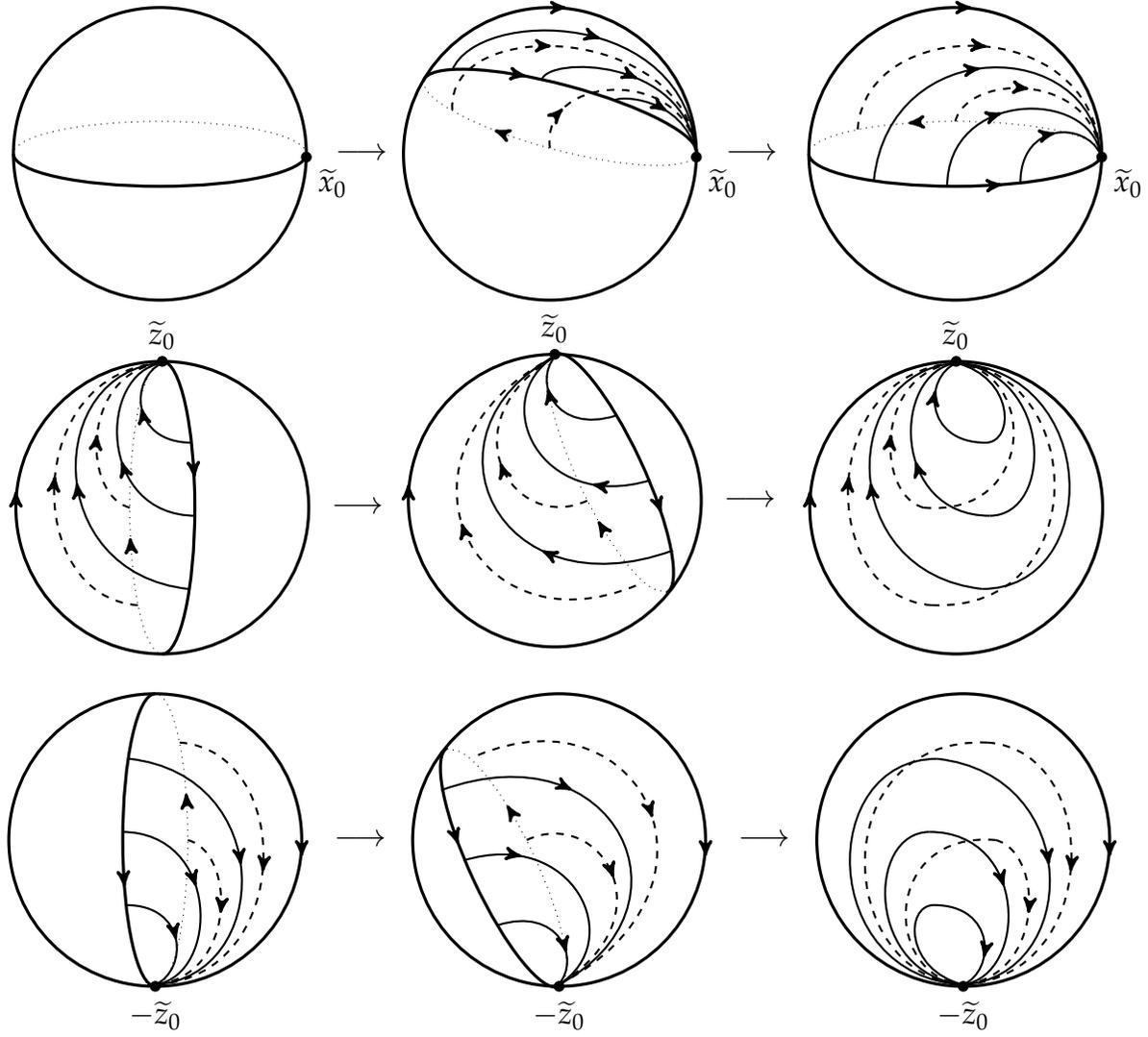
\begin{figure}[!h]
\hspace*{\fill}
\begin{tikzpicture}[scale=0.45, >=stealth]
\draw[very thick] (0,0) circle(4.5);
\draw[dotted] (4.5,0) arc (0:180:4.5 and 1);
\draw[very thick] (-4.5,0) arc (180:360:4.5 and 1);
%

\draw[fill] (4.5,-0.1) circle(0.15);
\node at (5.3,-0.9){$\widetilde{x}_{0}$};
\node at (6.2,0){$\to$};
\end{tikzpicture}\hfill
\begin{tikzpicture}[scale=0.45, >=stealth]
\draw[very thick,middlearrow=0.25] (0,0) circle(4.5);

\begin{scope}[rotate around={180:(0,0)}]
\begin{scope}[rotate around={-15:(-4.5,0)}]
\draw[dotted,middlearrowrev=0.7] (-4.5,0) arc (0:180:-4.33 and 1);
\draw[very thick, middlearrow=0.7] (-4.5,0) arc (360:180:-4.34 and 1);
\end{scope}


\draw[thick, middlearrowrev=0.5] (-2,-1.63) .. controls (-2,-1.8) and (-4.1,-1.8) .. (-4.5,0);
\draw[thick, middlearrowrev=0.5] (0.25,-2.27) .. controls (0.1,-3) and (-4.15,-3) .. (-4.5,0);
\draw[thick, middlearrowrev=0.5] (3,-2.58) .. controls (2.2,-4.5) and (-4.2,-4.5) .. (-4.5,0);
\draw[thick, dashed, middlearrowrev=0.4] (3,-1.35) .. controls (3.2,-4.3) and (-4.2,-4) .. (-4.5,0);
\draw[thick, dashed, middlearrowrev=0.2] (0,-0.18) .. controls (0.3,-2.5) and (-4.2,-2.7) .. (-4.5,0);
\end{scope}

\draw[fill] (4.5,-0.1) circle(0.15);
\node at (5.3,-0.9){$\widetilde{x}_{0}$};
\node at (6.2,0){$\to$};
\end{tikzpicture}
\hspace*{\fill}
\begin{tikzpicture}[scale=0.45, >=stealth]
\draw[very thick,middlearrow=0.25] (0,0) circle(4.5);
\draw[dotted,middlearrowrev=0.65] (4.5,0) arc (0:180:4.5 and 1);
\draw[very thick, middlearrowrev=0.65] (-4.5,0) arc (180:360:4.5 and 1);

\draw[thick, middlearrowrev=0.55] (2,-0.9) .. controls (2,1) and (4,1) .. (4.45,0);
\draw[thick, middlearrowrev=0.45] (-0.25,-1.03) .. controls (-0.25,2) and (4.15,2) .. (4.45,0);
\draw[thick, middlearrowrev=0.5] (-2.5,-0.8) .. controls (-2,3.75) and (4.2,3.6) .. (4.45,0);
\draw[thick, dashed, middlearrowrev=0.5] (-3,0.75) .. controls (-2.75,4.3) and (4.2,4.3) .. (4.45,0);
\draw[thick, dashed, middlearrowrev=0.45] (0,1) .. controls (0.4,2.5) and (4.2,2.5) .. (4.45,0);

%

\draw[fill] (4.5,-0.1) circle(0.15);
\node at (5.3,-0.9){$\widetilde{x}_{0}$};
\end{tikzpicture}
\hspace*{\fill}\linebreak
\noindent\hspace*{\fill}
\begin{tikzpicture}[scale=0.45, >=stealth]
\draw[very thick,middlearrow=0.5] (0,0) circle(4.5);
\begin{scope}[rotate around={90:(0,0)}]
\draw[dotted,middlearrow=0.65] (4.5,0) arc (0:180:4.5 and 1);
\draw[very thick, middlearrow=0.65] (-4.5,0) arc (180:360:4.5 and 1);

\draw[thick, middlearrowrev=0.55] (2,-0.9) .. controls (2,1) and (4,1) .. (4.45,0);
\draw[thick, middlearrowrev=0.45] (-0.25,-1.03) .. controls (-0.25,2) and (4.15,2) .. (4.45,0);
\draw[thick, middlearrowrev=0.5] (-2.5,-0.8) .. controls (-2,3.75) and (4.2,3.6) .. (4.45,0);
\draw[thick, dashed, middlearrowrev=0.5] (-3,0.75) .. controls (-2.75,4.3) and (4.2,4.3) .. (4.45,0);
\draw[thick, dashed, middlearrowrev=0.45] (0,1) .. controls (0.4,2.5) and (4.2,2.5) .. (4.45,0);
\end{scope}

\draw[fill] (0,4.5) circle(0.15);

\node at (5.3,-0.9){\phantom{$\widetilde{x}_{0}$}};
\node at (6,0){$\to$};
\node at (0,5.3){$\widetilde{z}_{0}$};
\end{tikzpicture}\hspace*{\fill}\hspace*{-5mm}
\begin{tikzpicture}[scale=0.45, >=stealth]
\draw[very thick,middlearrow=0.5] (0,0) circle(4.5);
\begin{scope}[rotate around={90:(0,0)}]
\begin{scope}[rotate around={25:(4.5,0)}]
\draw[dotted,middlearrow=0.7] (4.5,0) arc (0:180:4 and 1);
\draw[very thick, middlearrowrev=0.7] (4.5,0) arc (360:180:4.05 and 1);

\draw[thick, middlearrowrev=0.7] (2,-0.9) .. controls (2.25,1) and (4,1) .. (4.45,0);
\draw[thick, middlearrowrev=0.25] (-0.25,-1.03) .. controls (0.25,2.5) and (4.15,2.5) .. (4.45,0);
\draw[thick, middlearrowrev=0.35] (-2.5,-0.74) .. controls (-2,5) and (4.2,5) .. (4.45,0);
\draw[thick, dashed, middlearrowrev=0.5] (-3,0.75) .. controls (-2.75,6.3) and (4.2,6.3) .. (4.45,0);
\draw[thick, dashed, middlearrowrev=0.45] (0,1) .. controls (0,3.5) and (4,4) .. (4.45,0);
\end{scope}
\end{scope}

\draw[fill] (0,4.5) circle(0.15);

\node at (5.3,-0.9){\phantom{$\widetilde{x}_{0}$}};
\node at (6,0){$\to$};
\node at (0,5.3){$\widetilde{z}_{0}$};
\end{tikzpicture}
\hspace*{\fill}
\begin{tikzpicture}[scale=0.45, >=stealth]
\draw[very thick,middlearrow=0.5] (0,0) circle(4.5);

\begin{scope}[rotate around={90:(0,0)}]
\draw[thick, middlearrowrev=0.7] (2,-0.9) .. controls (2,1) and (4,1) .. (4.45,0);
\draw[thick, middlearrowrev=0.45] (-0.25,-1.03) .. controls (-0.25,2) and (4.15,2) .. (4.45,0);
\draw[thick, middlearrowrev=0.5] (-2.5,-0.8) .. controls (-2,3.75) and (4.2,3.6) .. (4.45,0);
\draw[thick, dashed, middlearrowrev=0.5] (-3,0.75) .. controls (-2.75,4.3) and (4.2,4.3) .. (4.45,0);
\draw[thick, dashed, middlearrowrev=0.45] (0,1) .. controls (0.4,2.5) and (4.2,2.5) .. (4.45,0);

\draw[thick] (2,-0.9) .. controls (2.,-1.8) and (4,-1.8) .. (4.45,0);
\draw[thick] (-0.25,-1.03) .. controls (-0.25,-3) and (4.15,-3) .. (4.45,0);
\draw[thick] (-2.5,-0.8) .. controls (-2.5,-4.5) and (4.2,-4.5) .. (4.45,0);
\draw[thick, dashed] (-3,0.75) .. controls (-3.2,-4.3) and (4.2,-4) .. (4.45,0);
\draw[thick, dashed] (0,1) .. controls (-0.3,-2.5) and (4.2,-2.5) .. (4.45,0);
\end{scope}

\draw[fill] (0,4.5) circle(0.15);

\node at (5.2,-0.9){\phantom{$\widetilde{x}_{0}$}};
\node at (0,5.3){$\widetilde{z}_{0}$};
\end{tikzpicture}
\hspace*{\fill}\vspace*{-2mm}\linebreak
\noindent\hspace*{\fill}
\begin{tikzpicture}[scale=0.45, >=stealth]
\draw[very thick,middlearrow=0] (0,0) circle(4.5);
\begin{scope}[rotate around={-90:(0,0)}]
\draw[dotted,middlearrowrev=0.65] (4.5,0) arc (0:180:4.5 and 1);
\draw[very thick, middlearrowrev=0.65] (-4.5,0) arc (180:360:4.5 and 1);

\draw[thick, middlearrowrev=0.55] (2,-0.9) .. controls (2,1) and (4,1) .. (4.45,0);
\draw[thick, middlearrowrev=0.45] (-0.25,-1.03) .. controls (-0.25,2) and (4.15,2) .. (4.45,0);
\draw[thick, middlearrowrev=0.5] (-2.5,-0.8) .. controls (-2,3.75) and (4.2,3.6) .. (4.45,0);
\draw[thick, dashed, middlearrowrev=0.5] (-3,0.75) .. controls (-2.75,4.3) and (4.2,4.3) .. (4.45,0);
\draw[thick, dashed, middlearrowrev=0.45] (0,1) .. controls (0.4,2.5) and (4.2,2.5) .. (4.45,0);
\end{scope}

\draw[fill] (0,-4.5) circle(0.15);

\node at (5.2,-0.9){\phantom{$\widetilde{x}_{0}$}};
\node at (6.3,0){$\to$};
\node at (0,5.3){\phantom{$\widetilde{z}_{0}$}};
\node at (0,-5.4){$-\widetilde{z}_{0}$};
\end{tikzpicture}
\hspace*{\fill}
\begin{tikzpicture}[scale=0.45, >=stealth]
\draw[very thick,middlearrow=0] (0,0) circle(4.5);
\begin{scope}[rotate around={-90:(0,0)}]
\begin{scope}[rotate around={25:(4.5,0)}]
\draw[dotted,middlearrowrev=0.7] (4.5,0) arc (0:180:4 and 1);
\draw[very thick, middlearrow=0.7] (4.5,0) arc (360:180:4.05 and 1);

\draw[thick, middlearrowrev=0.7] (2,-0.9) .. controls (2.25,1) and (4,1) .. (4.45,0);
\draw[thick, middlearrowrev=0.25] (-0.25,-1.03) .. controls (0.25,2.5) and (4.15,2.5) .. (4.45,0);
\draw[thick, middlearrowrev=0.35] (-2.5,-0.74) .. controls (-2,5) and (4.2,5) .. (4.45,0);
\draw[thick, dashed, middlearrowrev=0.5] (-3,0.75) .. controls (-2.75,6.3) and (4.2,6.3) .. (4.45,0);
\draw[thick, dashed, middlearrowrev=0.45] (0,1) .. controls (0,3.5) and (4,4) .. (4.45,0);
\end{scope}
\end{scope}

\draw[fill] (0,-4.5) circle(0.15);

\node at (5.2,-0.9){\phantom{$\widetilde{x}_{0}$}};
\node at (6.3,0){$\to$};
\node at (0,5.3){\phantom{$\widetilde{z}_{0}$}};
\node at (0,-5.4){$-\widetilde{z}_{0}$};
\end{tikzpicture}
\hspace*{\fill}
\begin{tikzpicture}[scale=0.45, >=stealth]
\draw[very thick,middlearrow=0] (0,0) circle(4.5);
\begin{scope}[rotate around={-90:(0,0)}]
\draw[thick, middlearrowrev=0.7] (2,-0.9) .. controls (2,1) and (4,1) .. (4.45,0);
\draw[thick, middlearrowrev=0.45] (-0.25,-1.03) .. controls (-0.25,2) and (4.15,2) .. (4.45,0);
\draw[thick, middlearrowrev=0.5] (-2.5,-0.8) .. controls (-2,3.75) and (4.2,3.6) .. (4.45,0);
\draw[thick, dashed, middlearrowrev=0.5] (-3,0.75) .. controls (-2.75,4.3) and (4.2,4.3) .. (4.45,0);
\draw[thick, dashed, middlearrowrev=0.45] (0,1) .. controls (0.4,2.5) and (4.2,2.5) .. (4.45,0);

\draw[thick] (2,-0.9) .. controls (2.,-1.8) and (4,-1.8) .. (4.45,0);
\draw[thick] (-0.25,-1.03) .. controls (-0.25,-3) and (4.15,-3) .. (4.45,0);
\draw[thick] (-2.5,-0.8) .. controls (-2.5,-4.5) and (4.2,-4.5) .. (4.45,0);
\draw[thick, dashed] (-3,0.75) .. controls (-3.2,-4.3) and (4.2,-4) .. (4.45,0);
\draw[thick, dashed] (0,1) .. controls (-0.3,-2.5) and (4.2,-2.5) .. (4.45,0);
\end{scope}

\draw[fill] (0,-4.5) circle(0.15);

\node at (5.2,0){\phantom{$\widetilde{x}_{0}$}};
\node at (0,5.3){\phantom{$\widetilde{z}_{0}$}};
\node at (0,-5.4){$-\widetilde{z}_{0}$};
\end{tikzpicture}
\hspace*{\fill}
\caption{The three components of $v_{t,\alpha}$ based at $\widetilde{x}_{0}$ (the first row), $\widetilde{z}_{0}$ (the second row), and $-\widetilde{z}_{0}$ (the third row), as $t$ increases from $0$ to $1$. To obtain $v_{t,\alpha}$, the figures in each column should be superimposed. The left-hand column represents $\xi_{n_{0}}$, and the right-hand column represents $-\widehat{\tau}_{n_{0}}+\widetilde{\delta}_{z_{0}}-\widetilde{\delta}_{-z_{0}}$.}
\label{fig:vtalpha}
\end{figure}
We claim that $v_{t,\alpha}\in I_{n_{0}}$. By~\reqref{defOmega}, the arc $\Omega_{\widetilde{x}_{0},1-t, \frac{2-\alpha}{2}}$ (resp.\ $\Omega'_{\widetilde{z}_{0},1-t, \frac{1+\alpha}{2}}$, $\Omega'_{-\widetilde{z}_{0},t, \frac{1+\alpha}{2}}$) joins $\Omega_{\widetilde{x}_{0},1-t, \frac{2-\alpha}{2}}(0)$ (resp.\ $\Omega'_{\widetilde{z}_{0},1-t, \frac{1+\alpha}{2}}(0)$, $\Omega'_{-\widetilde{z}_{0},t, \frac{1+\alpha}{2}}(0)$) to $\widetilde{x}_{0}$ (resp.\ $\widetilde{z}_{0}$, $-\widetilde{z}_{0}$). It thus remains to show that $(\Omega_{\widetilde{x}_{0},1-t, \frac{2-\alpha}{2}}(0), \Omega'_{\widetilde{z}_{0},1-t, \frac{1+\alpha}{2}}(0), \Omega'_{-\widetilde{z}_{0},t, \frac{1+\alpha}{2}}(0))$ belongs to $F_{3}(\St)$. First, for all $t,\alpha\in I$, $\Omega'_{\widetilde{z}_{0},1-t, \frac{1+\alpha}{2}}(0)=-\Omega'_{-\widetilde{z}_{0}, t, \frac{1+\alpha}{2}}(0)$ by \req{minusOmega}, so $\Omega'_{\widetilde{z}_{0},1-t, \frac{1+\alpha}{2}}(0)\neq \Omega'_{-\widetilde{z}_{0}, t, \frac{1+\alpha}{2}}(0)$. We now show that $\Omega_{\widetilde{x}_{0},1-t, \frac{2-\alpha}{2}}(0)\neq \Omega'_{\widetilde{z}_{0},1-t, \frac{1+\alpha}{2}}(0)$ for all $t, \alpha \in I$, or equivalently, that $\widetilde{\omega}_{\widetilde{x}_{0},1-t}({\textstyle\frac{2-\alpha}{2}}) \neq \widetilde{\omega}_{\widetilde{z}_{0},1-t}({\textstyle\frac{1+\alpha}{2}})$, by considering the following cases.
\begin{enumerate}[(a)]
\item If $t\in \brak{0,1}$ then $\widetilde{\omega}_{\widetilde{x}_{0},1-t}({\textstyle\frac{2-\alpha}{2}})=\widetilde{x}_{0}$ and $\widetilde{\omega}_{\widetilde{z}_{0},1-t}({\textstyle\frac{1+\alpha}{2}})=\widetilde{z}_{0}$.

\item If $t\in (0,1)$ and $\alpha=0$ (resp.\ $\alpha=1$) then $\widetilde{\omega}_{\widetilde{x}_{0},1-t}({\textstyle\frac{2-\alpha}{2}})=\widetilde{x}_{0}\neq \widetilde{\omega}_{\widetilde{z}_{0},1-t}({\textstyle\frac{1+\alpha}{2}})$ by~condition~(\ref{it:zb}) (resp.\ $\widetilde{\omega}_{\widetilde{z}_{0},1-t}({\textstyle\frac{1+\alpha}{2}})=\widetilde{z}_{0}\neq \widetilde{\omega}_{\widetilde{x}_{0},1-t}({\textstyle\frac{2-\alpha}{2}})$ by~condition~(\ref{it:xb})).

\item If $\alpha\in (0,1)$ and $t\in (0,1) \setminus \brak{\frac{1}{2}}$ then $p_{2}(\widetilde{\omega}_{\widetilde{x}_{0},1-t}({\textstyle\frac{2-\alpha}{2}})) \ldotp p_{2}(\widetilde{\omega}_{\widetilde{z}_{0},1-t}({\textstyle\frac{1+\alpha}{2}}))<0$ by~conditions~(\ref{it:xa}) and~(\ref{it:za}), in particular $\widetilde{\omega}_{\widetilde{x}_{0},1-t}({\textstyle\frac{2-\alpha}{2}}) \neq \widetilde{\omega}_{\widetilde{z}_{0},1-t}({\textstyle\frac{1+\alpha}{2}})$.
\item suppose that $t=\frac{1}{2}$ and $\alpha\in (0,1)$.
\begin{enumerate}[(i)]
\item if $\alpha\in (0, \frac{1}{2})$ (resp.\ $\alpha\in (\frac{1}{2},1)$) then $p_{3}(\widetilde{\omega}_{\widetilde{x}_{0},\frac{1}{2}}({\textstyle\frac{2-\alpha}{2}}))> 0$ and $p_{3}(\widetilde{\omega}_{\widetilde{z}_{0},\frac{1}{2}}({\textstyle\frac{1+\alpha}{2}}))< 0$ by~conditions~(\ref{it:xb}),~(\ref{it:zc}) and~(\ref{it:ze}) (resp.\ $p_{1}(\widetilde{\omega}_{\widetilde{x}_{0},\frac{1}{2}}({\textstyle\frac{2-\alpha}{2}}))<0$ and $p_{1}(\widetilde{\omega}_{\widetilde{z}_{0},\frac{1}{2}}({\textstyle\frac{1+\alpha}{2}}))> 0$ by~conditions~(\ref{it:xc}),~(\ref{it:xe}) and~(\ref{it:zb})), so $\widetilde{\omega}_{\widetilde{x}_{0},\frac{1}{2}}({\textstyle\frac{2-\alpha}{2}}) \neq \widetilde{\omega}_{\widetilde{z}_{0},\frac{1}{2}}({\textstyle\frac{1+\alpha}{2}})$.

\item if $\alpha=\frac{1}{2}$ then $p_{3}(\widetilde{\omega}_{\widetilde{x}_{0},\frac{1}{2}}({\textstyle\frac{3}{4}}))>0$ and $p_{3}(\widetilde{\omega}_{\widetilde{z}_{0},\frac{1}{2}}({\textstyle\frac{3}{4}}))=0$ by ~conditions~(\ref{it:xb}) and~(\ref{it:ze}), so $\widetilde{\omega}_{\widetilde{x}_{0},\frac{1}{2}}({\textstyle\frac{3}{4}}) \neq \widetilde{\omega}_{\widetilde{z}_{0},\frac{1}{2}}({\textstyle\frac{3}{4}})$.
\end{enumerate}
\end{enumerate}
Finally, we show that $\Omega_{\widetilde{x}_{0},1-t, \frac{2-\alpha}{2}}(0)\neq \Omega'_{-\widetilde{z}_{0},t, \frac{1+\alpha}{2}}(0)$ for all $t, \alpha \in I$, or equivalently, that $\widetilde{\omega}_{\widetilde{x}_{0},1-t}({\textstyle\frac{2-\alpha}{2}}) \neq -\widetilde{\omega}_{\widetilde{z}_{0},1-t}({\textstyle\frac{1+\alpha}{2}})$, by considering the following cases:
\begin{enumerate}[(a)]
\item if $\alpha\in [0,\frac{1}{2})$ (resp.\ $\alpha\in (\frac{1}{2},1]$) then $p_{1}(\widetilde{\omega}_{\widetilde{x}_{0},1-t}({\textstyle\frac{2-\alpha}{2}}))>0$ and $p_{1}(-\widetilde{\omega}_{\widetilde{z}_{0},1-t}({\textstyle\frac{1+\alpha}{2}}))\leq 0$ by~conditions~(\ref{it:xc}),~(\ref{it:xe}) and~(\ref{it:zb}) (resp.\ $p_{3}(\widetilde{\omega}_{\widetilde{x}_{0},1-t}({\textstyle\frac{2-\alpha}{2}}))\geq 0$ and $p_{3}(-\widetilde{\omega}_{\widetilde{z}_{0},1-t}({\textstyle\frac{1+\alpha}{2}}))< 0$ by~conditions~(\ref{it:xb}),~(\ref{it:zc}) and~(\ref{it:ze})), so $\widetilde{\omega}_{\widetilde{x}_{0},1-t}({\textstyle\frac{2-\alpha}{2}})\neq -\widetilde{\omega}_{\widetilde{z}_{0},1-t}({\textstyle\frac{1+\alpha}{2}})$.

\item suppose that $\alpha=\frac{1}{2}$. If $t\in \brak{0,1}$ then $\widetilde{\omega}_{\widetilde{x}_{0},1-t}({\textstyle\frac{3}{4}}) =\widetilde{x}_{0}\neq -\widetilde{z}_{0}=-\widetilde{\omega}_{\widetilde{z}_{0},1-t}({\textstyle\frac{3}{4}})$, and if $t\in (0,1)$ then $p_{3}(\widetilde{\omega}_{\widetilde{x}_{0},1-t}({\textstyle\frac{3}{4}}))>0$ and $p_{3}(-\widetilde{\omega}_{\widetilde{z}_{0},1-t}({\textstyle\frac{3}{4}}))\leq 0$ by~conditions~(\ref{it:xb}),~(\ref{it:zc}) and~(\ref{it:ze}), so $\widetilde{\omega}_{\widetilde{x}_{0},1-t}({\textstyle\frac{3}{4}})\neq -\widetilde{\omega}_{\widetilde{z}_{0},1-t}({\textstyle\frac{3}{4}})$.
\end{enumerate}
This completes the proof of the claim that $v_{t,\alpha}\in I_{n_{0}}$. 

For all $\alpha\in I$ and $t\in \brak{0,1}$, we have $w_{t,\alpha}= v_{t,\alpha}= (\widetilde{x}_{0}, \widetilde{z}_{0}, -\widetilde{z}_{0}, c_{\widetilde{x}_{0}}, c_{\widetilde{z}_{0}}, c_{-\widetilde{z}_{0}})$, hence $(w_{t,\alpha})_{t\in I}$ and $(v_{t,\alpha})_{t\in I}$ are families of based homotopic loops in $I_{n_{0}}$. In particular, $[(w_{t,0})_{t\in I}]= [(w_{t,1})_{t\in I}]$ in $\pi_{1}(I_{n_{0}})$, so by~\relem{boundaryS2}(\ref{it:boundaryS2a}), 
and~\reqref{defOmega},~\reqref{omegah},~\reqref{defdlS2} and~\reqref{defwtalpha}, we have: 
\begin{align}
\br{\widetilde{x}_{0}, \widetilde{z}_{0}, -\widetilde{z}_{0}}=& \bigl[\bigl(\widetilde{\omega}_{\widetilde{x}_{0},t}({\textstyle\frac{1}{2}}), \widetilde{z}_{0}, -\widetilde{z}_{0}, \Omega_{\widetilde{x}_{0},t, \frac{1}{2}}, c_{\widetilde{z}_{0}}, c_{-\widetilde{z}_{0}}\bigr)_{t\in I}\bigr]=\bigl[(w_{t,0})_{t\in I}\bigr] =\bigl[(w_{t,1})_{t\in I}\bigr]\label{eq:wt0}\\ 
=&\bigl[\bigl(\widetilde{x}_0, \widetilde{z}_0, -\widetilde{z}_0, \widetilde{\omega}_{\widetilde{x}_{0}, t}, c_{\widetilde{z}_0}, c_{-\widetilde{z}_0}\bigr)_{t\in I}\bigr]+\notag\\
& \bigl[\bigl(\widetilde{x}_{0}, \widetilde{\omega}_{\widetilde{z}_{0},1-t}({\textstyle\frac{1}{2}}), \widetilde{\omega}_{-\widetilde{z}_{0},t}({\textstyle\frac{1}{2}}), c_{\widetilde{x}_{0}}, \Omega'_{\widetilde{z}_{0},1-t, \frac{1}{2}}, \Omega'_{-\widetilde{z}_{0},t, \frac{1}{2}}\bigr)_{t\in I}\bigr]= \widetilde{\delta}_{\widetilde{x}_{0}}+ \dl{\widetilde{z}_{0}, -\widetilde{z}_{0}},\label{eq:wt1}
\end{align}
which yields~\reqref{breq1}. This equality is illustrated by the first and third columns of Figure~\ref{fig:wtalpha}. Similarly, $[(v_{t,0})_{t\in I}]= [(v_{t,1})_{t\in I}]$ in $\pi_{1}(I_{n_{0}})$. Further, by~\reqref{omegah},  
we see that:
\begin{equation}\label{eq:minustaunhat}
-\br{\widetilde{x}_{0}, \widetilde{z}_{0}, -\widetilde{z}_{0}}=\bigl[\bigl(\widetilde{\omega}_{\widetilde{x}_{0},1-t}({\textstyle\frac{1}{2}}), \widetilde{z}_{0}, -\widetilde{z}_{0}, \Omega_{\widetilde{x}_{0},1-t, \frac{1}{2}}, c_{\widetilde{z}_{0}}, c_{-\widetilde{z}_{0}}\bigr)_{t\in I}\bigr],
\end{equation}
so: 
\begin{align}
\dl{\widetilde{z}_{0}, -\widetilde{z}_{0}} =&\bigl[\bigl( \widetilde{x}_{0}, \widetilde{\omega}_{\widetilde{z}_{0},1-t}({\textstyle\frac{1}{2}}), \widetilde{\omega}_{-\widetilde{z}_{0},t}({\textstyle\frac{1}{2}}), c_{\widetilde{x}_{0}}, \Omega'_{\widetilde{z}_{0},1-t, \frac{1}{2}}, \Omega'_{-\widetilde{z}_{0},t, \frac{1}{2}}\bigr)_{t\in I} \bigr]= \bigl[(v_{t,0})_{t\in I}\bigr]\label{eq:vt0}\\
=& \bigl[(v_{t,1})_{t\in I}\bigr] =\bigl[\bigl( \widetilde{\omega}_{\widetilde{x}_{0},1-t}({\textstyle\frac{1}{2}}), \widetilde{z}_{0}, -\widetilde{z}_{0}, \Omega_{\widetilde{x}_{0},1-t, \frac{1}{2}}, \widetilde{\omega}_{\widetilde{z}_{0},1-t}^{-1}, \widetilde{\omega}_{-\widetilde{z}_{0},t}^{-1}\bigr)_{t\in I} \bigr]\notag\\
=& \bigl[\bigl(\widetilde{\omega}_{\widetilde{x}_{0},1-t}({\textstyle\frac{1}{2}}), \widetilde{z}_{0}, -\widetilde{z}_{0}, \Omega_{\widetilde{x}_{0},1-t, \frac{1}{2}}, c_{\widetilde{z}_{0}}, c_{-\widetilde{z}_{0}}\bigr)_{t\in I}\bigr]+\notag\\
& \bigl[\bigl(\widetilde{x}_{0}, \widetilde{z}_{0}, -\widetilde{z}_{0}, c_{\widetilde{x}_{0}}, \widetilde{\omega}_{\widetilde{z}_{0},1-t}^{-1}, c_{-\widetilde{z}_{0}}\bigr)_{t\in I}\bigr]+ \bigl[\bigl(\widetilde{x}_{0}, \widetilde{z}_{0}, -\widetilde{z}_{0}, c_{\widetilde{x}_{0}}, c_{\widetilde{z}_{0}}, \widetilde{\omega}_{-\widetilde{z}_{0},t}^{-1}\bigr)_{t\in I}\bigr]\notag\\
=& -\br{\widetilde{x}_{0}, \widetilde{z}_{0}, -\widetilde{z}_{0}}+ \bigl[\bigl(\widetilde{x}_{0}, \widetilde{z}_{0}, -\widetilde{z}_{0}, c_{\widetilde{x}_{0}}, \widetilde{\omega}_{\widetilde{z}_{0},t}, c_{-\widetilde{z}_{0}}\bigr)_{t\in I}\bigr]+\bigl[\bigl(\widetilde{x}_{0}, \widetilde{z}_{0}, -\widetilde{z}_{0}, c_{\widetilde{x}_{0}}, c_{\widetilde{z}_{0}}, \widetilde{\omega}_{-\widetilde{z}_{0},t}^{-1}\bigr)_{t\in I}\bigr]\label{eq:vt1}\\
=& -\br{\widetilde{x}_{0}, \widetilde{z}_{0}, -\widetilde{z}_{0}}+\widetilde{\delta}_{\widetilde{z}_{0}}-\widetilde{\delta}_{-\widetilde{z}_{0}}\notag
\end{align}
by~\reqref{defOmega},~\reqref{defdlS2},~\reqref{defvtalpha} and~\reqref{minustaunhat}, and Lemmas~\ref{lem:boundaryS2}(\ref{it:boundaryS2a}) and~\ref{lem:omegaoneminust}, which yields~\reqref{breq2}. This equality is illustrated by the first and third columns of Figure~\ref{fig:vtalpha}. This completes the proof of \reth{brrelnS2} in the case $M=\St$.
\end{proof}

\subsubsection{The proof of \reth{brrelnS2} in the case $M=\rp$}\label{sec:rp2nequals2}

In this section, we use the constructions of the proof of \reth{brrelnS2} in the case $M=\St$ to prove the result in the case $M=\rp$.


\begin{proof}[Proof of \reth{brrelnS2} in the case $M=\rp$] 
Let $I_{2}$ 
denote the homotopy fibre of the map $\map{\iota_{2}}{F_{2}(\rp)}[\prod_{i=1}^{2}\, \rp]$.
Let $\widetilde{I}_{2}$ be the set of elements $(\widetilde{x}, \widetilde{y},\lambda_{1},\lambda_{2})$ of $F_{2}^{\ang{\tau}}(\St)\times \left(\prod_{1}^{2} \St \right)^{I}$ that satisfy the following conditions:
\begin{equation}\label{eq:defI0w}
\text{$(\lambda_{1}(0),\lambda_{2}(0))=(\widetilde{x}, \widetilde{y})$, $\lambda_{1}(1)\in \brak{\widetilde{x}_{0}, -\widetilde{x}_{0}}$ and $\lambda_{2}(1)\in \brak{\widetilde{z}_{0}, -\widetilde{z}_{0}}$,}
\end{equation}
and let $\widetilde{I}_{3}$ be the subset of $I_{3}(\widetilde{x}_{0},\widetilde{z}_{0})$ (defined by~\reqref{defI0uv}) consisting of elements of the form $(\widetilde{x}, \widetilde{y},-\widetilde{y},\lambda_{1},\lambda_{2},-\lambda_{2})$. Then for $(i,j)\in \brak{(2,5),(3,6)}$, the map $\map{\rho_{i,j}}{\widetilde{I}_{3}}[\widetilde{I}_{2}]$ given by forgetting the $i$\textsu{th} and $j$\textsu{th} coordinates is well defined. Also, the following map:
\begin{equation*}
\left\{
\begin{aligned}
\widetilde{\pi} \colon\thinspace F_{2}^{\ang{\tau}}(\St)\times \Biggl(\prod_{1}^{2} \St \Biggr)^{I} &\to F_{2}(\rp)\times \Biggl(\prod_{1}^{2} \rp \Biggr)^{I}\\
(\widetilde{x} ,\widetilde{y},\lambda_{1},\lambda_{2})& \mapsto (\pi(\widetilde{x}), \pi(\widetilde{y}), \pi\circ \lambda_{1}, \pi\circ \lambda_{2})
\end{aligned}\right.
\end{equation*}
induced by the projection $\map{\pi}{\St}[\rp]$, restricts to a map from $\widetilde{I}_{2}$ to $I_{2}$, that we also denote by $\widetilde{\pi}$. For all $t, \alpha\in I$, let: 
\begin{align}
w'_{t,\alpha} &=\bigl( \Omega_{\widetilde{x}_{0},t,\frac{1-\alpha}{2}}(0), \Omega_{-\widetilde{z}_{0},t,\frac{\alpha}{2}}'(0), \Omega_{\widetilde{x}_{0},t,\frac{1-\alpha}{2}}, \Omega_{-\widetilde{z}_{0},t,\frac{\alpha}{2}}' \bigr),\notag\\
v'_{t,\alpha} &=\bigl( \Omega_{-\widetilde{x}_{0},t,\frac{2-\alpha}{2}}(0), \Omega_{-\widetilde{z}_{0},t,\frac{1+\alpha}{2}}'(0), \Omega_{-\widetilde{x}_{0},t,\frac{2-\alpha}{2}}, \Omega_{-\widetilde{z}_{0},t,\frac{1+\alpha}{2}}' \bigr) \; \text{and}\notag\\
\brrp &= \bigl[\bigl( x_{0},
\pi\circ \widetilde{\omega}_{-\widetilde{z}_{0},t}\left( {\textstyle\frac{1}{2}} \right),
c_{x_{0}}, \pi\circ \Omega_{-\widetilde{z}_{0},t,\frac{1}{2}}' 
\bigr)_{t\in I}\bigr]\; \text{in $\pi_{1}(I_{2})$.}\label{eq:defbrx0z0}
\end{align}
By equations~\reqref{minusOmega},~\reqref{defwtalpha} and~\reqref{defvtalpha}, for all $t, \alpha\in I$, we have $w_{t,\alpha}, v_{t,\alpha}\in \widetilde{I}_{3}$, $w'_{t,\alpha}=\rho_{2,5}(w_{t,\alpha})\in \widetilde{I}_{2}$ and $v'_{t,\alpha}=-\rho_{3,6}(v_{t,\alpha})\in \widetilde{I}_{2}$, so $\widetilde{\pi}(w'_{t,\alpha})$ and $\widetilde{\pi}(v'_{t,\alpha})$ belong to $I_{2}$. In particular, composing the homotopy $(w_{t,\alpha})_{t\in I}$ (resp.\ $(v_{t,\alpha})_{t\in I}$) in $\widetilde{I}_{3}$ (and so in $I_{3}(\St)$) with the map $\widetilde{\pi}\circ \rho_{2,5}$ (resp.\ $\widetilde{\pi}\circ (-\rho_{3,6})=\widetilde{\pi}\circ \rho_{3,6}$) yields the homotopy $(\widetilde{\pi}\circ w'_{t,\alpha})_{t\in I}$ (resp.\ $(\widetilde{\pi}\circ v'_{t,\alpha})_{t\in I}$) in $I_{2}$.
Now by~\reqref{minusomega} and~\reqref{defOmega}, we have:
\begin{equation}\label{eq:omegaz0}
\bigl[\bigl(\Omega_{-\widetilde{z}_{0},t,1}'\bigr)_{t\in I}\bigr]= \bigl[\bigl(\widetilde{\omega}_{-\widetilde{z}_{0},t}^{-1}\bigr)_{t\in I}\bigr]=  \bigl[\bigl(-\widetilde{\omega}_{\widetilde{z}_{0},1-t}^{-1}\bigr)_{t\in I}\bigr] = \bigl[\bigl(-\widetilde{\omega}_{\widetilde{z}_{0},t}\bigr)_{t\in I}\bigr]
\end{equation}
in $\pi_{2}(\St,c_{-\widetilde{z}_{0}})$. Applying $\widetilde{\pi}\circ \rho_{2,5}$ (resp.\ $\widetilde{\pi}\circ \rho_{3,6}$) to equations~\reqref{wt0} and~\reqref{wt1} (resp.~\reqref{vt0} and~\reqref{vt1}) and using equations~\reqref{minusOmega},~\reqref{omegah},~\reqref{defbrx0z0} and~\reqref{omegaz0} as well as \relem{boundaryS2}(\ref{it:boundaryS2b}), in $\pi_{1}(I_{2})$ we obtain respectively: 
\begin{align}
\widehat{\tau}_{2} &= \bigl[\bigl(\pi\circ \widetilde{\omega}_{\widetilde{x}_{0},t}({\textstyle\frac{1}{2}}),  z_{0}, \pi\circ \Omega_{\widetilde{x}_{0},t, \frac{1}{2}}, c_{z_{0}}\bigr)_{t\in I}\bigr]=\bigl[(\widetilde{\pi}\circ w'_{t,0})_{t\in I}\bigr] 
=\bigl[(\widetilde{\pi}\circ w'_{t,1})_{t\in I}\bigr]\notag\\
& =\bigl[\bigl(x_0,  z_0, \pi\circ \widetilde{\omega}_{\widetilde{x}_{0}, t},  c_{z_0}\bigr)_{t\in I}\bigr]+ \bigl[\bigl(x_{0}, \pi\circ \widetilde{\omega}_{-\widetilde{z}_{0},t}({\textstyle\frac{1}{2}}), c_{x_{0}},  \pi\circ \Omega'_{-\widetilde{z}_{0},t, \frac{1}{2}}\bigr)_{t\in I}\bigr]\notag\\
&= \delta_{x_{0}}+ \brrp, \;\text{and}\label{eq:br1}\\
\brrp &=\bigl[\bigl( x_{0}, \pi\circ \widetilde{\omega}_{-\widetilde{z}_{0},t}\left( {\textstyle\frac{1}{2}} \right), c_{x_{0}}, \pi\circ \Omega_{-\widetilde{z}_{0},t,\frac{1}{2}}' \bigr)_{t\in I}\bigr]= \bigl[(\widetilde{\pi}\circ v'_{t,0})_{t\in I}\bigr]= \bigl[(\widetilde{\pi}\circ v'_{t,1})_{t\in I}\bigr]\notag\\
&= \bigl[\bigl( \pi\circ\widetilde{\omega}_{\widetilde{x}_{0},1-t}({\textstyle\frac{1}{2}}), z_{0},  \pi\circ\Omega_{\widetilde{x}_{0},1-t, \frac{1}{2}}, c_{z_{0}}\bigr)_{t\in I}\bigr]+ \bigl[\bigl(x_{0}, z_{0},  c_{x_{0}}, \pi\circ\widetilde{\omega}_{\widetilde{z}_{0},t} \bigr)_{t\in I}\bigr]\notag\\
&= -\widehat{\tau}_{2}+ \delta_{z_{0}}.\label{eq:br2}
\end{align}
Summing equations~\reqref{br1} and~\reqref{br2} yields the result in the case $M=\rp$.
\end{proof}

\begin{rem}
Within the framework of this section ($M=\rp$ and $n_{0}=2$), it is important for us to know that the homomorphism $\map{(\alpha_{2}')_{\#}}{\pi_{1}(I_{2}')}[\pi_{1}(I_{2})]$ is an isomorphism. This is a consequence of \reco{alphanprime}. However, we may give an alternative proof of this fact without using \reco{alphanprime} as follows.
First note that $\pi_{1}(I_{2}') \cong \pi_{1}(I_{c})\cong \Z^{2}$ by \relem{fundhomoequiv}(\ref{it:fundhomoequiva}) and~(\ref{it:fundhomoequivb}), 
and that $\pi_{1}(I_{2})\cong \Z^{2}$ by \reth{prop5}(\ref{it:prop5b}) and \repr{presgn}. Arguing as in the proof of \repr{genspi1S2}, one sees that $\im{(\alpha_{2}')_{\#}}=\ang{\delta_{x_{0}}, \br{x_{0},{z}_{0}}}$. Since $\Z^{2}$ is Hopfian, it thus suffices to prove that $(\alpha_{2}')_{\#}$ is surjective. To do so, consider the long exact sequence~\reqref{lespartialn} for $M=\rp$ and $n=2$. Now $P_{2}(\rp)$ is isomorphic to the quaternion group of order~$8$, and the full twist $\ft[2]$ is its unique element of order~$2$~\cite{vB}.
By exactness, it follows that $\ker{(\widehat{\iota}_{2})_{\#}}=\ang{\ft[2]}$, and that the following sequence: 
\begin{equation}\label{eq:shortespartial}
\begin{tikzcd}[cramped]
1 \arrow{r} & \pi_{2}(\rp\times \rp) \arrow[r, "\partial_{2}"] & \pi_{1}(I_{2}) \arrow[rr, "(g_{2}\circ j_{2})_{\#}"]
&& \ker{(\widehat{\iota}_{2})_{\#}} \arrow{r}& 1.
\end{tikzcd}
\end{equation}
is exact. By \repr{deftaunhat}, $(g_{2}\circ j_{2})_{\#}(\widehat{\tau}_{2})=\ft[2]$, and
using standard properties of short exact sequences applied to~\reqref{shortespartial}, \reth{brrelnS2}, and the fact that $\partial_{2}(\lambda_{x_{0}})=\delta_{x_{0}}$ and $\partial_{2}(\lambda_{z_{0}})=\delta_{z_{0}}$, we see that $\pi_{1}(I_{2})=\ang{\delta_{x_{0}}, \delta_{z_{0}}, \br{x_{0},{z}_{0}}}=\ang{\delta_{x_{0}}, \br{x_{0},{z}_{0}}}$, and thus the homomorphism $\map{(\alpha_{2}')_{\#}}{\pi_{1}(I_{2}')}[\pi_{1}(I_{2})]$ is surjective as required. 
\end{rem}

\subsection{The boundary homomorphism in the general case}\label{sec:general}

In this section, we determine the boundary homomorphism in the general case using the conclusions of \reth{brrelnS2} and \reco{sesnn0} in the case $n=n_{0}$. Let $M=\St$ or $\rp$, suppose that $n>n_{0}$, and let $j\in \brak{0,1,\ldots, n-n_{0}}$. We take the basepoint of $F_{n}(M)$ and $\prod_{1}^{n}\, M$ to be $W_{n}$ as defined in \resec{generalities}. Let $\map{\nu_{j}}{F_{n}(M)}[F_{n_{0}}(M)]$ and $\map{\widetilde{\nu}_{j}}{\prod_{1}^{n}\, M}[\prod_{i=1}^{n_{0}}\, M]$ denote projection onto the $(j+1)$\textsu{st}, $(n-1)$\textsu{th} and $n$\textsu{th} coordinates (resp.\ $(j+1)$\textsu{st} and $n$\textsu{th} coordinates). 
We thus have a commutative diagram:
\begin{equation*}
\begin{tikzcd}[cramped]
F_{n}(M) \arrow{r}{\iota_{n}} \arrow[d, "\nu_{j}"]
& \prod_{1}^{n}\, M \arrow[d, "\widetilde{\nu}_{j}"]\\
F_{n_{0}}(M) \arrow{r}{\iota_{n_0}}
& \prod_{i=1}^{n_{0}}\, M.
\end{tikzcd}
\end{equation*}
In order that the maps $\nu_{j}$ and $\widetilde{\nu}_{j}$ be pointed, the basepoint of $F_{n_{0}}(M)$ and $\prod_{i=1}^{n_{0}}\, M$ is taken to be 
$\nu_{j}(W_{n})=\widetilde{\nu}_{j}(W_{n})$.
Applying~\cite[pages~91~and~108]{A}, we obtain the following commutative diagram of fibrations:
\begin{equation}\label{eq:commdiagfib0}
\begin{tikzcd}[cramped]
I_{n} \arrow[r, hook, "j_{n}"] \arrow[d, "\widehat{\nu}_{j}\left\lvert_{I_{n}}\right."]
& E_{n} \arrow{r}{\widehat{\iota}_{n}} \arrow[d, "\widehat{\nu}_{j}"]
& \prod_{1}^{n}\, M \arrow[d, "\widetilde{\nu}_{j}"]\\
I_{n_{0}} \arrow[r, hook, "j_{n_{0}}"]
& E_{n_{0}} \arrow{r}{\widehat{\iota}_{n_{0}}}
& \prod_{i=1}^{n_{0}}\, M,
\end{tikzcd}
\end{equation}
where $\map{\widehat{\nu}_{j}}{E_{n}}[E_{n_{0}}]$ is the projection induced by $\nu_{j}$. The basepoints of $E_{n}$ and $I_{n}$, (resp.\ $E_{n_{0}}$ and $I_{n_{0}}$) are those obtained from $\prod_{1}^{n}\, M$ (resp.\ $\prod_{i=1}^{n_{0}}\, M$) using the convention given in the first paragraph of  \resec{generalities}. 
Taking the long exact sequence in homotopy of~\reqref{commdiagfib0} as in \req{lespartialn}, 
we obtain
the following commutative diagram of exact sequences:
\begin{equation}\label{eq:longessgen}
\begin{tikzcd}[cramped, sep=normal]
1 \arrow[r] & \pi_{2}(\prod_{1}^{n}\, M) \arrow[r, "\partial_{n}"] \arrow{d}{(\widetilde{\nu}_{j})_{\#2}} & \pi_{1}(I_{n}) \arrow{r}{(j_{n})_{\#}}\arrow[d, "(\widehat{\nu}_{j}\left\lvert_{I_{n}}\!\right.)_{\#}"]
& \pi_{1}(E_{n}) \arrow{r}{(\widehat{\iota}_{n})_{\#}} \arrow[d, "(\widehat{\nu}_{j})_{\#}"]
& \pi_{1}(\prod_{1}^{n}\, M) \arrow[d, "(\widetilde{\nu}_{j})_{\#1}"] \arrow{r} & 1\\
1 \arrow{r} & \pi_{2}(\prod_{i=1}^{n_{0}}\, M) \arrow[r, "\partial_{n_{0}}"]  & \pi_{1}(I_{n_{0}}) \arrow[r, "(j_{n_{0}})_{\#}"]
& \pi_{1}(E_{n_{0}}) \arrow[r, "(\widehat{\iota}_{n_{0}})_{\#}"]
& \pi_{1}(\prod_{i=1}^{n_{0}}\, M) \arrow{r}& 1.
\end{tikzcd}
\end{equation}
The rows of~\reqref{longessgen} are short exact if $M=\St$.
With the notation of equations~\reqref{defI0uv} and~\reqref{defI0uvrp2}, $I_{n_{0}}$ here denotes $I_{n_{0}}(\widetilde{x}_{j}, \widetilde{z}_{0})$ (resp.\ $I_{n_{0}}(x_{j}, z_{0})$). We now prove \reth{tauhatsquareS2}, which expresses $\widehat{\tau}_{n}$ in terms of the elements that appear in~\reqref{basispi2}, and generalises \reth{brrelnS2}. Since $\pi_1(I_{n})$ is non Abelian if $n\geq 4$ (resp.\ $n\geq 3$), we write $\widehat{\tau}_{n}^{2}$ rather than $2\widehat{\tau}_{n}$.

\begin{proof}[Proof of \reth{tauhatsquareS2}]
Let $j\in \brak{0,1,\ldots,n-n_{0}}$. By \req{commdeltanS2} and the commutativity of diagram~\reqref{longessgen}, in $\pi_{1}(I_{n_{0}})$ we have: 
\begin{align}
(\widehat{\nu}_{j}\left\lvert_{I_{n}}\right.\!)_{\#}(\widehat{\tau}_{n}^{2})&= \partial_{n_{0}}\circ (\widehat{\nu}_{j})_{\#2}(m_{0}\widetilde{\lambda}_{\widetilde{x}_0}+\cdots+m_{n-3} \widetilde{\lambda}_{\widetilde{x}_{n-3}} +m_{n-2} \widetilde{\lambda}_{\widetilde{z}_{0}}+m_{n-1} \widetilde{\lambda}_{-\widetilde{z}_{0}})\notag\\
&= m_{j} \partial_{n_{0}}(\widetilde{\lambda}_{\widetilde{x}_j})+ m_{n-2} \partial_{n_{0}}(\widetilde{\lambda}_{\widetilde{z}_{0}})+m_{n-1} \partial_{n_{0}}(\widetilde{\lambda}_{-\widetilde{z}_{0}})\; \text{if $M=\St$, and}\label{eq:sumpartialS2}\\
(\widehat{\nu}_{j}\left\lvert_{I_{n}}\right.\!)_{\#}(\widehat{\tau}_{n}^{2})&= \partial_{n_{0}}\circ (\widehat{\nu}_{j})_{\# 2}(m_{0}\lambda_{x_0}+\cdots+m_{n-2} \lambda_{x_{n-2}} +m_{n-1} \lambda_{z_{0}})\notag\\
&= m_{j} \partial_{n_{0}}(\lambda_{x_j})+ m_{n-1} \partial_{n_{0}}(\lambda_{z_{0}})\; \text{if $M=\rp$.}\label{eq:sumpartial}
\end{align}
From~\reqref{psihone},~\reqref{psihat},~\reqref{defyns},~\reqref{tauhat} and the fact that $h_{\widetilde{z}_{0}}^{(2)}=c_{\widetilde{z}_{0}}$, we obtain:
\begin{equation}\label{eq:brxiz0S2}
\textstyle(\widehat{\nu}_{j}\left\lvert_{I_{n}}\right.\!)_{\#}(\widehat{\tau}_{n})=
\begin{cases}
\textstyle\bigl[\bigl(\widetilde{\omega}_{\widetilde{x}_{j},t}(\frac{1}{2}), \widetilde{z}_{0}, -\widetilde{z}_{0}, 
h_{j,t}, c_{\widetilde{z}_{0}}, c_{-\widetilde{z}_{0}}\bigr)_{t\in I}\bigr] & \text{if $M=\St$}\\
\bigl[\bigl(\pi(\widetilde{\omega}_{\widetilde{x}_{j},t}(\textstyle\frac{1}{2})), z_{0}, \pi\circ h_{j,t}, c_{z_{0}}\bigr)_{t\in I}\bigr] & \text{if $M=\rp$.}
\end{cases}
\end{equation}

Let $\map{\widetilde{K}_{j}}{\St}$ denote the rotation of angle $\frac{\pi j}{2(j+1)}$ about the $y$-axis, illustrated in  Figure~\ref{fig:mapKj}, that sends $\widetilde{x}_{j}$ to $\widetilde{x}_{0}$, let $\map{{K}_{j}}{\rp}$ be the homeomorphism of $\rp$ induced by $\widetilde{K}_{j}$, let $\widetilde{z}_{j}=\widetilde{K}_{j}(\widetilde{z}_{0})$, and let $z_{j}=\pi(\widetilde{z}_{j})$.
\begin{figure}[t]
\hspace*{\fill}
\begin{tikzpicture}[scale=0.75, >=stealth]
\draw[thick, dotted] (0,8) -- (3,10.5); 
\draw[fill, white] (2.4,9.7) circle(0.32); 

%
%

%
%
%
%

\draw[very thick] (0,8) circle(3);
\draw[dashed, very thick, middlearrow=0.4] (2.6,9.5) arc (0:180:2.6 and 0.6); 
\draw[fill, white] (0.81,8.81) circle(0.3); 
\draw[very thick, middlearrow=0.4] (-2.6,9.5) arc (180:360:2.6 and 0.8); 

\draw[fill] (3,8) circle(0.1);
\draw[fill] (0,11) circle(0.1);
\draw[fill] (2.6,9.5) circle(0.1);
\draw[fill] (0,8) circle(0.05);

\node at (2.4,8){$\widetilde{x}_{0}$};
\node at (0.5,11.5){$\widetilde{z}_{0}$};
\node at (4.2,9.8){$\widetilde{x}_{j}$};

\draw[fill, white] (-2.29,6.09) circle(0.1); 
\draw[thick] (0,8) -- (-3.5,5.08); 

\draw[fill, white] (-2.83,5.58) circle(0.13); 
\draw [very thick, ->, rotate around={130:(-3.65,6.25)}] (-3.65,6.25) arc [radius=3, start angle=290, end angle= 250]; 
\node at (-3.1,4.5){$\frac{\pi j}{2(j+1)}$};

\draw[thick, ->] (3.4,8) -- (5.6,8);
\node at (4.5,8.6){$\widetilde{K}_{j}$};

\draw[very thick] (9,8) circle(3); 
\draw[dotted,middlearrow=0.87] (12,8) arc (0:180:3 and 1); 
\draw[very thick, middlearrow=0.8] (6,8) arc (180:360:3 and 1); 

\draw[dashed, very thick, middlearrow=0.85, rotate around={-29:(12,8)}] (12,8) arc (0:180:2.6 and 0.6);
\draw[very thick, middlearrow=0.65, rotate around={-29:(12,8)}] (6.8,8) arc (180:360:2.6 and 0.8); 

\draw[middlearrowrev=0.7] (10.4,7.1) .. controls (10.25,7.3) and (10.5,7.9) .. (10.6,8);
\draw[middlearrowrev=0.7] (9.4,7) .. controls (9.25,7.3) and (9.2,7.9) .. (9.65,8.4);
\draw[middlearrowrev=0.7] (8.4,7) .. controls (8.25,7.3) and (8,8.3) .. (8.8,8.9);
\draw[middlearrowrev=0.55] (7.4,7.2) .. controls (7.25,7.3) and (7,8.7) .. (8,9.5);

\draw[middlearrowrev=0.7] (9.9,8.95) .. controls (9.75,9.1) and (10,9.7) .. (10.1,9.8);

\draw[middlearrowrev=0.7] (9,9) .. controls (8.85,9.1) and (8.75,9.75) .. (9.18,10.18);

\draw[fill] (12,8) circle(0.1);
\draw[fill] (9,11) circle(0.1);
\draw[fill] (10.5,10.6) circle(0.1);
\node at (13.8,8){$\widetilde{x}_{0}=\widetilde{K}_{j}(\widetilde{x}_{j})$};
\node at (9,11.5){$\widetilde{z}_{0}$};
\node at (12.1,11.1){$\widetilde{z}_{j}=\widetilde{K}_{j}(\widetilde{z}_{0})$};
\node at (8.4,6.55){$(\zeta_{t}(r))_{r\in [0,\frac{1}{2}]}$};

\end{tikzpicture}
\hspace*{\fill}
\caption{The homeomorphism $\widetilde{K}_{j}$ and the paths $\zeta_{t}$.}
\label{fig:mapKj}
\end{figure}
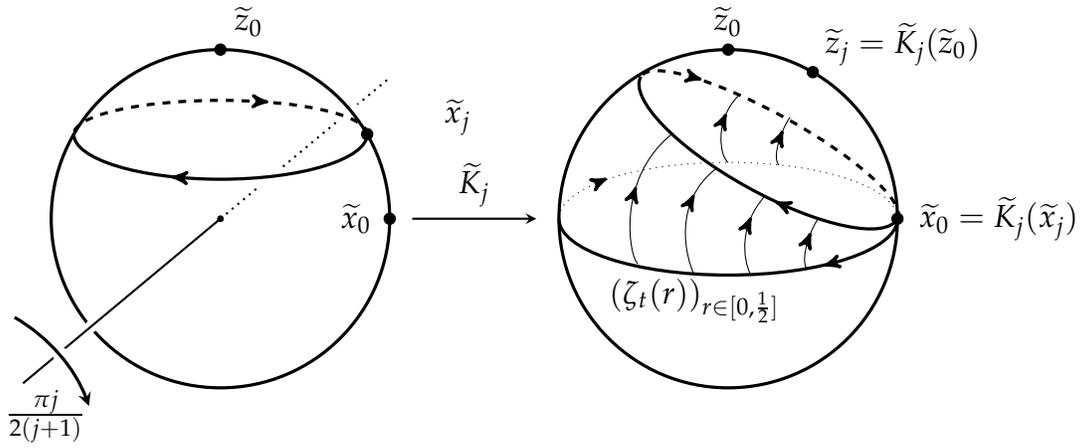
The homeomorphisms $\widetilde{K}_{j}$ and $K_{j}$ induce isomorphisms:
\begin{equation*}
\text{$\map{(\widetilde{K}_{j})_{\#}}{\pi_{1}(I_{n_{0}})}[\pi_{1}(I_{n_{0}}(\widetilde{x}_{0}, \widetilde{z}_{j}))]$ and $\map{(K_{j})_{\#}}{\pi_{1}(I_{n_{0}})}[\pi_{1}(I_{n_{0}}(x_{0}, z_{j}))]$}
\end{equation*}
respectively, that satisfy:
\begin{align}
\hspace*{-3mm}
(\widetilde{K}_{j})_{\#}\left(\textstyle(\widehat{\nu}_{j}\left\lvert_{I_{n}}\right.\!)_{\#}(\widehat{\tau}_{n})\right)\!&=\!
\textstyle \bigl[\bigl( \widetilde{K}_{j}(\widetilde{\omega}_{\widetilde{x}_{j},t}(\frac{1}{2})), \widetilde{z}_{j}, -\widetilde{z}_{j}, \widetilde{K}_{j} \circ h_{j,t}, c_{\widetilde{z}_{j}}, c_{-\widetilde{z}_{j}}\bigr)_{t\in I}\bigr] && \text{if $M=\St$}\label{eq:kjtaunS2}\\
(K_{j})_{\#}\left(\textstyle(\widehat{\nu}_{j}\left\lvert_{I_{n}}\right.\!)_{\#}(\widehat{\tau}_{n})\right)&= \bigl[\bigl(K_{j}\bigl(\pi(\widetilde{\omega}_{\widetilde{x}_{j},t}(\textstyle\frac{1}{2}))\bigr), z_{j}, K_{j}\circ \pi\circ h_{j,t}, c_{z_{j}}\bigr)_{t\in I}\bigr] &&\text{if $M=\rp$}\label{eq:kjtaun}
\end{align}
by \req{brxiz0S2}. For $t\in I$, consider the path $\map{\zeta_{t}}{I}[\St]$ defined by:
\begin{equation}\label{eq:defzeta}
\zeta_{t}(r)=
\frac{\max\left(0,(\frac{1}{2}-r)\right)\widetilde{\omega}_{\widetilde{x}_{0},t}(\frac{1}{2})+r \widetilde{K}_{j}\bigl(\widetilde{\omega}_{\widetilde{x}_{j},t}\left(\max(\frac{1}{2},r)\right)\bigr)}{\bigl\lVert \max\left(0,(\frac{1}{2}-r)\right)\widetilde{\omega}_{\widetilde{x}_{0},t}(\frac{1}{2})+r \widetilde{K}_{j}\bigl(\widetilde{\omega}_{\widetilde{x}_{j},t}\left(\max(\frac{1}{2},r)\right)\bigr)\bigr\rVert}\quad\text{for all $r\in I$.}
\end{equation}
We claim that $\zeta_{t}$ is well defined. To see this, first note that if $r\in [\frac{1}{2},1]$ then $\zeta_{\frac{1}{2}}(r)=\widetilde{K}_{j}(\widetilde{\omega}_{\widetilde{x}_{j},\frac{1}{2}}(r))$. We may thus  assume that $r\in [0,\frac{1}{2})$, so that $\widetilde{K}_{j}\bigl(\widetilde{\omega}_{\widetilde{x}_{j},t}\left(\max(\frac{1}{2},r)\right)\bigr)=\widetilde{K}_{j}\circ J_{j}\bigl(\widetilde{\omega}_{\widetilde{x}_{0},t}\left(\frac{1}{2}\right)\bigr)$ using~\reqref{hxjt}. If $t\in \brak{0,1}$ then $\zeta_{t}=c_{\widetilde{x}_{0}}$. If $t=\frac{1}{2}$, conditions~(\ref{it:xa}),~(\ref{it:xb}) and~(\ref{it:xe}) imply that $\widetilde{\omega}_{\widetilde{x}_{0},\frac{1}{2}}(\frac{1}{2})=-\widetilde{x}_{0}$. Hence $\widetilde{\omega}_{\widetilde{x}_{0},\frac{1}{2}}(\frac{1}{2})$ and $\widetilde{K}_{j}\circ J_{j}(\widetilde{\omega}_{\widetilde{x}_{0},\frac{1}{2}}(\frac{1}{2}))$ are non antipodal, for otherwise we would have $\widetilde{K}_{j}\circ J_{j}(-\widetilde{x}_{0})=\widetilde{x}_{0}$, which is impossible since $\widetilde{x}_{0}$ is the only preimage of itself by the homeomorphism $\widetilde{K}_{j}\circ J_{j}$. So suppose that $t\notin \brak{0,\frac{1}{2},1}$, and let $v=\widetilde{\omega}_{\widetilde{x}_{0},t}(\frac{1}{2})$. Now $p_{2}(v)$, $p_{2}(J_{j}(v))$, and $p_{2}(\widetilde{K}_{j}\circ J_{j}(v))$ are all non zero and of the same sign, so $v$ and $\widetilde{K}_{j}\circ J_{j}(v)$ are non antipodal also, and this proves the claim.
Each path $\zeta_{t}$ joins $\widetilde{\omega}_{\widetilde{x}_{0},t}(\frac{1}{2})$ to $\widetilde{x}_{0}$ via $\widetilde{K}_{j}(\widetilde{\omega}_{\widetilde{x}_{j},t}(\frac{1}{2}))$ by arcs that lie in the upper hemisphere. The subpaths $(\zeta_{t}(r))_{r\in [0,\frac{1}{2}]}$ are illustrated in the right-hand part of Figure~\ref{fig:mapKj}. From~\reqref{defzeta}, one may check that $\zeta_{t}$ satisfies the hypotheses of \relem{deformwS2}, where we take $\Gamma(r,t)=\zeta_{t}(r)$ for all $t,r\in I$, and $\widetilde{w}=\widetilde{z}_{j}$, from which we conclude that:
\begin{align}
\textstyle \bigl[ \bigl( \zeta_{t}(0), \widetilde{z}_{j}, -\widetilde{z}_{j}, \zeta_{t}, c_{\widetilde{z}_{j}}, c_{-\widetilde{z}_{j}}\bigr)_{t\in I}\bigr]&= \bigl[ \bigl( \widetilde{\omega}_{\widetilde{x}_{0},t}\left(\textstyle\frac{1}{2}\right), \widetilde{z}_{j}, -\widetilde{z}_{j}, \Omega_{\widetilde{x}_{0},t,\frac{1}{2}}, c_{\widetilde{z}_{j}}, c_{-\widetilde{z}_{j}}\bigr)_{t\in I}\bigr]\; \text{and}\label{eq:equalclassS2}\\
\textstyle \bigl[ \bigl( \pi(\zeta_{t}(0)), z_{j}, \pi\circ \zeta_{t}, c_{z_{j}}\bigr)_{t\in I}\bigr]&= \bigl[ \bigl( \pi\left(\widetilde{\omega}_{\widetilde{x}_{0},t}\left(\textstyle\frac{1}{2}\right)\right), z_{j}, \pi\circ \Omega_{\widetilde{x}_{0},t,\frac{1}{2}}, c_{z_{j}}\bigr)_{t\in I}\bigr]\label{eq:equalclass}
\end{align}
in $\pi_{1}(I_{n_{0}}(\widetilde{x}_{0}, \widetilde{z}_{j}))$ and in $\pi_{1}(I_{n_{0}}(x_{0}, z_{j}))$ respectively.
Setting $\eta_{t,r}(s)=\zeta_{t}((1-s)r+s)$ for all $t,r,s\in I$, we have $\eta_{t,0}=\zeta_{t}$, $\eta_{t,r}(0)=\zeta_{t}(r)$, $\eta_{t,r}(1)=\zeta_{t}(1)=\widetilde{x}_{0}$, and using~\reqref{hxjt} and~\reqref{defzeta}, we obtain:
\begin{equation}\label{eq:etahalf}
\text{$\eta_{t,\frac{1}{2}}(s)= \zeta_{t}\bigl(\textstyle\frac{1+s}{2}\bigr)= \widetilde{K}_{j}(\widetilde{\omega}_{\widetilde{x}_{j},t}(\textstyle\frac{1+s}{2}))=\widetilde{K}_{j} \circ h_{j,t}(s)$ for all $s\in I$, so $\eta_{t,\frac{1}{2}}=\widetilde{K}_{j} \circ h_{j,t}$.}
\end{equation}
We claim that $\eta_{t,r}(0)\notin \brak{\widetilde{z}_{j}, -\widetilde{z}_{j}}$ for all $r\in [0,\frac{1}{2}]$. To see this, first recall that $\eta_{t,r}(0)= \zeta_{t}(r)$ 
lies in the upper hemisphere, so $\eta_{t,r}(0)\neq -\widetilde{z}_{j}$. If $t\notin\brak{0,\frac{1}{2},1}$ then since $\eta_{t,0}(0)=\widetilde{\omega}_{\widetilde{x}_{0},t}(\textstyle\frac{1}{2})$ and $\eta_{t,\frac{1}{2}}(0)=\widetilde{K}_{j}\circ J_{j}(\widetilde{\omega}_{\widetilde{x}_{0},t}(\frac{1}{2}))$,
we deduce from above that $p_{2}(\eta_{t,0}(0))$ and $p_{2}(\eta_{t,\frac{1}{2}}(0))$ are non zero and have the same sign, which proves the claim in this case since $p_{2}(\widetilde{z}_{j})=0$. If $t\in \brak{0,1}$ then $\eta_{t,r}(0)=\widetilde{x}_{0}$ for all $r\in [0,\frac{1}{2}]$, so the claim also holds. Finally, let $t=\frac{1}{2}$. If $j=0$ then $\eta_{t,r}(0)= \widetilde{\omega}_{\widetilde{x}_{0},\frac{1}{2}}(\frac{1}{2})= -\widetilde{x}_{0}\neq \widetilde{z}_{0}$ for all $r\in [0,\frac{1}{2}]$. If $j\geq 1$, then with respect to the spherical coordinates of \resec{generalities}, the arc $(\eta_{\frac{1}{2},r}(0))_{r\in I}$ joins $-\widetilde{x}_{0}$ to $\widetilde{K}_{j}\circ J_{j}(\widetilde{\omega}_{\widetilde{x}_{0},\frac{1}{2}}\left(\textstyle\frac{1}{2}\right))=(0, \frac{\pi}{j+1})$ along the geodesic arc that avoids $\widetilde{z}_{j}=(0,\frac{\pi}{2(j+1)})$. This completes the proof of the claim. So in $\pi_{1}(I_{n_{0}}(\widetilde{x}_{0},\widetilde{z}_{j}))$, we obtain:
\begin{align}
\hspace*{-3mm}\bigl[\bigl( \widetilde{K}_{j}(\widetilde{\omega}_{\widetilde{x}_{j},t}(\textstyle\frac{1}{2})), \widetilde{z}_{j}, -\widetilde{z}_{j}, \widetilde{K}_{j} \!\circ\! h_{j,t}, c_{\widetilde{z}_{j}}, c_{-\widetilde{z}_{j}}\bigr)_{t\in I}\bigr] & \!\!=\!\!  \textstyle\bigl[ \!\bigl( \eta_{t,\frac{1}{2}}(0), \widetilde{z}_{j}, -\widetilde{z}_{j}, \eta_{t,\frac{1}{2}},c_{\widetilde{z}_{j}}, c_{-\widetilde{z}_{j}}  \bigr)_{t\in I} \bigr]\notag\\
& \!\!=\!\!  \textstyle\bigl[\! \bigl( \eta_{t,0}(0), \widetilde{z}_{j}, -\widetilde{z}_{j}, \eta_{t,0},c_{\widetilde{z}_{j}}, c_{-\widetilde{z}_{j}}  \bigr)_{t\in I} \bigr]\notag\\
& \!\!=\!\!\bigl[\! \bigl( \widetilde{\omega}_{\widetilde{x}_{0},t}\!\left(\textstyle\frac{1}{2}\right)\!, \widetilde{z}_{j}, -\widetilde{z}_{j}, \Omega_{\widetilde{x}_{0},t,\frac{1}{2}}, c_{\widetilde{z}_{j}}, c_{-\widetilde{z}_{j}}\bigr)_{t\in I}\bigr]\label{eq:kjomegaS2}
\end{align}
using equations~\reqref{hxjt},~\reqref{defzeta},~\reqref{equalclassS2} and~\reqref{etahalf}, and in $\pi_{1}(I_{n_{0}}(x_{0},z_{j}))$, we see that:
\begin{align}
\bigl[\bigl(K_{j}\bigl(\pi(\widetilde{\omega}_{\widetilde{x}_{j},t}(\textstyle\frac{1}{2}))\bigr), z_{j}, K_{j}\!\circ\! \pi\!\circ\! h_{j,t}, c_{z_{j}}\bigr)_{t\in I}\bigr] &\!=\! \textstyle\bigl[ \bigl( \pi(\eta_{t,\frac{1}{2}}(0)), z_{j}, \pi\circ\eta_{t,\frac{1}{2}},c_{z_{j}}  \bigr)_{t\in I} \bigr]\notag\\
&\!=\! \textstyle\bigl[ \bigl( \pi(\eta_{t,0}(0)), z_{j}, \pi\circ\eta_{t,0},c_{z_{j}}  \bigr)_{t\in I} \bigr]\notag\\
&\!=\! \bigl[ \bigl( \pi\!\left(\widetilde{\omega}_{\widetilde{x}_{0},t}\left(\textstyle\frac{1}{2}\right)\right), z_{j}, \pi\circ \Omega_{\widetilde{x}_{0},t,\frac{1}{2}}, c_{z_{j}}\bigr)_{t\in I}\bigr]\label{eq:kjomega}
\end{align}
using also \req{equalclass}. Consequently,
\begin{align}
(\widetilde{K}_{j})_{\#}\left(\textstyle(\widehat{\nu}_{j}\left\lvert_{I_{n}}\right.\!)_{\#}(\widehat{\tau}_{n})\right)&= \bigl[ \bigl( \widetilde{\omega}_{\widetilde{x}_{0},t}\left(\textstyle\frac{1}{2}\right), \widetilde{z}_{j}, -\widetilde{z}_{j}, \Omega_{\widetilde{x}_{0},t,\frac{1}{2}}, c_{\widetilde{z}_{j}}, c_{-\widetilde{z}_{j}}\bigr)_{t\in I}\bigr] &&\text{if $M=\St$, and}\label{eq:kjzjS2}\\
(K_{j})_{\#}\left(\textstyle(\widehat{\nu}_{j}\left\lvert_{I_{n}}\right.\!)_{\#}(\widehat{\tau}_{n})\right)&= \bigl[ \bigl( \pi\left(\widetilde{\omega}_{\widetilde{x}_{0},t}\left(\textstyle\frac{1}{2}\right)\right), z_{j}, \pi\circ \Omega_{\widetilde{x}_{0},t,\frac{1}{2}}, c_{z_{j}}\bigr)_{t\in I}\bigr] && \text{if $M=\rp$}\label{eq:kjzj}
\end{align}
by equations~\reqref{kjtaunS2},~\reqref{kjtaun},~\reqref{kjomegaS2} and~\reqref{kjomega} in $\pi_{1}(I_{n_{0}}(\widetilde{x}_{0},\widetilde{z}_{j}))$ and $\pi_{1}(I_{n_{0}}(x_{0},z_{j}))$ respectively. 

Let $H_{u}=\setl{(x,y,z)\in \St}{z\geq 0}$ denote the upper hemisphere of $\St$, let $\dt$ denote the closed unit $2$-disc in $\R^{2}$ whose centre is the origin, and let $\map{P}{H_{u}}[\dt]$ be the homeomorphism defined by $P(x,y,z)=(x,y)$ whose inverse is given by $P^{-1}(x,y)=(x,y,\sqrt{1-x^{2}-y^{2}})$ for all $(x,y)\in \dt$. For $w\in [0,1)$, let $\map{\phi_{w}}{\dt}$ be the homeomorphism defined by:
\begin{equation*}
\phi_{w}(x,y)= \begin{cases}
\Bigl( \frac{w}{w-1}\bigl(1-\bigl\lvert x\bigr\rvert-\bigl\lvert y\bigr\rvert\bigr)+x,y\Bigr) & \text{if $\bigl\lvert x\bigr\rvert+\bigl\lvert y\bigr\rvert\leq 1$}\\
(x,y) & \text{otherwise.}
\end{cases}
\end{equation*}
Note that $\phi_{0}=\id_{\dt}$ and $\phi_{w}(w,0)=(0,0)$. 
If $u\in [-1,1]$, the effect of $\phi_{w}$ is to map linearly the segment joining $(0,1)$ (resp.\ $(0,-1)$) to $(u,0)$ to the segment joining $(0,1)$ (resp.\ $(0,-1)$) to $\bigl(\frac{w}{w-1}(1-\lvert u\rvert)+u,0\bigr)$ as in Figure~\ref{fig:defphiw}. 
\begin{figure}[t]
\hspace*{\fill}
\begin{tikzpicture}[scale=0.6, >=stealth]
\draw[thick] (0,0) circle(5);
\draw[thick] (-5,0) -- (5,0);
\draw[line width=1mm] (0,5) -- (2,0) -- (0,-5);
\foreach \x in {-5,-4,...,1,3,4,5}{\draw[thick] (0,5) -- (\x,0) -- (0,-5);}
\draw[fill] (2,0) circle(0.12);
\draw[fill] (4,0) circle(0.1);
\node at (2.3,-0.3){$\scriptstyle w$};
\node at (4.1,-0.3){$\scriptstyle u$};
\draw[very thick, ->] (5.7,0) -- (7.3,0);
\node at (6.5,0.6){$\varphi_{w}$};
\begin{scope}[shift={(13,0)}]
\draw[thick] (0,0) circle(5);
\draw[thick] (-5,0) -- (5,0);
\draw[line width=1mm] (0,5) -- (0,-5);
\foreach \x in {0,1,...,6}{\draw[thick] (0,5) -- (-5+5*\x/7,0) -- (0,-5);}
\foreach \x in {1,2,3}{\draw[thick] (0,5) -- (5*\x/3,0) -- (0,-5);}
\draw[fill] (0,0) circle(0.12);
\draw[fill] (10/3,0) circle(0.1);
\node at (0.75,-0.4){$\scriptstyle \varphi(w)$};
\node at (3.75,-0.4){$\scriptstyle\varphi(u)$};
\end{scope}
\end{tikzpicture}
\hspace*{\fill}
\caption{The homeomorphism $\protect\map{\phi_{w}}{\dt}$.}
\label{fig:defphiw}
\end{figure}
We define $\map{\Phi_{w}}{\St}$ by:
\begin{equation*}
\Phi_{w}(x,y,z)= \begin{cases}
P^{-1}\circ \phi_{w}\circ P(x,y,z) & \text{if $(x,y,z)\in H_{u}$}\\
-\Phi_{w}(-(x,y,z)) & \text{if $(x,y,z)\in \St \setminus H_{u}$.}
\end{cases}
\end{equation*}
Observe that $\Phi_{0}=\id_{\St}$, $\Phi_{w}$ is a homeomorphism that fixes $(\widetilde{\omega}_{\widetilde{x}_{0},t}\left(\textstyle\frac{1}{2}\right))_{t\in I}$ pointwise, and that satisfies $\Phi_{w}(w,0,\sqrt{1-w^{2}})=\widetilde{z}_{0}$. Now $\widetilde{z}_{j}=\bigl(\sin\bigl(\frac{\pi j}{2(j+1)}\bigr), 0, \cos\bigl(\frac{\pi j}{2(j+1)})\bigr)\bigr)$ in Cartesian coordinates, and
so $p_{1}(\widetilde{z}_{j})=\sin\bigl(\frac{\pi}{2}(\frac{j}{j+1})\bigr)$ belongs to $[0,1)$. We define the homeomorphism $\map{\widetilde{f}_{j}}{\St}$ by $\widetilde{f}_{j}=\Phi_{p_{1}(\widetilde{z}_{j})}$ (see Figure~\ref{fig:mapfj}).
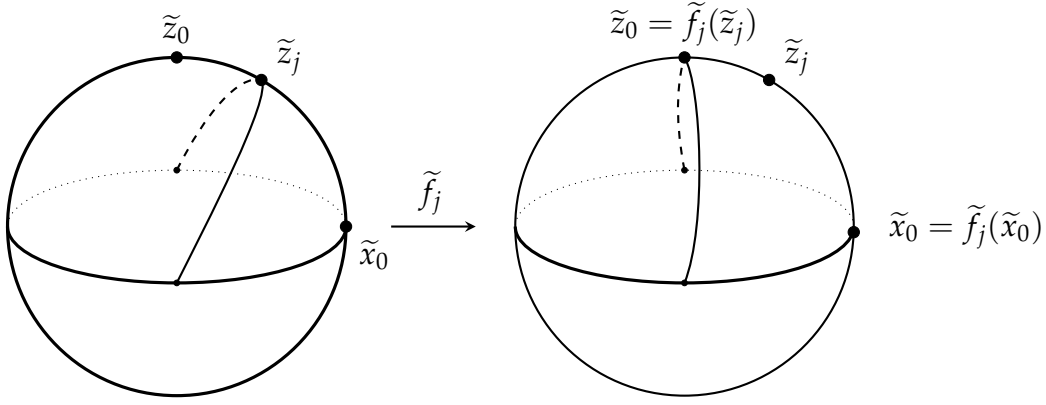
\begin{figure}[t]
\hspace*{\fill}
\begin{tikzpicture}[scale=0.75, >=stealth]
%
%
%
%

\draw[thick] (18,0) circle(3);
\draw[dotted] (21,0) arc (0:180:3 and 1);
\draw[very thick] (15,0) arc (180:360:3 and 1);

\draw[fill] (21,-0.1) circle(0.1);
\draw[fill] (18,3) circle(0.1);
\node at (23,0){$\widetilde{x}_{0}=\widetilde{f}_{j}(\widetilde{x}_{0})$};
\node at (18,3.6){$\widetilde{z}_{0}=\widetilde{f}_{j}(\widetilde{z}_{j})$};
\draw[fill] (19.5,2.6) circle(0.1);
\node at (20,3){$\widetilde{z}_{j}$};

\draw[fill] (18,-1) circle(0.05);
\draw[fill] (18,1) circle(0.05);
\draw[fill] (9,-1) circle(0.05);
\draw[fill] (9,1) circle(0.05);

\draw[thick] (9,-1) .. controls (9.5,0) and (10.7,2.3) .. (10.5,2.6); 
\draw[thick,dashed] (9,1) .. controls (9.5,2) and (10.1,2.7) .. (10.5,2.6);

\draw[thick] (18,-1) .. controls (18.4,0) and (18.3,2.5) .. (18,3);
\draw[thick,dashed] (18,1) .. controls (17.8,2) and (17.9,2.7) .. (18,3);

\draw[thick, ->] (12.8,0) -- (14.3,0);
\node at (13.5,0.6){$\widetilde{f}_{j}$};

\draw[very thick] (9,0) circle(3); 
\draw[dotted] (12,0) arc (0:180:3 and 1); 
\draw[very thick] (6,0) arc (180:360:3 and 1); 


%
%

\draw[fill] (12,0) circle(0.1);
\draw[fill] (9,3) circle(0.1);
\draw[fill] (10.5,2.6) circle(0.1);
\node at (12.5,-0.5){$\widetilde{x}_{0}$};
\node at (9,3.5){$\widetilde{z}_{0}$};
\node at (11,3){$\widetilde{z}_{j}$};

%
%
%
%
%
%
\end{tikzpicture}
\hspace*{\fill}
\caption{The homeomorphism $\widetilde{f}_{j}$.}
\label{fig:mapfj}
\end{figure}
Then $\widetilde{f}_{j}$ satisfies $\widetilde{f}_{j}(\widetilde{z}_{j})= P^{-1}\circ\phi_{p_{1}(\widetilde{z}_{j})}(p_{1}(\widetilde{z}_{j}),0)= \widetilde{z}_{0}$. By definition, $\widetilde{f}_{j}$ is $\Z_{2}$-equivariant with respect to $\tau$, so induces a homeomorphism $\map{f_{j}}{\rp}$ that fixes $\pi\bigl(\widetilde{\omega}_{\widetilde{x}_{0},t}\bigl(\textstyle\frac{1}{2}\bigr)\bigr)$ pointwise for each $t\in I$, and that satisfies $f_{j}(z_{j})=z_{0}$. Let $\map{\gamma_{t}}{I}[\St]$ be the arc defined by $\gamma_{t}= \widetilde{f}_{j}\circ \Omega_{\widetilde{x}_{0},t,\frac{1}{2}}$.
By condition~(\ref{it:xb}) and \req{defOmega}, $p_{3}(\Omega_{\widetilde{x}_{0},t,\frac{1}{2}}(s))\geq 0$ for all $s\in I$. Since $\Phi_{w}$ leaves $H_{u}$ invariant, we see that $p_{3}(\gamma_{t}(s))\geq 0$ for all $s,t\in I$. Further, $\gamma_{t}(0)=\widetilde{\omega}_{\widetilde{x}_{0},t}\left(\textstyle\frac{1}{2}\right)$, and $\gamma_{t}(1)=\gamma_{0}(s)=\gamma_{1}(s)=\widetilde{x}_{0}$ for all $t,s\in I$, so \relem{deformwS2} applies to $\gamma_{t}$, where $\widetilde{w}$ is taken to be $\widetilde{z}_{0}$.  
The homeomorphism $\widetilde{f}_{j}$ (resp.\ $f_{j}$) induces an isomorphism $\map{(\widetilde{f}_{j})_{\#}}{\pi_{1}(I_{n_{0}}(\widetilde{x}_{0},\widetilde{z}_{j}))\!}[\!\pi_{1}(I_{n_{0}}(\widetilde{x}_{0},\widetilde{z}_{0}))]$ (resp.\ $\map{(f_{j})_{\#}}{\pi_{1}(I_{n_{0}}(x_{0},z_{j}))\!}[\!\pi_{1}(I_{n_{0}}(x_{0},z_{0}))]$) for which:
\begin{align}
(\widetilde{f}_j)_{\#} \circ (\widetilde{K}_{j})_{\#}\left(\textstyle(\widehat{\nu}_{j}\left\lvert_{I_{n}}\right.\!)_{\#}(\widehat{\tau}_{n})\right) &= (\widetilde{f}_j)_{\#}\bigl( \bigl[ \bigl( \widetilde{\omega}_{\widetilde{x}_{0},t}\left(\textstyle\frac{1}{2}\right), \widetilde{z}_{j}, -\widetilde{z}_{j}, \Omega_{\widetilde{x}_{0},t,\frac{1}{2}}, c_{\widetilde{z}_{j}}, c_{-\widetilde{z}_{j}}\bigr)_{t\in I}\bigr]\bigr)\notag\\
&=\bigl[ \bigl( \widetilde{\omega}_{\widetilde{x}_{0},t}\left(\textstyle\frac{1}{2}\right), \widetilde{z}_{0}, -\widetilde{z}_{0}, \widetilde{f}_{j}\circ \Omega_{\widetilde{x}_{0},t,\frac{1}{2}}, c_{\widetilde{z}_{0}}, c_{-\widetilde{z}_{0}}\bigr)_{t\in I}\bigr]\notag\\
&=\bigl[ \bigl( \widetilde{\omega}_{\widetilde{x}_{0},t}\left(\textstyle\frac{1}{2}\right), \widetilde{z}_{0}, -\widetilde{z}_{0}, \Omega_{\widetilde{x}_{0},t,\frac{1}{2}}, c_{\widetilde{z}_{0}}, c_{-\widetilde{z}_{0}}\bigr)_{t\in I}\bigr]= \br{\widetilde{x}_{0}, \widetilde{z}_{0}, -\widetilde{z}_{0}},\label{eq:fkjbrS2}
\end{align}
using equations~\reqref{omegah} and~\reqref{kjzjS2} and \relem{deformwS2}(\ref{it:deformwS2a}) if $M=\St$, and:
\begin{align}
(f_{j})_{\#} \circ (K_{j})_{\#}\left(\textstyle(\widehat{\nu}_{j}\left\lvert_{I_{n}}\right.\!)_{\#}(\widehat{\tau}_{n})\right) &= (f_{j})_{\#}\bigl( \bigl[ \bigl( \pi\left(\widetilde{\omega}_{\widetilde{x}_{0},t}\left(\textstyle\frac{1}{2}\right)\right), z_{j}, \pi\circ \Omega_{\widetilde{x}_{0},t,\frac{1}{2}}, c_{z_{j}}\bigr)_{t\in I}\bigr]\bigr)\notag\\
&=\bigl[ \bigl( \pi\left(\widetilde{\omega}_{\widetilde{x}_{0},t}\left(\textstyle\frac{1}{2}\right)\right), z_{0}, f_{j}\circ \pi\circ \Omega_{\widetilde{x}_{0},t,\frac{1}{2}}, c_{z_{0}}\bigr)_{t\in I}\bigr]\notag\\
&=\bigl[ \bigl( \pi\left(\widetilde{\omega}_{\widetilde{x}_{0},t}\left(\textstyle\frac{1}{2}\right)\right), z_{0}, \pi\circ \widetilde{f}_{j}\circ \Omega_{\widetilde{x}_{0},t,\frac{1}{2}}, c_{z_{0}}\bigr)_{t\in I}\bigr]\notag\\
&=\bigl[ \bigl( \pi\left(\widetilde{\omega}_{\widetilde{x}_{0},t}\left(\textstyle\frac{1}{2}\right)\right), z_{0}, \pi\circ \Omega_{\widetilde{x}_{0},t,\frac{1}{2}}, c_{z_{0}}\bigr)_{t\in I}\bigr] = \br{x_{0},{z}_{0}},\label{eq:fkjbr}
\end{align}
using equations~\reqref{omegah} and~\reqref{kjzj} and \relem{deformwS2}(\ref{it:deformwS2b}) if $M=\rp$. 

We now compute the image of the terms on the right-hand side of \req{sumpartialS2} (resp.\ \req{sumpartial}) by the isomorphism $\map{(\widetilde{f}_j)_{\#} \circ (\widetilde{K}_{j})_{\#}}{I_{n_{0}}}[\pi_{1}(I_{n_{0}}(\widetilde{x}_0, \widetilde{z}_0))]$ (resp.\ by the isomorphism $\map{(f_{j})_{\#} \circ (K_{j})_{\#}}{I_{n_{0}}}[\pi_{1}(I_{n_{0}}(x_0, z_0))]$). Applying \relem{boundaryS2}, we have:
\begin{align}
(\widetilde{f}_j)_{\#} \circ (\widetilde{K}_{j})_{\#}(\partial_{n_{0}}(\widetilde{\lambda}_{\widetilde{z}_{0}})) &= (\widetilde{f}_j)_{\#} \circ (\widetilde{K}_{j})_{\#}(\widetilde{\delta}_{\widetilde{z}_{0}})\notag\\
&= (\widetilde{f}_j)_{\#} \circ (\widetilde{K}_{j})_{\#} \bigl(\bigl[ \bigl(\widetilde{x}_j, \widetilde{z}_0, -\widetilde{z}_0, c_{\widetilde{x}_j}, \widetilde{\omega}_{\widetilde{z}_{0}, t}, c_{-\widetilde{z}_0}\bigr)_{t\in I}\bigr]\bigr)\notag\\
&= \bigl[ \bigl(\widetilde{x}_0, \widetilde{z}_0, -\widetilde{z}_0, c_{\widetilde{x}_0}, \widetilde{f}_{j}\circ \widetilde{K}_{j}\circ \widetilde{\omega}_{\widetilde{z}_{0}, t}, c_{-\widetilde{z}_0}\bigr)_{t\in I}\bigr]
\label{eq:fkjz0S2}
\end{align}
if $M=\St$, and:
\begin{align}
(f_{j})_{\#} \circ (K_{j})_{\#}(\partial_{n_{0}}(\lambda_{z_{0}}))&= (f_{j})_{\#} \circ (K_{j})_{\#}(\delta_{z_{0}})\notag\\
&=(f_{j})_{\#} \circ (K_{j})_{\#} \bigl(\bigl[ \bigl( x_{j}, z_{0}, c_{x_{j}}, \pi\circ \widetilde{\omega}_{\widetilde{z}_{0},t} \bigr)_{t\in I}\bigr]\bigr)\notag\\
&= \bigl[ \bigl( x_{0}, z_{0}, c_{x_{0}},  \pi \circ\widetilde{f}_{j}\circ \widetilde{K}_{j} \circ \widetilde{\omega}_{\widetilde{z}_{0},t} \bigr)_{t\in I}\bigr]\label{eq:fkjz0}
\end{align}
if $M=\rp$. Now the two homeomorphisms $\map{\widetilde{f}_{j}\circ \widetilde{K}_{j}\circ J_{j}}{(\St,\widetilde{x}_{0})}$ and $\map{\widetilde{f}_{j}\circ \widetilde{K}_{j}}{(\St,\widetilde{z}_{0})}$ preserve orientation, so:
\begin{equation}\label{eq:fjKjomega}
\text{$\bigl[\bigl(\widetilde{f}_{j}\circ \widetilde{K}_{j}\circ J_{j}\circ\widetilde{\omega}_{\widetilde{x}_{0},t}\bigr)_{t\in I}\bigr] =\bigl[\bigl(\widetilde{\omega}_{\widetilde{x}_{0},t}\bigr)_{t\in I}\bigr]$ and $\bigl[\bigl(\widetilde{f}_{j}\circ \widetilde{K}_{j}\circ\widetilde{\omega}_{\widetilde{z}_{0},t}\bigr)_{t\in I}\bigr]=\bigl[\bigl(\widetilde{\omega}_{\widetilde{z}_{0},t}\bigr)_{t\in I}\bigr]$}
\end{equation}
in $\pi_{1}(\Omega(\St),c_{\widetilde{x}_{0}})$ and $\pi_{1}(\Omega(\St),c_{\widetilde{z}_{0}})$ respectively.
It follows from equations~\reqref{fkjz0S2} and~\reqref{fjKjomega}, and from \relem{boundaryS2}(\ref{it:boundaryS2a}) that:
\begin{equation}\label{eq:fkjdeltaz0S2}
(\widetilde{f}_j)_{\#} \circ (\widetilde{K}_{j})_{\#}(\partial_{n_{0}}(\widetilde{\lambda}_{\widetilde{z}_{0}}))=\left[ \bigl(\widetilde{x}_0, \widetilde{z}_0, -\widetilde{z}_0, c_{\widetilde{x}_0}, \widetilde{\omega}_{\widetilde{z}_{0}, t}, c_{-\widetilde{z}_0}\bigr)_{t\in I}\right]=\widetilde{\delta}_{\widetilde{z}_0}
\end{equation}
in $\pi_{1}(I_{n_{0}}(\widetilde{x}_0, \widetilde{z}_0))$. Exchanging the rôles of $\widetilde{z}_0$ and $-\widetilde{z}_0$ in this argument yields:
\begin{equation}\label{eq:fkjdeltaminusz0S2}
(\widetilde{f}_j)_{\#} \circ (\widetilde{K}_{j})_{\#}(\partial_{n_{0}}(\widetilde{\lambda}_{-\widetilde{z}_{0}}))=\widetilde{\delta}_{-\widetilde{z}_0}.
\end{equation}
Applying $\map{\pi_{\#}}{\pi_{1}(\St,\widetilde{z}_{0})}[\pi_{1}(\rp,z_{0})]$  to \reqref{fjKjomega}, we see that $([(\pi \circ \widetilde{f}_j \circ \widetilde{K}_j\circ\widetilde{\omega}_{\widetilde{z}_{0},t})_{t\in I}])=[(\pi\circ\widetilde{\omega}_{\widetilde{z}_{0},t})_{t\in I}]$, and it follows from \reqref{fkjz0} and \relem{boundaryS2}(\ref{it:boundaryS2b}) that:
\begin{equation}\label{eq:fkjdeltaz0}
(f_{j})_{\#} \circ (K_{j})_{\#}(\partial_{n_{0}}(\lambda_{z_{0}}))=\left[ \left( x_{0}, z_{0}, c_{x_{0}},  \pi \circ \widetilde{\omega}_{\widetilde{z}_{0},t} \right)_{t\in I}\right]=\delta_{z_{0}}.
\end{equation}
Recall from \resec{notation} that $\widetilde{\lambda}_{\widetilde{x}_{j}}=[(\widetilde{\omega}_{\widetilde{x}_{j},t})_{t\in I}]$. So in a similar manner, using \relem{boundary}, we obtain: 
\begin{align}
(\widetilde{f}_j)_{\#} \circ (\widetilde{K}_{j})_{\#}(\partial_{n_{0}}(\widetilde{\lambda}_{\widetilde{x}_{j}})) &= (\widetilde{f}_j)_{\#} \circ (\widetilde{K}_{j})_{\#} \bigl(\bigl[ \bigl(\widetilde{x}_j, \widetilde{z}_0, -\widetilde{z}_0, \widetilde{\omega}_{\widetilde{x}_{j}, t}, c_{\widetilde{z}_0}, c_{-\widetilde{z}_0}\bigr)_{t\in I}\bigr]\bigr)\notag\\
&= \bigl[ \bigl(\widetilde{x}_0, \widetilde{z}_0, -\widetilde{z}_0, \widetilde{f}_{j}\circ \widetilde{K}_{j}\circ J_{j}\circ \widetilde{\omega}_{\widetilde{x}_{0}, t}, c_{\widetilde{z}_0}, c_{-\widetilde{z}_0}\bigr)_{t\in I}\bigr]\label{eq:fkjz1S2}
\end{align}
using \relem{boundaryS2}(\ref{it:boundaryS2a}) in $\pi_{1}(I_{n_{0}}(\widetilde{x}_j, \widetilde{z}_0))$ and \req{hxjt}, and
\begin{align}
(f_{j})_{\#} \circ (K_{j})_{\#}(\partial_{n_{0}}(\lambda_{x_{j}}))&= (f_{j})_{\#} \circ (K_{j})_{\#}\bigl(\bigl[ \bigl( x_{j}, z_{0}, \pi\circ \widetilde{\omega}_{\widetilde{x}_{j},t}, c_{z_{0}} \bigr)_{t\in I}\bigr]\bigr)\notag\\
&= \bigl[ \bigl( x_{0}, z_{0}, \pi\circ \widetilde{f}_{j}\circ \widetilde{K}_{j} \circ J_{j} \circ \widetilde{\omega}_{\widetilde{x}_{0},t}, c_{z_{0}} \bigr)_{t\in I}\bigr]\label{eq:fkjz1}
\end{align}
using \relem{boundaryS2}(\ref{it:boundaryS2b}) in $\pi_{1}(I_{n_{0}}(x_j, z_0))$. It follows from equations~\reqref{fjKjomega},~\reqref{fkjz1S2} and~\reqref{fkjz1} and \relem{boundaryS2} that:
\begin{align}
(\widetilde{f}_j)_{\#} \circ (\widetilde{K}_{j})_{\#}(\partial_{n_{0}}(\widetilde{\lambda}_{\widetilde{x}_{j}}))&=\left[ \bigl(\widetilde{x}_0, \widetilde{z}_0, -\widetilde{z}_0, \widetilde{\omega}_{\widetilde{x}_{0}, t}, c_{\widetilde{z}_0}, c_{-\widetilde{z}_0}\bigr)_{t\in I}\right]=\widetilde{\delta}_{\widetilde{x}_{0}}\label{eq:fkjdeltaz1S2}\\
(f_{j})_{\#} \circ (K_{j})_{\#}(\partial_{n_{0}}(\lambda_{x_{j}}))&=\left[ \left( x_{0}, z_{0}, \pi\circ \widetilde{\omega}_{\widetilde{x}_{0},t}, c_{z_{0}} \right)_{t\in I}\right]=\delta_{x_{0}}\label{eq:fkjdeltaz1}
\end{align}
in $\pi_{1}(I_{n_{0}}(\widetilde{x}_0, \widetilde{z}_0))$, and in $\pi_{1}(I_{n_{0}}(x_0, z_0))$ respectively. Taking the image of both sides of~\req{sumpartialS2} (resp.\ \req{sumpartial}) by $(\widetilde{f}_j)_{\#} \circ (\widetilde{K}_{j})_{\#}$ (resp.\ by $(f_{j})_{\#} \circ (K_{j})_{\#}$), we conclude using equations~\reqref{fkjbrS2},~\reqref{fkjdeltaz0S2},~\reqref{fkjdeltaminusz0S2} and~\reqref{fkjdeltaz1S2} (resp.\ equations~\reqref{fkjbr},~\reqref{fkjdeltaz0} and~\reqref{fkjdeltaz1}) that:
\begin{align}
2 \br{\widetilde{x}_{0}, \widetilde{z}_{0}, -\widetilde{z}_{0}}&=m_{j} \widetilde{\delta}_{\widetilde{x}_{0}}+m_{n-2}\widetilde{\delta}_{\widetilde{z}_0} +m_{n-1} \widetilde{\delta}_{-\widetilde{z}_0}\;\text{in $\pi_{1}(I_{n_{0}}(\widetilde{x}_0, \widetilde{z}_0))$}\label{eq:brmjS2}\\
2 \br{x_{0},{z}_{0}}&=m_{j} \delta_{x_{0}}+m_{n-1}\delta_{z_{0}}\;\text{in $\pi_{1}(I_{n_{0}}(x_0, z_0))$.}\label{eq:brmj}
\end{align}
Comparing \req{brmjS2} (resp.\ \req{brmj}) with \reth{brrelnS2}, 
we see that $m_{j}=m_{n-2}=1$ and $m_{n-1}=-1$ (resp.\ $m_{j}=m_{n-1}=1$). The statement of the proposition for $M=\St$ (resp.\ for $M=\rp$) then follows from \req{commdeltanS2}.
\end{proof}

The following result generalises \reco{sesnn0}.

\begin{cor}\label{cor:sesnngen}
Let $n\geq n_{0}$. 
The boundary homomorphism $\map{\partial_{n}}{\pi_{2}(\prod_{1}^{n}\, M)}[\pi_{1}(I_{n})]$ of the exact sequence~\reqref{lespartialn} satisfies: 
\begin{align*}
\partial_{n}(\widetilde{\lambda}_{u}) &= 
\begin{cases}
\widetilde{\delta}_{u} & \text{if $u\in \brak{\widetilde{x}_{0},\widetilde{x}_{1},\ldots, \widetilde{x}_{n-3}, \widetilde{z}_{0}}$}\\
\widetilde{\delta}_{\widetilde{x}_{0}}+\widetilde{\delta}_{\widetilde{x}_{1}}+\cdots +\widetilde{\delta}_{\widetilde{x}_{n-3}}
-\widehat{\tau}_{n}^{2}
& \text{if $u=-\widetilde{z}_{0}$}
\end{cases}
&& \text{and if $M=\St$,}\\
\partial_{n}(\lambda_{u})&=
\begin{cases}
\delta_{u} & \text{if $u\in \brak{x_{0},x_{1},\ldots, x_{n-2}}$}\\
\widehat{\tau}_{n}^{2}-(\delta_{x_{0}}+\delta_{x_{1}}+\cdots +\delta_{x_{n-2}})
& \text{if $u=z_{0}$}
\end{cases}
&& \text{and if $M=\rp$,}
\end{align*}
the notation being that of equations~\reqref{basispi2} and~\reqref{defdeltav}.
\end{cor}

\begin{proof}
If $M=\St$ (resp.\ $M=\rp$), a basis of $\pi_{2}(\prod_{1}^{n}\, M)$ is as given in~\reqref{basispi2},
and the equalities involving $\partial_{n}(\widetilde{\lambda}_{u})$ (resp.\ $\partial_{n}(\lambda_{u})$) then follow from~\reqref{defdeltav} and \reth{tauhatsquareS2}.
\end{proof}

To finish this section, we prove \repr{sesnngen}.

\begin{proof}[Proof of \repr{sesnngen}]
%
%
Let $M=\St$ (resp.\ $M=\rp$), and let $n\geq n_{0}$. 
By~\cite[equation~(2) and~Proposition~1(a)(i)]{GGgold} and exactness of~\reqref{lespartialn},
\begin{equation}\label{eq:isoGamma2}
\ker{(\iota_{n})_{\#}}= \im{(g_{n}\circ j_{n})_{\#}} \cong \pi_{1}(I_{n})/\im{\partial_{n}},
\end{equation}
where $\ker{(\iota_{n})_{\#}}=P_{n}(\St)$ (resp.\ $\ker{(\iota_{n})_{\#}}=\Gamma_{2}(P_{n}(\rp))$), and the given isomorphism is induced by the homomorphism $(g_{n}\circ j_{n})_{\#}$. From \reco{sesnngen}, we have $\im{\partial_{n}}=\bigl\langle\widetilde{\delta}_{\widetilde{x}_{0}},\widetilde{\delta}_{\widetilde{x}_{0}}, \ldots, \widetilde{\delta}_{\widetilde{x}_{n-3}}, \widetilde{\delta}_{\widetilde{z}_{0}}, \widehat{\tau}_{n}^{2}\bigr\rangle$ (resp.\ $\im{\partial_{n}}=\ang{\delta_{x_{0}},\delta_{x_{1}}, \ldots, \delta_{x_{n-2}}, \widehat{\tau}_{n}^{2}}$). Now the homomorphism 
$\map{(\alpha_{n}'\circ \alpha_{\pi} \circ \alpha_{c})_{\#}}{\pi_{1}(I_{c})}[\pi_{1}(I_{n})]$ 
is an isomorphism by \relem{fundhomoequiv} and \reco{alphanprime}, and since $(\alpha_{n}'\circ \alpha_{\pi} \circ \alpha_{c})_{\#}(\widetilde{\lambda}_{\widetilde{x}_{i}})=\widetilde{\delta}_{\widetilde{x}_{i}}$ for all $0\leq i\leq n-3$ and $(\alpha_{n}'\circ \alpha_{\pi} \circ \alpha_{c})_{\#}(\widetilde{\lambda}_{\widetilde{z}_{0}})=\widetilde{\delta}_{\widetilde{z}_{0}}$ (resp.\ $(\alpha_{n}'\circ \alpha_{\pi} \circ \alpha_{c})_{\#}(\widetilde{\lambda}_{\widetilde{x}_{i}})=\delta_{x_{i}}$ for all $0\leq i\leq n-2$) and $(\alpha_{n}'\circ \alpha_{\pi} \circ \alpha_{c})_{\#}([\tau_{n}])=\widehat{\tau}_{n}$
by \relem{fundhomoequiv}(\ref{it:fundhomoequivd}) and \repr{deftaunhat}, we conclude that:
\begin{equation}\label{eq:isopic}
\pi_{1}(I_{n})/\im{\partial_{n}}\cong 
\begin{cases}
\pi_{1}(I_{c})\bigl/\bigl\langle \widetilde{\lambda}_{\widetilde{x}_{0}}, \widetilde{\lambda}_{\widetilde{x}_{1}},\ldots, \widetilde{\lambda}_{\widetilde{x}_{n-3}}, \widetilde{\lambda}_{\widetilde{z}_{0}}, [\tau_{n}^{2}] \bigr\rangle\bigr. & \text{if $M=\St$}\\
\pi_{1}(I_{c})\bigl/\bigl\langle \widetilde{\lambda}_{\widetilde{x}_{0}}, \widetilde{\lambda}_{\widetilde{x}_{1}},\ldots, \widetilde{\lambda}_{\widetilde{x}_{n-2}}, [\tau_{n}^{2}] \bigr\rangle\bigr. & \text{if $M=\rp$,}
\end{cases}
\end{equation}
the isomorphism being induced by $(\alpha_{n}'\circ \alpha_{\pi} \circ \alpha_{c})_{\#}$. 
It follows from~\reqref{defIc} and~\reqref{isopic} that:
\begin{equation}\label{eq:isoGn1}
\pi_{1}(I_{n})/\im{\partial_{n}}\cong
\begin{cases}
P_{n-1}/\langle [\tau_{n}^{2}] \rangle & \text{if $M=\St$}\\
G_{n-1}/\langle [\tau_{n}^{2}] \rangle & \text{if $M=\rp$.}
\end{cases}
\end{equation}
If $M=\St$ then from \req{deftaun} and~Figure~\ref{fig:fulltwist}, $[\tau_{n}]$ may be interpreted as the full-twist braid $\ft[n-1]$ of $P_{n-1}$, and so $\pi_{1}(I_{n})/\im{\partial_{n}}\cong P_{n-1}/\ang{\Delta_{n-1}^{4}}\cong \F[n-2]\rtimes (\F[n-3] \rtimes (\cdots\rtimes(\F[3]\rtimes \F[2])\cdots)) \times \Z_{2}$ using~\cite[Proposition~8]{GG2} and the Artin combing operation. We thus obtain the isomorphisms of~\reqref{isoPnS2}. Now suppose that $M=\rp$. 
To interpret $[\tau_{n}]\in \pi_{1}(I_{c})$, which is given by \req{deftaun} in the case $M=\rp$, as an element of $G_{n-1}$, we modify slightly Figure~\ref{fig:fulltwist} by replacing $n-3$ by $n-2$ and by removing the markings on $\widetilde{z}_{0}$ and $-\widetilde{z}_{0}$ in that figure. In order to obtain a geometric representation of $[\tau_{n}]$, we also suppose that each loop $(\widetilde{\omega}_{\widetilde{x}_{i},t}(\frac{1}{2}))_{t\in I}$ is  equipped with the associated constant path $c_{\widetilde{x}_{i}}$ for $i=0,1,\ldots,n-2$ as in~\reqref{deftaun}. To obtain Figure~\ref{fig:gensorb}, but taking $n-1$ in place of $n$, we first rotate the geometric representative of $[\tau_{n}]$ by $\pi$ about the vertical axis, then we remove the points $\widetilde{z}_{0}$ and $-\widetilde{z}_{0}$, and finally we flatten down the resulting open cylinder so that $\widetilde{z}_{0}$ is the centre, as in Figure~\ref{fig:gensorb}. We thus identify the element $v_{i}$ of Figure~\ref{fig:gensorb} with $\widetilde{x}_{i-1}$ for all $i=1,\ldots, n-1$. Via this construction, $[\tau_{n}]$ may be identified with the element $\Theta_{n-1}^{-1}=\rho_{n-1,0}^{-1}\cdots \rho_{1,0}^{-1}$  of $G_{n-1}$, and up to this identification, we have a surjective homomorphism $\map{\upsilon_{n}}{G_{n-1}}[K_{n}]$ induced by the map $\map{g_{n}\circ j_{n}\circ\alpha_{n}'\circ \alpha_{\pi} \circ \alpha_{c}}{I_{c}}[F_{n}(\rp)]$, and an isomorphism $\map{\overline{\upsilon_{n}}}{G_{n-1}/\langle \Theta_{n-1}^{2} \rangle}[K_{n}]$ that is induced by $\upsilon_{n}$. This yields the first isomorphism of~\reqref{isoKnGn}.

Let $\map{p}{P_{n+1}(\rp)}[P_{n}(\rp)]$ be the Fadell-Neuwirth projection given geometrically by removing the penultimate string (this corresponds to forgetting the basepoint $v_{n}$, so that $z_{0}$ remains as the final basepoint). Using~\cite[Proposition~8]{GGgold}, we obtain the following commutative diagram of short exact sequences:
\begin{equation}\label{eq:commdiagKn}
\begin{tikzcd}[cramped]
& 1\arrow[d]& 1\arrow[d]& 1\arrow[d]&\\
1\arrow[r] & \operatorname{\text{Ker}}(p\bigl\lvert_{K_{n+1}}\bigr.) \arrow[d]\arrow[r] & K_{n+1} \arrow[d]\arrow[r, "p\bigl\lvert_{K_{n+1}}\bigr."] & K_{n} \arrow[d]\arrow[r] & 1\\
1\arrow[r] & \ker{p} \arrow[d, "(\iota_{n+1})_{\#}\bigl\lvert_{\ker{p}}\bigr."]\arrow[r] & P_{n+1}(\rp) \arrow[d, "(\iota_{n+1})_{\#}"]\arrow[r, "p"] & P_{n}(\rp) \arrow[d, "(\iota_{n})_{\#}"]\arrow[r] & 1\\
1\arrow[r] & \Z_{2} \arrow[d]\arrow[r] & \Z_{2}^{n+1} \arrow[d]\arrow[r, "\widehat{p}"] & \Z_{2}^{n} \arrow[d]\arrow[r] & 1,\\
& 1& 1& 1&
\end{tikzcd}
\end{equation}
where $\map{\widehat{p}}{\Z_{2}^{n+1}}[\Z_{2}^{n}]$ is the homomorphism that forgets the $n$\up{th} coordinate, from which it follows that $\operatorname{\text{Ker}}(p\bigl\lvert_{K_{n+1}}\bigr.)=\operatorname{\text{Ker}}((\iota_{n+1})_{\#}\bigl\lvert_{\ker{p}}\bigr.)$. On the other hand, $\ker{p}$ is a free group of rank $n$ for which a basis is given by $\brak{A_{1,n},\ldots, A_{n-1,n}, \rho_{n}}$, so by the first column of~\reqref{commdiagKn} and~\cite[Proposition~8]{GGgold}, $\operatorname{\text{Ker}}(p\bigl\lvert_{K_{n+1}}\bigr.)$ is a free group of rank $2n-1$ for which a basis $\mathcal{B}$ is given by $\setl{A_{i,n},\rho_{n}A_{i,n}\rho_{n}^{-1}, \rho_{n}^{2}}{1\leq i\leq n-1}$.

We now study the image under $\upsilon_{n+1}$ of the generating set $(\rho_{j,i})_{1\leq j\leq n,\, 0\leq i\leq 2j-2}$ of $G_{n}$ as given in \repr{presgn}. We take the left-hand half of Figure~\ref{fig:gensorb} to be a fundamental domain for the projection $\map{\widehat{\pi}}{C}[\rp\setminus\brak{z_{0}}]$ defined in \resec{generalities}, and when we compose by $\alpha_{n+1}'$, we push $z_{0}$ into the interior of the disc to obtain a model of $\rp$ similar to that of~\cite[Figure~1]{GGgold}, where the points $x_{1},\ldots,x_{n+1}$ of that figure should be identified with $\pi(v_{1}),\ldots,\pi(v_{n}), z_{0}$ respectively. Making use of the presentation of the pure braid groups of $\rp$ given in~\cite[Theorem~4]{GG3} and also described in~\cite[Theorem~7]{GGgold}, as well as the element $C_{i,j}$ defined by~\cite[equation~(8)]{GGgold} and~\cite[relations~(III) and~(V), Proposition~11]{GGgold}, we see that: 
\begin{equation}\label{eq:upsilonn}
\upsilon_{n+1}(\rho_{j,i})=\begin{cases}
A_{i,j} & \text{if $1\leq i<j$}\\
\rho_{j} A_{1,j} \cdots A_{j-1,j} \rho_{j} =A_{j,j+1}\cdots A_{j,n+1} & \text{if $i=0$}\\
\rho_{j}^{-1} C_{i-j+1,j}^{-1} \rho_{j} =\rho_{j} A_{i-j+1,j} \rho_{j}^{-1} & \text{if $j\leq i\leq 2j-2$.}
\end{cases}
\end{equation}
Recall from \resec{presorb} that the kernel of the homomorphism $\map{(q_{n})_{\#}}{G_{n}}[G_{n-1}]$ is a free group of rank $2n-1$ for which $(\rho_{n,i})_{0\leq i\leq 2n-2}$ is a basis.   
Using~\reqref{upsilonn}, we thus obtain:
\begin{equation}\label{eq:upsqn}
\upsilon_{n+1}\bigl(\ker{(q_{n})_{\#}}\bigr)=\setangl{A_{i,n},\rho_{n} A_{i,n}\rho_{n}^{-1},A_{n-1,n}}{1\leq i\leq n-1},
\end{equation} 
and since
\begin{equation*}
\rho_{n}^{-2}= A_{n,n+1}^{-1}\rho_{n}\ldotp \rho_{n}^{-1} A_{n,n+1}\rho_{n}^{-1} \ldotp \rho_{n}^{-1}= A_{n,n+1}^{-1}\prod_{i=1}^{n-1} \rho_{n} A_{i,n}\rho_{n}^{-1}
\end{equation*}
by the surface relation~(V) of~\cite[Proposition~11]{GGgold}, it follows that $\rho_{n}^{2}\in \upsilon_{n+1}(\ker{(q_{n})_{\#}})$ by~\reqref{upsqn}. We conclude from~\reqref{upsqn} and the form of the basis $\mathcal{B}$ of $\operatorname{\text{Ker}}(p\bigl\lvert_{K_{n+1}}\bigr.)$ that the restriction $\upsilon_{n+1}\bigl\lvert_{\operatorname{\text{Ker}}((q_{n})_{\#})}\bigr.$ is a surjective homomorphism between two free groups of the same rank, so is an isomorphism. We thus obtain the following commutative diagram of short exact sequences:
\begin{equation}\label{eq:commdiagGnbig}
\begin{tikzcd}[cramped]
& & 1\arrow[d]& 1\arrow[d]&\\
& & \langle \Theta_{n}^{2} \rangle \arrow[d] 
\arrow{r}[yshift=0.6ex]{(q_{n})_{\#}\bigl\lvert_{\langle \Theta_{n}^{2} \rangle}\bigr.}[swap]{\cong}
& \langle \Theta_{n-1}^{2} \rangle\arrow[d]&\\
1 \arrow[r] & \ker{(q_{n})_{\#}} \arrow[r]  \arrow[d, "\cong", "\upsilon_{n+1}\bigl\lvert_{\operatorname{\text{Ker}}((q_{n})_{\#})}\bigr."'] & G_{n}\arrow[r, "(q_{n})_{\#}"] \arrow[d, "\upsilon_{n+1}"'] & G_{n-1} \arrow[r] \arrow[d, "\upsilon_{n}"'] & 1\\
1 \arrow[r] & \operatorname{\text{Ker}}\bigl( p\bigl\lvert_{K_{n+1}}\bigr. \bigr) \arrow[r]  & K_{n+1} \arrow[r, "p\bigl\lvert_{K_{n+1}}\bigr."] \arrow[d] & K_{n} \arrow[r] \arrow[d] & 1,\\
& & 1& 1&
\end{tikzcd}
\end{equation}
where the commutativity of the bottom right-hand square is a consequence of~\reqref{upsilonn} and the definitions of $(q_{n})_{\#}$ and $p$. Furthermore, since $\langle \Theta_{n}^{2} \rangle \cap \ker{(q_{n})_{\#}}=\brak{e}$, $(q_{n})_{\#}$ induces a homomorphism $\map{Q_{n}}{G_{n}/\langle \Theta_{n}^{2} \rangle}[G_{n-1}/\langle \Theta_{n-1}^{2} \rangle]$, which gives rise to the following commutative diagram of short exact sequences:
\begin{equation}\label{eq:commdiagGnquot}
\begin{tikzcd}[cramped]
1 \arrow[r] & \ker{(q_{n})_{\#}} \arrow[r]  \arrow[d, "\cong", "\upsilon_{n+1}\bigl\lvert_{\operatorname{\text{Ker}}((q_{n})_{\#})}\bigr."'] & G_{n}/\langle \Theta_{n}^{2} \rangle \arrow[r, "Q_{n}"] \arrow[d, "\overline{\upsilon_{n+1}}"', "\cong"] & G_{n-1}/\langle \Theta_{n-1}^{2} \rangle \arrow[r] \arrow[d, "\overline{\upsilon_{n}}"', "\cong"] & 1\\
1 \arrow[r] & \operatorname{\text{Ker}}\bigl( p\bigl\lvert_{K_{n+1}}\bigr. \bigr) \arrow[r]  & K_{n+1} \arrow[r, "p\bigl\lvert_{K_{n+1}}\bigr.", yshift=1ex] & K_{n} \arrow[r] & 1.
\end{tikzcd}
\end{equation}
It follows from~\reqref{commdiagGnquot} that $p\bigl\lvert_{K_{n+1}}\bigr.$ admits a section if and only if $Q_{n}$ does, and from~\reqref{commdiagGnbig} that $Q_{n}$ admits a section if and only if $(q_{n})_{\#}$ admits a section that sends $\ang{\Theta_{n-1}^{2}}$ to $\ang{\Theta_{n}^{2}}$. Such a section for $(q_{n})_{\#}$ is provided by the homomorphism $(s_{n}')_{\#}$ of \rerem{secondsec}, and we deduce the second isomorphism of~\reqref{isoKnGn}.
\end{proof}

\begin{rem} 
In~\cite[Proposition~(1)(a)(ii) and Theorem~3]{GGgold}, we showed that $K_{n}\cong L_{n}\times \Z_{2}$, where $L_{n}\cong \F[2n-3]\rtimes (\F[2n-5] \rtimes (\cdots\rtimes(\F[5]\rtimes \F[3])\cdots))$. Due to the fact that the basepoints involved are not the same, we cannot compare $L_{n}$ directly with the corresponding factor of~\reqref{isoKnGn} arising from the isomorphism $\overline{\upsilon_{n}}$. Using~\reqref{upsilonn} and the section of \rerem{secondsec}, it should however be possible to make explicit the factors $\F[2i+1]$, where $i=1,\ldots, n-2$, as subgroups of $K_{n}$. 
Independently of \repr{sesnngen}, using the results of~\cite{GGgold}, one may show that the homomorphism $\map{p\bigl\lvert_{K_{n+1}}\bigr.}{K_{n+1}}[K_{n}]$ splits by making use of the fact that the restriction $\map{p\bigl\lvert_{L_{n+1}}\bigr.}{L_{n+1}}[L_{n}]$ splits.
\end{rem}

\appendix
\setcounter{secnumdepth}{0}
\section{Appendix}


As we mentioned in \rerems{homfibre}, the map $\map{\alpha_{n}'}{I_{n}'}[I_{n}]$ defined in \resec{generalities} is a homotopy equivalence. This follows from \repr{homofibre} given below relating the homotopy fibres of fibre spaces and certain subspaces, and is a consequence of \relem{crabb1}. This lemma seems to be folklore, and we were not able to find a proof in the literature, although the statement is given as an exercise in~\cite[Section~3.2, Exercises(2), page~99]{N}. We are grateful to Michael~Crabb for proposing proofs of \relem{crabb1} and \repr{homofibre}. In this section, we will denote the homotopy fibre of a map $\map{f}{X}[Y]$ by $I(f)$ rather than $I_{f}$.

We first fix some notation and recall the homotopy fibres that we will analyse. Let $\map{f}{X}[Y]$ and $\map{g}{Y}[Z]$ be pointed maps. By~\reqref{defhomfibre}, we have: 
\begin{align*}
I(f) &=\setl{(x,\rho)\in X\times Y^{I}}{\rho(0)=f(x),\, \rho(1)=\ast_{Y}},\\
I(g)&=\setl{(y,\gamma)\in Y\times Z^{I}}{\gamma(0)=g(y),\, \gamma(1)=\ast_{Z}}\;\text{and}\\
I(g\circ f) &=\setl{(x,\gamma)\in X\times Z^{I}}{\gamma(0)=g(f(x)),\, \gamma(1)=\ast_{Z}}.
\end{align*}
Let $\map{f_{\ast}}{I(g\circ f)}[I(g)]$ be the map defined by $f_{\ast}(x,\gamma)=(f(x),\gamma)$. Then:
\begin{equation*}
I(f_{\ast}) =\setl{(x,\gamma,\Phi)\in X\times Z^{I}\times I(g)^{I}}{\Phi(0)=f_{\ast}(x,\gamma),\, \Phi(1)=\ast_{I(g)}}.
\end{equation*}
If $\map{\Phi}{I}[I(g)]$ belongs to $I(g)^{I}$ then there exists $(\rho,\Gamma)\in (Y\times Z^{I})^{I}$ such that
for all $s\in I$, $\Phi(s)=(\rho(s), \Gamma(s,\cdot))\in I(g)$, $\Gamma(s,0)=g(\rho(s))$ and $\Gamma(s,1)=\ast_{Z}$ for all $s\in I$.
Hence:
\begin{align*}
I(f_{\ast}) =&\Bigl\{(x,\gamma, \rho,\Gamma)\in X\times Z^{I}\times (Y\times Z^{I})^{I} \,\Bigl\vert\,\gamma(0)=g(f(x)),\, \gamma(1)=\ast_{Z},\\
& \text{$\Gamma(s,0)=g(\rho(s))$ and $\Gamma(s,1)=\ast_{Z}$ for all $s\in I$},\\
& (f(x),\gamma)=(\rho(0),\Gamma(0,\cdot)), (\rho(1), \Gamma(1,\cdot))= (\ast_{Y},c_{\ast_{Z}}) \Bigr\}.
\end{align*}
Therefore $\gamma(t)=\Gamma(0,t)$ for all $t\in I$, and thus $\gamma(1)=\Gamma(0,1)=\ast_{Z}$ and $\gamma(0)=\Gamma(0,0)=g(\rho(0))=g(f(x))$, which implies that two of the defining conditions of $I(f_{\ast})$ may be removed, and $\gamma$ is redundant. Hence we may identify $I(f_{\ast})$ with the set:
\begin{align}
& \Bigl\{(x, \rho,\Gamma)\in X\times (Y\times Z^{I})^{I} \,\Bigl\vert\, f(x)=\rho(0),\, \rho(1)=\ast_{Y},\, \text{$\Gamma(1,t)=\ast_{Z}$ for all $t\in I$,}\notag\\
&\text{$\Gamma(s,0)=g(\rho(s))$ and $\Gamma(s,1)=\ast_{Z}$ for all $s\in I$}\Bigr\}.\label{eq:ifstarreduced}
\end{align}


\begin{lem}\label{lem:crabb1}
Let $\map{f}{X}[Y]$ and $\map{g}{Y}[Z]$ be pointed maps, and let $\map{f_{\ast}}{I(g\circ f)}[I(g)]$ be the map induced by $f$ on the level of homotopy fibres. Then $I(f_{\ast})$ is homotopy equivalent to $I(f)$.
\end{lem}

\begin{proof}
With the identification of $I(f_{\ast})$ with the set given in~\reqref{ifstarreduced}, let $\map{\psi}{I(f_{\ast})}[I(f)]$ be defined by $\psi(x, \rho,\Gamma)=(x,\rho)$ and $\map{\phi}{I(f)}[I(f_{\ast})]$ by $\phi(x,\rho)=(x,\rho,\Gamma')$, where $\Gamma'(s,t)=g(\rho(\max(s,t)))$ for all $s,t\in I$. Note that $\phi$ is well defined since $f(x)=\rho(0)$, $\rho(1)=\ast_{Y}$, $\Gamma'(s,0)=g(\rho(s))$ and $\Gamma'(s,1)=\Gamma'(1,t)=g(\rho(1))=\ast_{Z}$ for all $s,t\in I$. Clearly $\psi\circ \phi=\id_{I(f)}$. It thus remains to prove that $\phi\circ \psi \simeq\id_{I(f_{\ast})}$, or equivalently, that there exists a homotopy in $I(f_{\ast})$ that takes $(x, \rho,\Gamma)$ to $(x, \rho,\Gamma')$, where $\Gamma'$ is as defined above. It is clear that there exists a homotopy $\map{H_{u}}{I\times I}[I\times I]$, where $u\in I$, for which $H_{0}$ is the identity, $H_{1}(s,t)\in I \times \brak{0}$, $H_{u}(s,0)=(s,0)$, and $H_{u}(A)\subset A$ for all $u,s,t\in I$, where $A=I\times \brak{0} \cup\brak{0} \times  I$. For example, we may take:
\begin{equation*}
H_{u}(s,t)=\begin{cases}
((1-2u)s+2u \max(s,t),t) & \text{if $0\leq u\leq \frac{1}{2}$}\\
(\max(s,t), 2(1-u)t) & \text{if $\frac{1}{2} < u\leq 1$.}
\end{cases}
\end{equation*}
Let $u,s,t\in I$. Since $H_{\frac{1}{2}}(s,t)=(\max(s,t),t)$ and $H_{\frac{1}{2}}(0,t)=(t,t)$, we see that $H_{u}$ is continuous. For $u, s, t \in I$, let $\Gamma_u(s,t) =\Gamma (H_u(s,t))$. Then $\Gamma_{0}=\Gamma$, $\Gamma(t,0)= g(\rho(t))$, and $\Gamma_{1}(s,t)=\Gamma(\max(s,t),0)= g(\rho(\max(s,t)))=\Gamma'(s,t)$, so $\Gamma_{1}=\Gamma'$. It remains to show that $(x,\rho,\Gamma_{u})\in I(f_{\ast})$ for all $u\in I$. This is the case, since $f(x)=\rho(0)$, $\rho(1)=\ast_{Y}$, $\Gamma_{u}(s,0)=\Gamma(H_{u}(s,0))=\Gamma(s,0)=g(\rho(s))$, $\Gamma_{u}(s,1)=\Gamma(H_{u}(s,1))= \ast_{Z}$ and $\Gamma_{u}(1,t)=\Gamma(H_{u}(1,t))=\ast_{Z}$ using the properties of $H_{u}$ given above.
It follows that $\phi\circ \psi \simeq\id_{I(f_{\ast})}$, which proves the result.
\end{proof}

\begin{prop}\label{prop:homofibre}
Let $\map{p}{E}[B]$ be a fibration, and let $E_{0}$ be a subspace of $E$ such that the restriction $\map{p_{0}=p\left\lvert_{E_{0}}\right.}{E_{0}}[B]$ is also a fibration. Let $F$ (resp.\ $F_{0}$) denote the fibre of $p$ (resp.\ $p_{0}$) over a base point in $B$, let $\map{\iota}{E_{0}}[E]$, $\map{\iota_{0}}{F_{0}}[F]$, $\map{j}{F}[E]$, and $\map{j_{0}}{F_{0}}[E_{0}]$ denote the respective inclusions, where $\iota_{0}=\iota\left\lvert_{F_{0}}\right.$, and let $\map{\alpha'}{I(\iota_{0})}[I(\iota)]$ be the map defined by $\alpha'(x,\gamma)=(j_{0}(x), j\circ \gamma)$ for all $(x,\gamma)\in I(\iota_{0})$. Then $\alpha'$ is a homotopy equivalence between $I(\iota_{0})$ and $I(\iota)$.
\end{prop}

\begin{proof}
First note that we have the following commutative diagram:
\begin{equation}\label{eq:commdiaghomfib}
\begin{tikzcd}[cramped]
I(\iota_{0}) \arrow[dotted]{d}{\alpha'} \arrow{r}
& F_{0} \arrow{r}{\iota_{0}} \arrow[d, "j_{0}"]
& F \arrow[d, "j"]\\
I(\iota) \arrow{r}
& E_{0} \arrow{r}{\iota} \arrow{d}{p_{0}}
& E \arrow{d}{p}\\
& B \arrow[equal]{r}  & B.
\end{tikzcd}
\end{equation}
Let $x_{0} \in F_{0}$ be a basepoint that we propagate to $F,E_{0},E$ and $B$ using this diagram. In particular, let $b_{0}=p_{0}\circ j_{0}(x_{0})$ be the basepoint of $B$. The map $\map{\alpha'}{I(\iota_{0})}[I(\iota)]$ defined in the statement is well defined, since if $(x,\gamma)\in I(\iota_{0})$ then $\iota_{0}(x)=\gamma_{0}$ and $\gamma(1)=\iota_{0}(x_{0})$, so $\alpha'(x,\gamma)=(j_{0}(x), j\circ \gamma)\in E_{0}\times E^{I}$, and we have $j\circ \gamma(0)=j\circ \iota_{0}(x)=\iota(j_{0}(x))$, and $j\circ \gamma(1)=j\circ \iota_{0}(x_{0})=\iota(j_{0}(x_{0}))$, which is the basepoint of $E$. 

Taking $X=E_{0}$, $Y=E$, $Z=B$, $f=\iota$ and $g=p$ in the statement and proof of \relem{crabb1}, and noting that $p_{0}=p\circ \iota$, we obtain a map $\map{\iota_{\ast}}{I(p_{0})}[I(p)]$, and the map $\map{\psi}{I(\iota_{\ast})}[I(\iota)]$ defined in the first line of the proof of \relem{crabb1} is a homotopy equivalence. We will exhibit a homotopy equivalence $\map{h}{I(\iota_{0})}[I(\iota_{\ast})]$ for which $\alpha'=\psi\circ h$, which will prove the result. Since $p$ and $p_{0}$ are fibrations, the maps $\map{g}{F}[I(p)]$ and $\map{g_{0}}{F_{0}}[I(p_{0})]$ defined by $g(x)=(j(x),c_{b_{0}})$ for all $x\in F$ and $g_{0}(y)=(j_{0}(y),c_{b_{0}})$ for all $y\in F_{0}$ are homotopy equivalences~\cite[Proposition~3.5.10 and Remark~3.5.11]{A}. Consider the following diagram:
\begin{equation*}
\begin{tikzcd}[cramped]
I(\iota_{0}) \arrow[dotted]{d}{h} \arrow{r} & F_{0} \arrow{r}{\iota_{0}} \arrow{d}{g_{0}} & F \arrow{d}{g}\\
I(\iota_{\ast}) \arrow{r} & I(p_{0}) \arrow{r}{\iota_{\ast}} & I(p).
\end{tikzcd}
\end{equation*}
Using the commutativity of~\reqref{commdiaghomfib}, it is straightforward to check that the right-hand square is commutative, and we thus obtain an induced map $\map{h}{I(\iota_{0})}[I(\iota_{\ast})]$ defined by $h(x,\gamma)=(g_{0}(x), g\circ \gamma)$ for all $(x,\gamma)\in I(\iota_{0})$ that is a homotopy equivalence because $g$ and $g_{0}$ are homotopy equivalences. If $\map{c'}{I}[B^{I}]$ is the path defined by $c'(t)=c_{b_{0}}$ for all $t\in I$, then using the definition of $\psi$, for all $(x,\gamma)\in I(\iota_{0})$, we have:
\begin{equation*}
\psi\circ h(x,\gamma)=\psi(g_{0}(x), g\circ \gamma)=\psi((j_{0}(x),c_{b_{0}}), (j\circ \gamma, c'))=(j_{0}(x), j\circ \gamma)=\alpha'(x,\gamma).
\end{equation*}
So $\alpha'=\psi\circ h$, and hence $\alpha'$ is a homotopy equivalence as required.
\end{proof}

We end this paper with two corollaries of \repr{homofibre}.

\begin{cor}\label{cor:alphanprime}
The map $\map{\alpha_{n}'}{I_{n}'}[I_{n}]$ defined in \resec{generalities} is a homotopy equivalence.
\end{cor}

\begin{proof}
It suffices to take the first two rows (resp.\ the fibrations $p_{0}$ and $p$) of~\reqref{commdiaghomfib} to be the bottom two rows (the fibrations $p_{n}$ and $\widetilde{p}_{n}$ of \resec{mapiota}) of~\reqref{bigcommdiagfib}, and to apply \repr{homofibre}.  
\end{proof}

\begin{cor}
For all $k\geq 1$, the inclusion $\map{j}{F}[E]$ and its restriction $\map{j\left\lvert_{{F_{0}}}\right.}{F_{0}}[E_{0}]$ induce an isomorphism $\map{j_{\#k}}{\pi_{k}(F,F_{0})}[\pi_{k}(E,E_{0})]$.
\end{cor}


\begin{proof}
Recall from~\cite[Definition~4.5.3]{A} and the paragraph that precedes it that the relative homotopy group $\pi_{k}(E,E_{0})$ (resp.\ $\pi_{k}(F,F_{0})$) is the homotopy group $\pi_{k-1}(I(\iota))$ (resp.\ $\pi_{k-1}(I(\iota_{0}))$) of the homotopy fibre of $\iota$ (resp.\ $\iota_{0}$) for all $k\geq 1$. So for all $k\geq 2$, the map of pairs $\map{j}{(F,F_{0})}[(E,E_{0})]$ induces a homomorphism $\map{j_{\#k}}{\pi_{k}(F,F_{0})}[\pi_{k}(E,E_{0})]$ that coincides with the isomorphism $\map{\alpha_{\#(k-1)}'}{\pi_{k-1}(I(\iota_{0}))}[\pi_{k-1}(I(\iota))]$ given by~\repr{homofibre} (observe that if $k=1$ then as in~\cite[Definition~4.5.3]{A}, we refer to the map $\map{j_{\#1}}{\pi_{1}(F,F_{0})}[\pi_{1}(E,E_{0})]$ as a homomorphism, and we know that it is also a bijection). 
\end{proof}

\end{document}